\pdfoutput=1

\documentclass{amsart}
\usepackage{amsaddr}

\usepackage{amsmath}

\usepackage{times}
\usepackage{bm}
\usepackage{mathdots}
\usepackage[utf8]{inputenc} 
\usepackage{csquotes}

\usepackage{amsmath, amssymb,bm, cases, mathtools, thmtools}
\usepackage{verbatim}
\usepackage{graphicx}\graphicspath{{figures/}}
\usepackage{multicol}
\usepackage{tabularx}
\usepackage[usenames,dvipsnames]{xcolor}
\usepackage{url}
\usepackage[normalem]{ulem}
\usepackage{xstring}
\usepackage{enumitem}

\newcommand{\makebib}{%
\bibliography{biblio}}

\usepackage{datetime}

\DeclareMathAlphabet\EuRoman{U}{eur}{m}{n}
\SetMathAlphabet\EuRoman{bold}{U}{eur}{b}{n}

\definecolor{scolor}{rgb}{0,0,0}
\usepackage[colorlinks,citecolor=scolor,urlcolor=scolor,linkcolor=scolor,hypertexnames=false]{hyperref}

\usepackage[capitalize]{cleveref}

\crefname{enumerate}{Part}{Parts}
\crefname{assumption}{Assumption}{Assumptions}
\crefname{claim}{Claim}{Claims}

\declaretheorem[style=plain,name=Theorem]{theorem}
\declaretheorem[style=plain,name=Lemma]{lemma}
\declaretheorem[style=plain,name=Proposition]{proposition}
\declaretheorem[style=plain,name=Corollary]{corollary}

\declaretheorem[style=definition,name=Definition]{definition}
\declaretheorem[style=definition,name=Assumption]{assumption}
\declaretheorem[style=definition,name=Example]{example}
\declaretheorem[numbered=no,style=remark,name=Remark]{remark}

\declaretheorem[style=definition,name=Condition]{conditionINNER}
\newenvironment{mycondition}[1]
 {\renewcommand{\theconditionINNER}{(#1)}\conditionINNER}
 {\endconditionINNER}
\crefname{conditionINNER}{Condition}{Conditions}
\crefformat{conditionINNER}{#1}
\crefmultiformat{conditionINNER}
  {#2#1#3}
  { and~#2#1#3}
  {, #2#1#3}
  { and~#2#1#3}
\Crefformat{conditionINNER}{#2#1#3}
\Crefmultiformat{conditionINNER}
  {#2#1#3}
  { and~#2#1#3}
  {, #2#1#3}
  { and~#2#1#3}

\usepackage{amssymb}
\usepackage{leftidx}

\usepackage{natbib}

\usepackage{tikz}
\usepackage{siunitx}

\usepackage{leftidx}

\usepackage{xifthen}
\usepackage{relsize}

\usepackage[plain,noend]{algorithm2e}

\makeatletter
\renewcommand{\algocf@captiontext}[2]{#1\algocf@typo. \AlCapFnt{}#2} %
\def\@algocf@capt@plain{top}
\renewcommand{\algocf@makecaption}[2]{%
  \addtolength{\hsize}{\algomargin}%
  \sbox\@tempboxa{\algocf@captiontext{#1}{#2}}%
  \ifdim\wd\@tempboxa >\hsize%
    \hskip .5\algomargin%
    \parbox[t]{\hsize}{\algocf@captiontext{#1}{#2}}%
  \else%
    \global\@minipagefalse%
    \hbox to\hsize{\box\@tempboxa}%
  \fi%
  \addtolength{\hsize}{-\algomargin}%
}
\makeatother

\def\[#1\]{\begin{align}#1\end{align}}
\def\*[#1\]{\begin{align*}#1\end{align*}}

\def\clap#1{\hbox to 0pt{\hss#1\hss}}

\newcommand{\defas}{\vcentcolon=}  %

\newcommand{\argdot}{\cdot}

\newcommand{\st}{\,:\,}

\newcommand{\Reals}{\mathbb{R}}
\newcommand{\Nats}{\mathbb{N}}

\newcommand{\NNReals}{\Reals_{\ge 0}}
\newcommand{\PosReals}{\Reals_{> 0}}

\newcommand{\dee}{\mathrm{d}}

\DeclareMathOperator*{\essinf}{ess\,inf}

\DeclareMathOperator*{\newlim}{\mathrm{lim}\vphantom{\mathrm{infsup}}}
\DeclareMathOperator*{\newmin}{\mathrm{min}\vphantom{\mathrm{infsup}}}
\DeclareMathOperator*{\newmax}{\mathrm{max}\vphantom{\mathrm{infsup}}}
\DeclareMathOperator*{\newinf}{\mathrm{inf}\vphantom{\mathrm{infsup}}}
\DeclareMathOperator*{\newsup}{\mathrm{sup}\vphantom{\mathrm{infsup}}}
\renewcommand{\lim}{\newlim}
\renewcommand{\min}{\newmin}
\renewcommand{\max}{\newmax}
\renewcommand{\inf}{\newinf}
\renewcommand{\sup}{\newsup}

\newcommand{\cH}{\mathcal H}
\newcommand{\cR}{\mathcal R}
\newcommand{\cP}{\mathcal P}
\newcommand{\fS}{\PowerSet_{\smash{f}}(\Theta)}

\def\EE{\mathbb{E}}

\newcommand{\defn}[1]{\textit{#1}}

\newcommand{\floor}[1]{\lfloor #1 \rfloor}

\newcommand{\cF}{\mathcal F}

\newcommand{\set}[2][]{#1\{#2 #1\}}

\newcommand{\parens}[2][]{#1(#2 #1)}
\newcommand{\abs}[2][]{#1\lvert#2 #1\rvert}
\newcommand{\norm}[2][]{#1\lVert#2 #1\rVert}

\newcommand{\tuple}[2][]{#1 \langle #2 #1 \rangle}

\newcommand{\supnorm}[2][]{\norm[#1]{#2}_{\infty}}
\newcommand{\lipnorm}[2][]{\norm[#1]{#2}_{\mathrm{Lip}}}

\newcommand{\Linfnorm}[3][]{\norm[#1]{#3}_{L^\infty\optparen{#2}}}
\newcommand{\dsup}{d_{\sup}}

\newcommand{\BorelSets}[1]{\mathcal{B}_{#1}}
\newcommand{\NSE}[1]{{^{*}#1}}
\newcommand{\ST}{\mathsf{st}}
\newcommand{\SP}[1]{{^{\circ}#1}}

\newcommand{\PowerSet}{\mathcal{P}}
\newcommand{\HReals}{\NSE{\Reals}}

\newcommand{\Model}{P}

\newcommand{\NS}[1]{\mathrm{NS}(#1)}

\newcommand{\cD}{\mathcal{D}}

\newcommand{\cA}{\mathcal{A}}

\newcommand{\cC}{\mathcal{C}}

\newcommand{\IntFuncs}[2]{\mathrm{I}(#2^{#1})}

\newtheorem{open problem}{Open Problem}

\newcommand{\Loeb}[1]{\overline{#1}}

\newcommand{\LoebAlgebraX}[2]{\overline{#1}_{#2}}

\newcommand{\ProbMeasures}[1]{\mathcal{M}_1(#1)}

\newcommand{\interior}[1]{%
  {\kern0pt#1}^{\mathrm{o}}%
}

\newcommand{\refproof}[1]{See \cref{#1} for \IfSubStr{#1}{,}{proofs}{a proof}. }

\newcommand{\TTheta}{T_{\Theta}}

\newcommand{\PM}[1]{\mathcal{M}(#1)}
\newcommand{\PD}[1]{\mathcal{D}(#1)}

\newcommand{\FPM}[1]{\mathcal{M}_{\mathrm{fin}}(#1)}

\newcommand{\Lip}[2]{\mathcal{L}_{#2}(#1)}

\newcommand{\Main}{M}
\newcommand{\oMain}{\ccc{\Main_{o}}}

\newcommand{\ssetPM}[2]{\NSE{\!\!\!\mathcal{M}}_{#1}(#2)}

\newcommand{\Q}[1]{Q_{#1}}
\newcommand{\FS}{\mathcal{F}_{\mathrm{fin}}}
\newcommand{\AGD}[2]{\mathcal{A}_{#2,#1}}
\newcommand{\expect}{\mathbb{E}}

\newcommand{\pd}[1]{{#1}_p}
\newcommand{\closure}[1]{\overline{#1}}
\newcommand{\IVSpace}[2]{\mathrm{IV}({#1}^{#2})}

\newcommand{\boundary}{\partial}

\renewcommand{\:}{\colon}

\renewcommand{\defas}{\overset{\text{\smash{\tiny{def}}}}{=}}

\newcommand\optparen[1]{\ifthenelse{\isempty{#1}}{}{(#1)}}
\newcommand\optargs[3]{\ifthenelse{\isempty{#1}}{#2}{#3}}

\newcommand{\corr}[2]{A_{#2,#1}}

\newcommand{\mm}[1]{\textcolor{red}{#1}}
\renewcommand{\mm}[1]{#1}

\newcommand{\ccc}[1]{\textcolor{blue}{#1}}
\renewcommand{\ccc}[1]{#1}

\newcommand{\candidgen}[9]{\mm{\psi^{\smash{\prime}}_{#1#2#3#4#5#6}(#9,#8\,;#7)}} %
\newcommand{\candid}[2]{\candidgen{B;}{\alpha}{\beta}{\gamma}{\delta}{\eta}{#1}{#2}{\cdot}}
\newcommand{\papf}[3]{\candidgen{}{\alpha}{\beta}{\gamma}{\delta}{\eta}{#1}{#3}{#2}}

\newcommand{\finalsetgen}[9]{\mm{\psi_{#1#2#3#4#5#6}\optargs{#7#8}{}{(#9,#8 \,;#7)}}}
\newcommand{\finalset}[2]{\finalsetgen{B;}{\alpha}{\beta}{\gamma}{\delta}{\eta}{#1}{#2}{\cdot}}
\newcommand{\gfinalset}[3][\cdot]{\finalsetgen{}{\alpha}{\beta}{\gamma}{\delta}{\eta}{#2}{#3}{#1}}

\newcommand{\sfrset}[3]{#2 \times #3 \times #1}

\newcommand{\bapf}[5][]{\ccc{\psi^{#1}\optargs{#3#4#5}{}{(#4,#5\,;#3)}}} %

\newcommand{\prset}[1]{\varphi\optargs{#1}{}{(#1)}} %
\newcommand{\rset}[2]{\varphi\optargs{#1#2}{}{(#1,#2)}} %
\newcommand{\frset}[4][]{\ccc{\varphi^{#1}\optargs{#2#3#4}{}{(#3,#4\,;#2)}}} %
\newcommand{\frsetalt}[3]{\ccc{\varphi'\optargs{#1#2#3}{}{(#2,#3\,;#1)}}} %
\newcommand{\credset}[4]{\psi_{#1;#2\ifthenelse{\equal{#3}{}}{}{,#3}}\ifthenelse{\equal{#4}{}}{}{(#4)}}
\newcommand{\tBall}{I}
\newcommand{\tHPD}{H}
\newcommand{\credball}[3]{\credset{\tBall}{#1}{#2}{#3}}
\newcommand{\hpdreg}[3]{\credset{\tHPD}{#1}{#2}{#3}}

\newcommand{\cbL}[1]{L_{\tBall;\alpha}(#1)}
\newcommand{\cbLR}[1]{L_{\tBall;\alpha,\beta}(#1)}
\newcommand{\cbLv}[2]{L_{\tBall;#2}(#1)}
\newcommand{\hpdL}[1]{L_{\tHPD;\alpha}(#1)}
\newcommand{\hpdLR}[1]{L_{\tHPD;\alpha,\beta}(#1)}

\newcommand{\ffB}[2]{f_{\tBall;#1}\ifthenelse{\equal{#2}{}}{}{(#2)}}
\newcommand{\ffHPD}[2]{f_{\tHPD;#1}\ifthenelse{\equal{#2}{}}{}{(#2)}}

\newcommand{\CRegions}[3]{(#1_{#2})_{#2 \in #3}}
\newcommand{\CRegionsv}[4]{(#1_{#2})_{#3 \in #4}}

\newcommand{\sTheta}{\NSE{\Theta}}
\newcommand{\Priors}{\PM{\Theta}}
\newcommand{\AltPriors}{\PM{\Theta'}}
\newcommand{\sPriors}{\NSE{\!\PM{\Theta}}}
\newcommand{\charfunc}[1]{\chi_{#1}}
\newcommand{\SModel}{\tuple{X,\Theta,P}}
\newcommand{\SModelfullalt}{\tuple{X,\Theta',P'}}
\newcommand{\SModelalt}[1]{\tuple{X,#1,\Restrict{P}{#1}}}

\newcommand{\sSModel}{\tuple{\NSE{X},\sTheta,\NSE{P}}}
\newcommand{\sSModelalt}[1]{\tuple{\NSE{X},#1,\Restrict{\NSE{P}}{#1}}}

\newcommand{\Post}[2]{P^{\vphantom{x}#1}_{#2}\!}
\newcommand{\Postalt}[2]{P^{'\vphantom{x}#1}_{#2}\!}

\newcommand{\marg}[1]{#1{P}}
\newcommand{\margalt}[1]{#1{P'}}
\newcommand{\rbcs}[4]{\psi_{#4;#1,#2,#3}}
\newcommand{\rcb}[3]{\rbcs{#1}{#2}{#3}{\tBall}}
\newcommand{\rhpd}[3]{\rbcs{#1}{#2}{#3}{\tHPD}}

\newcommand{\bg}{\ccc{g}}
\newcommand{\pb}[1]{\ccc{h_{#1}}}

\newcommand{\fTheta}{\ccc{\Theta'}}
\newcommand{\fPriors}{\ProbMeasures{\fTheta}}

\newcommand{\cd}{\ccc{q}}

\newcommand{\pt}[2]{\ccc{{#1}_{#2}}}

\newcommand{\extn}[2]{\ccc{#1^{#2}}}

\newcommand{\Restrict}[2]{#1|_{#2}}

\newcommand{\mpf}{\ccc{acceptance probability function}}
\newcommand{\hr}[2]{h_r(#1,#2)}

\newcommand{\dX}{d_{X}}
\newcommand{\dT}{d_{\Theta}}

\newcommand{\CFuncs}{\mathcal{C}}

\newcommand{\mC}{\ccc{C'}}

\newcommand{\rtau}{\bar\tau}
\newcommand{\PPriors}{\mathcal M_{\gamma}(\Theta)}

\newcommand\hl{\bgroup\markoverwith
  {\textcolor{yellow}{\rule[-.5ex]{.1pt}{2.5ex}}}\ULon}

\makeatletter
\providecommand*{\toclevel@definition}{0}
\providecommand*{\toclevel@theorem}{0}
\providecommand*{\toclevel@lemma}{0}
\makeatother

\definecolor{WowColor}{rgb}{.75,0,.75}
\definecolor{SubtleColor}{rgb}{0.9,0,0}

\newcounter{margincounter}

\newcounter{latercounter}

\begin{document}
\bibliographystyle{abbrvnat}

\newcommand{\affil}[1]{\address{#1}}
\title%
{
Existence of matching priors on compact spaces
}

\author{HAOSUI DUANMU}
\affil{Institute for Advanced Study in Mathematics, Harbin Institute of Technology\\
Department of Economics, University of California, Berkeley}

\author{DANIEL M. ROY}
\affil{Department of Statistical Sciences, University of Toronto; Vector Institute}

\author{AARON SMITH}
\affil{Department of Mathematics and Statistics, University of Ottawa}

\begin{abstract}
A matching prior at level $1-\alpha$ is a prior such that an associated $1-\alpha$ credible region is also a $1-\alpha$ confidence set.
We study the existence of matching priors for general families of credible regions.  Our main result gives topological conditions under which matching priors for specific families of credible regions exist. Informally, we prove that, on compact parameter spaces, a matching prior exists if the so-called rejection-probability function is jointly continuous when we adopt the Wasserstein metric on priors. In light of this general result, we observe that typical families of credible regions, such as credible balls, highest-posterior density regions, quantiles, etc., fail to meet this topological condition. We show how to design approximate posterior credible balls and highest-posterior-density regions that meet these topological conditions, yielding matching priors. Finally, we evaluate a numerical scheme for computing approximately matching priors based on discretization and iteration. The proof of our main theorem uses tools from nonstandard analysis and establishes new results about the nonstandard extension of the Wasserstein metric that may be of independent interest.
\end{abstract}

\maketitle

           \thispagestyle{empty}

\let\oldgamma\gamma
\newcommand{\abver}{h}
\renewcommand{\epsilon}{\varepsilon}
\newcommand{\scare}[1]{#1}

\renewcommand{\defas}{=}
\section{Introduction}

In both Bayesian and frequentist statistics, set estimators are used to summarize inferences about model parameters. 
Although both the usual frequentist set estimators (called confidence sets) and the usual Bayesian set estimators (called credible regions) satisfy many desirable properties,
in general they do not enjoy the \emph{same} desirable properties.
Hence, there is a long-standing interest in constructing sets that have \emph{both} types of desirable properties.
The recent book by \cite{datta2012probability} gives a detailed history.

In the literature, if a prior is associated with a credible region that is also a confidence set of the same level,
then the prior is called a \emph{matching prior.}
Besides constructing set estimators with both good Bayesian and frequentist properties,
there are other, more elaborate,
justifications for studying matching priors.
There is substantial interest in finding so-called default or objective priors for Bayesian inference,
and matching priors have been suggested as good candidates in this search.
See, e.g., \citet{ghosh2011objective} for a survey on objective priors, including an extensive discussion of matching priors. 
Even for those with no interest in Bayesian statistics,  matching priors are useful in constructing confidence sets with certain desirable properties \citep{stein1985coverage,tibshirani1989noninformative,daita1995priors}.
Although we do not attempt to survey the literature on objective and matching priors, 
we note that matching priors are closely related to so-called bet-proof confidence sets, as the bet-proof condition is  a relaxation of the matching-prior condition \citep{buehler1959some}. 
Bet-proof confidence sets for invariant problems were recently studied by \citet{mullereconometrica}.

Much of the work on existence and computation of matching priors is summarized in the survey by \citet{datta2005probability}.
For the purposes of this article, we are most interested in distinguishing between three main lines of research on matching priors.
The first concerns the construction of explicit matching priors in specific statistical models for specific families of credible regions---especially posterior quantiles, highest posterior density regions, and their posterior predictive counterparts.
This explicit work generally focuses on problems that have a great deal of symmetry.
For certain families of credible regions, under strong enough conditions, there is a unique matching prior that is moreover independent of the level \citep{Severini}.
The second and largest line of research is concerned with priors that are \emph{asymptotically} matching at some rate $c > 0$.
These are priors for which the associated credible region at level $1-\alpha$ is also a confidence set at level $1-\alpha + O(n^{-c})$,
where $n$ is the number of i.i.d data points and one attempts to show that $c$ is as large as possible; $c = {1}/{2}$ is often immediate by, e.g., the central limit theorem. This subject was initiated by \citet{welch1963formulae}, and \citet{datta2012probability} summarize modern work on the subject.

The third line of research concerns the existence of exact matching priors for general families of credible regions.
Negative results abound: 
 \citet[][Ex.~1 and Thm.~1]{datta2000bayesian} gives a concrete example where matching priors for certain families of credible regions do not exist and %
\citet{sweeting2008predictive} gives necessary conditions for the existence of an exact matching prior on $\Reals^{d}$.
Conversely, recent work by \citet{muller2016coverage} shows that matching priors \emph{do} exist for finite parameter spaces under very weak conditions.
The present paper is most closely related to this third line of research, and in particular the approach of \citeauthor{muller2016coverage}.
Its main contribution is to partially fill the gap between the results of \citeauthor{datta2000bayesian}\ and \citeauthor{muller2016coverage},
showing that large classes of matching priors exist for \emph{compact} parameter spaces. %

We use tools in mathematical logic and nonstandard analysis \citep{AR65} to establish our main result.
These tools have been used recently to establish  general complete class theorems connecting Bayesian and frequentist statistics \citep{DR2017}.
The nonstandard analytic approach provides techniques for extending results on finite discrete spaces to continuous ones,
and it seems likely that there are many more potential applications within statistics.

We mention one quirk of terminology that can be misleading. Although the term ``matching prior" emphasizes the role of the prior and the literature emphasizes the role of the statistical model, questions of existence, uniqueness, and (efficient) computability of matching priors are in general answered with respect to \emph{both} a fixed statistical model \emph{and} a fixed family of credible regions.
More carefully, a family of credible regions
is in fact a map from the collection of prior measures and points in the sample space to the collection of credible regions. Whether matching priors exist depends on both the model \emph{and} the family of credible regions under consideration.

This dependance on the particular choice of family of credible regions turns out to be significant. A central contribution of this paper is the construction of  families of credible regions that \emph{both} yield matching priors \emph{and} have many of the same desirable properties as typical families of credible regions, \emph{even when} the typical families do not yield matching priors. 

We look closely at two typical families: credible balls and highest posterior density regions.
These families are not guaranteed, by our results, to have matching priors, due to severe discontinuity. 
Instead, our results produce matching priors for explictly constructed families of credible regions that closely approximate them.
We consider the Bernoulli model as a concrete example. 

\section{Methodological questions}

Our theoretical inquiry is motivated by the following informal methodological questions: \emph{(i) should one expect matching priors to exist and (ii) where should one look for them?}
It may not be possible to ever provide definitive answers to these question, but theoretical research surely guides the expectations of practicing statisticians.
From our perspective, nearly all previous work suggests that the answers are (i) \emph{matching priors do not exist, except in a few very special cases, often when there is considerable symmetry} and (ii) \emph{when matching priors exist, they are easy to find}. As an example that illustrates this dominant point of view, \citet{fraser2011bayes} simply asserts ``the Bayes and the confidence results [are] different when the model [is] not location." 
It was well-known that exceptions to this assertion are \emph{possible}, but we are not aware of anyone suggesting that exceptions are \emph{common} in any statistically interesting settings.

The results by \citet{muller2016coverage} suggest a very different viewpoint, showing that matching priors are in fact ubiquitous for \emph{discrete} models, despite the lack of symmetry.
A reader of the work of \citet{muller2016coverage} and \citet{fraser2011bayes} might come to a heuristic along the following lines: matching priors are ubiquitous for \emph{discrete} models and very rare for \emph{continuous} models.
We suspect that this heuristic is largely correct \emph{when one restricts one's attention to highest-posterior density regions.}

The results in this paper suggest to us a different modification to the heuristic in \cite{fraser2011bayes}: 
one can find matching priors \emph{easily} if one has a highly symmetric or discrete model; 
if not, one can often find them anyway with some extra work. 
The main methodological contribution of this paper is demonstrating that, in some cases, the extra work consists of finding a suitably continuous families of credible regions. 
We, however, do not know of other generic strategies for finding families of credible regions that yield matching priors, and view this as the main methodological question that is suggested and left open by the results in this paper.

\section{Confidence sets and credible regions}\label{secsummary}

\subsection{Models and preliminaries}
Fix a \defn{statistical model}, i.e., a triple $\SModel$,
where $X$ and $\Theta$ are a sample space and parameter space, respectively,
and $P = (P_\theta)_{\theta \in \Theta}$ specifies a family of probability distributions on $X$ indexed by elements of $\Theta$. 
We assume $(X,\dX)$ and $(\Theta,\dT)$ are metric spaces, write $\BorelSets X$ and $\BorelSets \Theta$ to denote their respective Borel $\sigma$-algebras, and
let $\PM{X}$ and $\Priors$ denote the space of (Borel) probability measures on each space. Note that these spaces of probability measures are themselves measurable spaces under the Borel $\sigma$-algebras generated by their weak topologies. Formally, we assume the model $(P_\theta)_{\theta \in \Theta}$ is a probability kernel $P : \Theta \to \PM{X}$, i.e., a measurable map from $\Theta$ to the space of probability measures on $X$. 
Let $P_\theta$ denote $P(\theta)$.

We assume that the model is \defn{dominated},
i.e., there exists a $\sigma$-finite measure $\nu$ on $(X,\BorelSets X)$
and a product measurable function $\cd \: \Theta \times X \to \NNReals$
such that, for all $\theta \in \Theta$, $P_{\theta}$ admits a density $\cd(\theta,\argdot)$ with respect to $\nu$.
Let $\cd_{\theta}$ denote $\cd(\theta,\cdot)$. We sometimes refer to $\cd$ as the conditional density.

Our focus is on summaries of Bayesian inference and so we begin by introducing notation for posterior distributions.
For a given prior $\pi \in \Priors$,
the \defn{marginal distribution}
$\marg{\pi} \in \PM{X}$
is the distribution on the sample space defined by
\[
(\marg{\pi})(A)
= \int_{\Theta} P_{\theta}(A) \pi(\dee \theta), \quad A\in \BorelSets X.
\]
Write
$\Post{x}{\pi}$ to denote the posterior distribution given an observation $x \in X$,
which necessarily satisfies, for $\marg{\pi}$-almost all $x$ and all $B \in \BorelSets \Theta$,
\[\label{stEOSUThs}
\Post{x}{\pi}(B)
= \frac { \int_B \cd(\theta,x) \pi(\dee \theta)} { \int_{\Theta} \cd(\theta,x) \pi(\dee \theta) }.
\]
The map $x \mapsto \Post{x}{\pi}$ is measurable.
As is well known, conditional distributions are only defined up to a $\marg{\pi}$-measure one set.
If $\marg{\pi}$ dominates $\nu$ then all subsequent definitions are invariant
to which version of the posterior distribution is used.
In general, different versions can affect frequentist properties, and so we assume that we have fixed some version of the conditional distribution for all $\pi$ and $x$. 
In general, continuous versions, if they exist, are unique on the support.
We make repeated use of the fact that, in the setting outlined here, the posterior is dominated by the prior \citep[Thm.~1.31]{schervish1995theory}.

\subsection{Families of credible regions and matching priors}

One of the central structures in this paper is that of a credible region.
For a fixed \emph{level} $1-\alpha \in [0,1]$ and prior $\pi \in \Priors$,
a \emph{$(1-\alpha)$ credible region (with respect to $\pi$)} 
associates, to $\marg{\pi}$-almost all possible observations $x \in X$, a Borel subset $\tau(x) \in \BorelSets{\Theta}$ 
whose posterior probability is $1-\alpha$. (We are emphasizing the view of $\tau$ as a set estimator, mapping $X$ to $\BorelSets{\Theta}$, to allow us to later consider frequentist properties.)
Before introducing a more formal definition, we note that credible regions may not exist for a particular level $1-\alpha$ when the posterior has atoms. 
A typical fix, which is also used by \citet{muller2016coverage}, 
is to introduce \emph{randomization} and to ask for the correct level \textit{in expectation}. 

\newcommand{\credibility}[2]{\mathbb{E}[#1(#2)]}
More carefully, 
for $S \subseteq \Theta$, let $\charfunc{S} : \Theta \to \set{0,1}$ denote the characteristic function of $S$.
Viewing $[0,1]$ as a probability space under Lebesgue measure, 
a \defn{random set} is a map $u \mapsto T(u) : [0,1] \to \BorelSets{\Theta}$ such that the map $(\theta,u) \mapsto \charfunc{T(u)}(\theta)$ is product measurable. 
It is useful to introduce nomenclature for the expected value of a random set's characteristic function and measure:
The \defn{acceptance probability function} of a random set $T$ is the measurable map $\psi : \Theta \to [0,1]$ given by 
$\psi(\theta) = \int_{[0,1]} \charfunc{T(u)}(\theta)\, \dee u$,
and the \defn{$\mu$-credibility of $T$}, for $\mu \in \Priors$, is the expectation
$\int_{\Theta} \psi(\theta)\, \mu(\dee \theta)$.
Finally, a randomized set estimator is a map $\tau : X \times [0,1] \to \BorelSets{\Theta}$
  such that $(\theta,x,u) \mapsto \charfunc{\tau(x,u)}(\theta) : \Theta \times X \times [0,1] \to \set{0,1}$ is product measurable. In particular, $\tau(x,\cdot)$ is a random set for every $x \in X$.

\begin{definition}[Credible regions, families thereof]\label{credfamdefn}
  Let $\alpha \in (0,1)$ and $\pi \in \Priors$.
  A \defn{$(1-\alpha)$ credible region} (with respect to $\pi$, in $\SModel$) 
  is a randomized set estimator $\tau$ such that, for $\marg{\pi}$-almost all $x \in X$,
 $\tau(x,\cdot)$ has $1-\alpha$ $\Post{x}{\pi}$-credibility.
  A \defn{family of $1-\alpha$ credible regions (in $\SModel$)} 
  is a collection $\CRegions{\tau}{\pi}{F}$, indexed by some fixed subset $F \subseteq \Priors$ of priors,
  such that, for each $\pi \in F$, $\tau_\pi$ is a $1-\alpha$ credible region with respect to $\pi$.
\end{definition}

Many common families of credible regions---e.g., credible balls and highest-posterior density regions---arise from fixed maps $\rho : \Priors \to \BorelSets{\Theta}$ taking (posterior) distributions to sets (or random sets) with the desired posterior credibility. %

The measurability requirements in our definition allow us to simultaneously consider the frequentist coverage of a credible region.
The acceptance probability function and rejection probability function for a set estimator $\tau$ are the measurable maps given by, respectively
  \[
  \psi(\theta,x) = \int_{[0,1]} \charfunc{\tau(x,u)}(\theta) \, \dee u
  \text{\quad and \quad}
  \rset{\theta}{x} = 1- \psi(\theta,x), \qquad  \theta \in \Theta, \ x \in X.
  \]
The set estimator $\tau$ is a \defn{$1-\alpha$ (randomized) confidence set} if and only if,
  for all $\theta \in \Theta$,
  $
  \int_{[0,1]} \rset{\theta}{x}\,  P_{\theta}(\dee x) \le \alpha.
  $
  We sometimes adopt the convention of saying that $\tau$ \defn{has coverage} $1-\alpha$.
  For both credible regions and confidence sets, we also refer to $1-\alpha$ as the \defn{level}. 

A matching prior delivers both Bayesian and frequentist properties:
\begin{definition}%
  \label{defnmatching}
  A prior $\pi_0$ is a matching prior (at level $1-\alpha$) for a randomized set estimator $\tau$ (in $\SModel$) if
  $\tau$ is both a $1-\alpha$ credible region with respect to $\pi_0$ and a $1-\alpha$ confidence set.
  A prior $\pi_0 \in F$ is a matching prior for a family $\CRegions{\tau}{\pi}{F}$ if it is matching for $\tau_{\pi_0}$.
\end{definition}

By allowing randomization, one can construct trivial examples of set estimators that are both  $1-\alpha$ confidence sets and $1-\alpha$ credible regions, irrespective of the prior. One example is the set estimator that is the empty set with probability $\alpha$ and is the set $\Theta$ otherwise. 
Such trivial families are of little interest because they lack other desirable properties, relating, e.g., to their size, shape, and location.
Rather than being focused on any one type of credible region, 
our main result provides sufficient conditions on families of credible regions
for a matching prior to exist. We can then choose to apply these theorems to non-trivial families of credible regions.

We pause to present a concrete and simple example, based on a Bernoulli observation. 

\begin{example}[Bernoulli model] \label{ExBernoulliMatching}
Consider a sample space $X = \{0,1\}$, parameter space $\Theta = [0,1]$, and model with conditional density $\cd(\theta,x) = \theta^{x} (1 - \theta)^{1-x}$. 
The prior $\pi_{\alpha} = \mathrm{Unif}(\{\alpha,1-\alpha\})$ 
is a matching prior at level $1-\alpha$  for the 
(nonrandomized) set estimator $\tau$ given by $\tau(0,\cdot) = [0,1-\alpha)$ and $\tau(1,\cdot)=(\alpha,1]$.
See \cref{ThmNewBernEx1} (\cref{SecAppExamplesBernoulli}) for a rigorous proof.
See \cref{SecAppExamplesBinomial,SecAppExamplesGaussian}
for, resp., a Binomial and Gaussian example (with an improper prior).
\end{example}

The rejection (or acceptance) probability function of a set estimator suffices to determine whether it is a credible region or confidence set.
In general, every rejection probability function $\prset{}{}$ may correspond to several credible regions. 
One such credible region $\tau$ 
is that given by $\charfunc{\tau(x,u)}(\theta) = 0$ if and only if $u \le \rset{\theta}{x}$.
Rejection (acceptance) probability functions for families will be a primary focus:

\begin{definition}[Rejection probability function]
  \label{rpfunc}
  A \defn{rejection probability function}, over some family $F \subseteq \Priors$, is a function
  $(\theta,x,\pi) \mapsto \frset{\pi}{\theta}{x} \: \sfrset{F}{\Theta}{X} \to [0,1]$ 
  such that,
  for every $\pi \in F$,
  the function $\frset{\pi}{\cdot}{\cdot} : \Theta \times X \to [0,1]$ is product measurable.
  The rejection probability function $\frset{}{}{}$
  \defn{for a family of credible regions $\tau = \CRegions{\tau}{\pi}{F}$}
  is one that satisfies $\frset{\pi}{\theta}{x} = 1 - \int_{[0,1]} \charfunc{\tau_{\pi}(x,u)}(\theta) \, \dee u$
  for all $\pi \in F$, $\theta \in \Theta$, and $x \in X$.
  We refer to $1-\frset{}{}{}$ as the acceptance probability function for the family.
\end{definition}

The rejection probability function contains the same information about coverage 
as the full family of credible regions. It also determines the ``support'' of the credible regions:
Recall that the support of a function $f: \Theta \to [0,1]$ to be the closure of the set $\{\theta \in \Theta: f(\theta) \neq 0\}$. 
\begin{definition}[Support of a random set]
The \defn{support of a random set} $T : [0,1] \to \BorelSets{\Theta}$  
is 
the support of the map
$\theta \mapsto \int_{\smash{[0,1]}} \charfunc{T(u)}(\theta) \, \dee u : \Theta \to [0,1]$.
For a credible region $\tau$, %
the \defn{support of $\tau$ at $x$} is defined to be the support of the random set $\tau(x,\cdot)$.
\end{definition}

\section{First generic result: Existence on compact spaces under continuity} \label{SecGenRes}

For a metric Borel measurable space $(\Omega, \BorelSets \Omega, d)$ and $\mu,\nu \in \PM{\Omega}$, recall that the \defn{Wasserstein distance between $\mu$ and $\nu$}  is
  $
  W_{1}(\mu,\nu)=\inf_{\tau \in \Pi(\mu,\nu)} \int d(x,y) \, \dee \tau(x,y),
  $
where $\Pi(\mu,\nu)$ is the collection of Borel measures on $\Omega \times \Omega$ with marginals $\mu$ and $\nu$.

We begin with an easy-to-state continuity condition, which we later relax: 

\begin{assumption}\label{assumptionjc}
  The rejection probability function $\frset{}{}{}$ for a family $\tau$ %
  is such that, for every $\theta\in \Theta$,
  the function $\frset{\cdot}{\theta}{\cdot} \: X\times\Priors\to [0,1]$ is jointly continuous,
  where $\Priors$ is equipped with the Wasserstein metric $W_1$.
\end{assumption}

This assumption leads to our first generic existence theorem:

\begin{theorem}\label{ThmMainExistenceResult}
  Let $\tau = \CRegions{\tau}{\pi}{\Priors}$ be a family of $1-\alpha$ credible regions.
  Suppose $\Theta$ is compact and $\tau$ satisfies \cref{assumptionjc}.
  Then there exists a matching prior for $\tau$.

\end{theorem}

\cref{assumptionjc} is not necessary. \cref{WeakThmMainExistenceResult} (\cref{sec:existenceproof}, supplementary material) gives the same conclusion under much weaker conditions. We highlight one difference that is critical for many applications: \cref{ThmMainExistenceResult} requires the map in \cref{assumptionjc} to be continuous \emph{everywhere}, while \cref{WeakThmMainExistenceResult} allows some discontinuities as long as they can be controlled.

\section{Applications} \label{SecApplHead}

\subsection{Overview}

As noted immediately after \cref{credfamdefn}, there are always matching priors for certain \emph{trivial} families of credible regions. 
In this section, we apply \cref{ThmMainExistenceResult} and its generalizations to prove the existence of matching priors for models and families of credible regions that are statistically interesting, \emph{even when} \cref{assumptionjc} fails to hold. Before giving details, we discuss the main obstacles to obtaining an interesting result using \cref{ThmMainExistenceResult}. 

Recall that \cref{assumptionjc} requires the map sending data and priors to a rejection probability function be continuous. It is natural to try to prove this by establishing the continuity of \emph{both} 
(i)
the map sending data and prior to a posterior, \emph{and}
(ii)
the map sending a posterior to a rejection probability function.
Unfortunately, \emph{both} maps are discontinuous for many popular models and families of credible regions: the first map has two discontinuities even for the Bernoulli model, while, as illustrated in \cref{secrelation}, the second map is typically discontinuous for both credible balls and highest-posterior density regions on continuous parameter spaces.

Fortunately, both obstacles can be overcome. In \cref{simplexist}, we show that there exists a family that is, in a technical sense, arbitrarily close to the usual family of credible balls \emph{and} for which the second map above is continuous. In \cref{SecAppBern}, we show that it is possible to ignore discontinuities in the first map that are in a certain sense isolated. 
See \cref{SecSimpleExamples} in the supplemental material for further applications, including analogous results for the highest-posterior density family of credible regions.

\subsection{Application to approximate credible balls}
\label{simplexist}

For simplicity, we assume the parameter space $\Theta$ is a compact subset of $\Reals^d$ for some $d$, equipped with the usual Euclidean metric denoted $\norm{\cdot}$. We also assume that $\Theta$ has positive Lebesgue measure. As in the previous section, all definitions are relative to a statistical model, but we elide this dependence. We begin by defining credible balls formally.

\begin{definition} [Credible balls]  \label{DefUsualInterval}
  Let $\alpha \in (0,1)$ and $\mu\in \Priors$, and assume $\mu$ has mean $M(\mu)$. 
  Let $B(r) = \set{\theta \in \Theta \st \norm{\theta - M(\mu)} \leq r }$ be the closed ball centered at $M(\mu)$ with radius $r$
  and let
  \[ \label{cbLdefn}
  \cbL{\mu} \defas \inf \set[\big]{ r > 0 \st \mu\parens[\big]{B(r)} \geq 1 - \alpha }.
  \]
  The \defn{$1-\alpha$ $\mu$-credible ball} is the subset $\smash{B(\cbL{\mu})}$ of $\Theta$ with characteristic function $\credball{\mu}{\alpha}{}$.
  A family $\CRegions{\tau}{\pi}{F}$ of $1-\alpha$ credible regions
  is a \defn{family of $1-\alpha$ credible balls} if,
  for all $\pi \in F$ and
  $\marg{\pi}$-almost all $x \in X$,
  the posterior $\Post{x}{\pi}\,$ has a mean and, for almost all $u \in [0,1]$,
  $\charfunc{\tau_{\pi}(x,u)} = \credball{\Post{x}{\pi}}{\alpha}{}$, i.e.,
  $\tau_{\pi}(x,u)$ is a $1-\alpha$ $\Post{x}{\pi}$-credible ball.
\end{definition}

Note that the $\mu$-measure of a $1-\alpha$ $\mu$-credible ball is no less than $1-\alpha$, and may be strictly more. For this and other reasons, we do not apply our main result to this family of credible regions. Instead, we show that there exists a family of credible regions that both \emph{approximates} this family \emph{and} has matching priors. See \cref{SecConstructions} for the explicit construction.

We formalize our notion of approximation in terms of \defn{fattenings}  of the supports of randomized set estimates. Recall that, for any metric space $Y$, subset $A \subseteq Y$, and real $\epsilon>0$, the \emph{$\epsilon$-fattening} of $A$ (in $Y$) is the set $A_{\epsilon}=\{y\in Y \, : \, \inf_{x \in A} d(x,y)<\epsilon\}$.
Using the notion of support of a random set (such as a credible region), our main result, \cref{ThmMatchingPriorSimpleCred}, gives containment relationships between (fattenings of) the supports of our approximate credible balls and those of standard credible balls. 
(The proof of \cref{ThmMatchingPriorSimpleCred} is deferred to \cref{SecSimpleExamples}.)

Before we can state \cref{ThmMatchingPriorSimpleCred}, we must introduce some regularity conditions on the model. Setting notation, for two metric spaces $(\Omega_{i}, d_{i})_{i \in \set{1,2} }$,
let $(\Omega_{1} \times \Omega_{2}, d_{1} \otimes d_{2})$ denote the usual product metric space, with metric
\[
(d_{1} \otimes d_{2}) ((x_{1},x_{2}), (y_{1},y_{2})) = d_{1}(x_{1},y_{1}) + d_{2}(x_{2},y_{2})
\]
for $x_{1},y_{1} \in \Omega_{1}$ and $x_{2},y_{2} \in \Omega_{2}$.

\begin{assumption}[Unnecessary but straightforward]
\label{assumptionde}
  There exists a function $\cd \: \Theta \times X \to \Reals$ such that,
  (i)
  for all $\theta \in \Theta$, $\cd(\theta,\cdot)$ is a density of $P_{\theta}$ with respect to $\nu$;
 (ii)
  $\cd : (\Theta \times X, \norm{\cdot} \otimes \dX) \to (\Reals, \norm{\cdot})$ is $\mC$-Lipschitz continuous for some $\mC\in [0,\infty)$;
  and (iii)
  $\log \cd$ is bounded.
\end{assumption}

This assumption, which we relax in %
 \cref{sec:weakerconditions}, allows us to prove that our \scare{approximate} credible balls are credible regions that meet the hypotheses of our main theorem.

We reiterate that \cref{assumptionde}, like \cref{assumptionjc}, is stronger than necessary. 
In particular, \cref{assumptionde} rules out the many statistical models whose conditional density functions are not bounded away from zero. This includes, e.g., the usual Bernoulli model on $\{{0,1}\}$, whose density is not bounded away from 0 in the neighbourhood of $0$ or $1$. While we give an overly strong assumption in this section, our final results and the approximate credible balls that we define are not limited to models with densities bounded away from zero. To demonstrate this, in \cref{SecAppBern}, we show how to obtain the same conclusions for the Bernoulli model, even though it does not satisfy \cref{assumptionde}.

Our main application is: 

\begin{theorem}\label{ThmMatchingPriorSimpleCred}
  Suppose \cref{assumptionde} holds
  and $\Theta$ is a compact subset of $\Reals^{d}$.
  Let $\alpha \in (0,1)$ and $\epsilon \in (0,\alpha)$.
  There exists
  a family of $1-\alpha$ credible regions $\tau = \CRegions{\tau}{\pi}{\Priors}$
  and prior $\pi_0 \in \Priors$
  such that
  (i)
  $\pi_0$ is a matching prior for $\tau$; and %
  (ii)
  for all $\pi \in \Priors$ and $x \in X$,
    the support of $\tau_{\pi}$ at $x$
    is contained in the $\epsilon$-fattening of the support of
    the ordinary $1-(\alpha-\epsilon)$ $\Post{x}{\pi}$-credible ball
    and
    the support of 
    the ordinary $1-\alpha$ $\Post{x}{\pi}$-credible ball
    is contained in
    the $\epsilon$-fattening of the support of $\tau_{\pi}$ at $x$.
\end{theorem}

The proof of this result follows from \cref{ThmMatchingPriorsExist}, as described at the end of \cref{sec:weakerconditions}.
We pause again to make several remarks.
First, unlike the literature on asymptotically matching priors,
  the parameter $\epsilon > 0$ appearing in this theorem \emph{does not} depend on the data; it can be chosen by the statistician.
  That is, this result guarantees the existence of a large family of \emph{exact} matching priors
  and shows that their support is close to that of the usual credible balls.
  Although this theorem only claims existence, \cref{SecConstructions} gives explicit constructions.
  These explicit credible regions are not much more difficult to work with or compute than the usual credible regions.

  Second, this theorem is closely related to a result by \citet[Thm.~3.3]{muller2016coverage} that attempts to extend their result for finite spaces to more general spaces.
  In contrast to their result, our theorem guarantees the existence of a matching prior that induces a confidence set at a specified coverage level $1-\alpha$. The result by \citeauthor{muller2016coverage} does not guarantee the existence of a matching prior,  only the existence of a prior that induces a set with coverage at a level that is not too different from $1-\alpha$.
  We also note that our guarantees are explicit and non-asymptotic in the parameter $\epsilon$ that controls the degree to which the supports of our sets agree with typical credible regions.

\subsection{Unbounded log-density: Application to Bernoulli model} \label{SecAppBern}

\cref{ThmMatchingPriorSimpleCred} does not apply as written to models for which the log density $\log q$ is not bounded, since \cref{assumptionde} is not satisfied. Fortunately, it is not too difficult to avoid this problem, essentially by checking that a large class of prior distributions will not concentrate \scare{too much} mass on regions for which $\log q$ is \scare{too large}. 

The following is a concrete application of this idea to the Bernoulli model with sample space $X = \{0,1\}$, parameter space $\Theta = [0,1]$ and conditional density $\cd(\theta,x) = \theta^{x} (1 - \theta)^{1-x}$:

\begin{theorem} \label{ThmBernoulliMain}
\cref{ThmMatchingPriorSimpleCred} holds without \cref{assumptionde} for $\alpha \in (0,0.2)$.
\end{theorem}

The proof is deferred to \cref{SecAppBernProofs}, and contains a discussion of how the basic proof strategy can be extended to other common probability models with unbounded log-densities.

\section{Constructions} \label{SecConstructions}

\subsection{Construction of uniformly Lipschitz credible regions}\label{secfamilymain}

Our main result, \cref{ThmMainExistenceResult}, establishes the existence of matching priors for families of credible regions that satisfy certain continuity conditions. It turns out that these conditions are not satisfied by most popular families of credible regions, but that it is often easy to find small perturbations that \emph{do} satisfy our requirements. This section gives a proof-free, step-by-step guide for perturbing credible balls, with a concrete example for illustration. See \cref{SecSimpleExamples} of the Supplementary Materials for proofs and an analogous construction for highest-posterior density regions.
Up to this point, we have been working with rejection probability functions $\varphi$. In this section, we use acceptance probability functions $\psi = 1- \varphi$, as they lead to more natural definitions and proofs.

The first step is to define relaxed versions of credible balls, whose acceptance probability functions are Lipschitz continuous:

\begin{definition} [Relaxed credible balls] \label{DefSimpleCredInt}
  Fix a level $1-\alpha \in (0,1)$, \defn{slope} $\beta \in (0,\infty)$,
  and distribution $\mu\in \Priors$ with mean $M(\mu)$.  %
  For every $r > 0$, let
  \begin{gather} \label{EqLevelsBlah}
  \ffB{r}{\theta} \defas \min(1, \, \max(0,\, r - \beta \norm{\theta - M(\mu)})),
  \\
  \cbLR{\mu} \defas \inf \set[\Big]{ r > 0 \st \int_{\Theta} \ffB{r}{} \,\dee\mu \geq  1- \alpha }, %
  \quad\text{and}\quad
     \rcb{\mu}{\alpha}{\beta} \defas \ffB{\cbLR{\mu}}{}.
  \end{gather}
  A \defn{$(1-\alpha)$ relaxed $\mu$-credible ball with slope $\beta$} is
  any random set with \mpf\ 
  $\rcb{\mu}{\alpha}{\beta}$. %
\end{definition}

\begin{figure}[t]
\centering
\includegraphics[width=.49\linewidth]{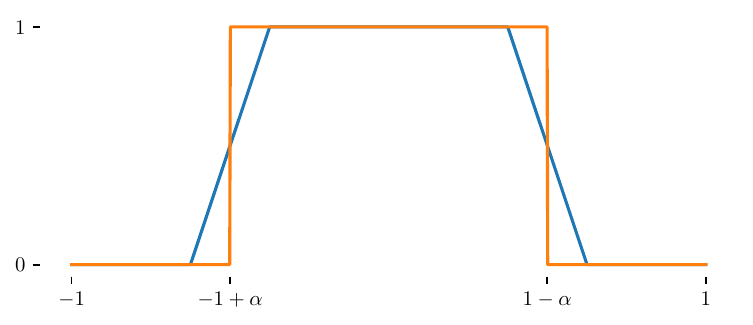}%
\includegraphics[width=.49\linewidth]{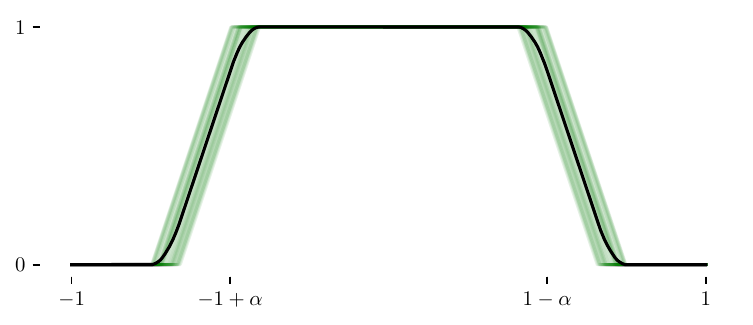}

\caption{Consider a flat posterior on [-1,1]. The left figure portrays (the acceptance probability function of) a standard $1-\alpha$ credible ball (red, rectilinear) and a \emph{relaxed} credible ball with slope $4$ (blue). 
The right figure portrays the corresponding $(1/4,1/4)$-\emph{perturbed} credible ball (solid black), computed using a Monte-Carlo estimate (green) of the defining integral, \cref{tfrsetdefnMain}.
}
\end{figure}

We know that $1-\alpha$ credible balls do not always have $1-\alpha$ credibility, e.g., in the presence of atoms. Further, their rejection probability functions are never Lipschitz. In contrast, we establish in \cref{rcbisacredset} that relaxed credible balls always have the right level and their rejection probability functions are $\beta$-Lipschitz.

\subsection{Perturbed credible regions}\label{subsecpertAlt}

In our next step, we give  a general construction for  perturbing a generic family of credible regions whose rejection probability functions are uniformly Lipschitz in the data and prior. 

The basic idea behind our construction is to smoothly average the usual acceptance probability function across some small range of levels, then perform a small post-averaging correction to obtain the exact level.
To set notation, for $\eta \in (0,1)$, let $\pb{\eta} \: \NNReals \to \Reals$ be the density of some probability distribution on $\NNReals$ such that the support of $\pb{\eta}$ is $[\eta,1]$ and $\pb{\eta}$ is continuous. For $c > 0$, let $\pb{\eta}^{(c)}(x) = c^{-1} \pb{\eta}(\frac{x}{c})$ be the probability density after the change of variable transformation $x \mapsto c x$. 
We use this density to define a random inflation to the level of an acceptance probability function.

\renewcommand{\gamma}{}

We now introduce a base family of credible regions. 
Fix $\beta > 0$.
For every $\alpha \in (0,1)$,
let $\rtau^{\alpha} = \CRegionsv{\rtau^\alpha}{\pi}{\pi}{\Priors}$ be a family of $1-\alpha$ credible regions
with acceptance probability function $\bapf[\alpha]{\beta}{}{}{}$
such that $\bapf[\alpha]{\beta}{\pt{\pi}{\gamma}}{\theta}{x}$ is $\beta$-Lipschitz in $\theta \in \Theta$ for every $\pi \in \Priors$ and $x \in X$. Let $\rtau = (\rtau^\alpha)_{\alpha \in (0,1)}$.

For $\alpha,\eta \in (0,1)$ and $\delta \in (0,\alpha)$, 
the \defn{$(\eta,\delta)$-perturbation of $\bapf[\alpha]{\beta}{}{}{}$} is the function 
$(\theta,x,\pi) \mapsto \papf{\pi}{\theta}{x} : \Theta \times X \times \Priors \to [0,1]$ given by
\[ \label{tfrsetdefnMain}
\papf{\pi}{\theta}{x}  
= \int_{[\delta \eta,\delta]} \bapf[\alpha-z]{\beta}{\pt{\pi}{\gamma}}{\theta}{x} \,\pb{\eta}^{(\delta)}(z) \dee z.
\]

In general, $\papf{\cdot}{\cdot}{\cdot}$ is not an acceptance probability function for a family of $1-\alpha$ credible regions. However, it has  credibility \emph{at least} $1-\alpha$, and it is $\beta$-Lipschitz if each of the associated acceptance probability functions are.

For $x \in X$ and $\pi \in \Priors$,
define
\[
\corr{\pi}{x} (r) = \Post{x}{\pi}(\max(0, \papf{\pi}{\cdot}{x} - r)), \qquad r \in [0,1],
\]
and
\[
R(x,\pi) = \sup \set{ r \in [0,1] \st \corr{\pi}{x}(r) \geq  1 - \alpha }.
\]

Under our assumptions, one can check that $R$ takes finite values (see \cref{minorlem}). We are now in a position to give the main definition of this section: %

\begin{definition} [Perturbed credible regions] \label{DefContFamilyCredibleMain}
Fix $\beta > 0$
and let $\rtau,\alpha,\eta,\delta$ be defined as above.
The \defn{family of $(\delta,\eta)$-perturbed $1-\alpha$ credible regions based on $\rtau$} 
is one whose acceptance probability function $\gfinalset{}{}$ satisfies
\[ \label{credformMain}
\gfinalset{\pi}{x} = \max(0, \papf{\pi}{\cdot}{x} - R(x,\pi))
\]
for all $x \in X$ and $\pi \in \Priors$. 
\end{definition}

This family will serve as our approximation of credible balls.
In particular, the family of credible regions shown to exist in the proof of \cref{ThmMatchingPriorSimpleCred} is exactly that in \cref{DefContFamilyCredibleMain}, with  $\rtau$ given by the family constructed in \cref{DefSimpleCredInt}. The choice of parameters is given explicitly in \cref{EqExpParam} in \cref{SubsecProofMainThm}.

\renewcommand{\gamma}{\oldgamma}

\section{Unsuitability of several credible regions}\label{secrelation}

\subsection{Lack of continuity} \label{SubsecLacCont}

In order to apply our techniques, we require  our families of credible regions to be continuous functions of their associated prior distributions.
We point out credible balls (\cref{DefUsualInterval}) and highest-posterior density regions (\cref{DefUsualHPD}) do not have this property.

\begin{example} [Credible ball discontinuity] \label{ExCredBallDisc}
Define the family $\set{ \mu_{c} }_{c \in [0,1]}$ of distributions 
\[
\mu_{c} = (1 - c)\, \mathrm{Unif}([-0.1,0.1]) + c \, \mathrm{Unif}([-11,-10] \cup [10,11]), \ \text{for }c \in [0,1].
\]
For any fixed level $1-\alpha$, the support of the (characteristic function of the) credible ball $\credball{\mu_{c}}{\alpha}{}$ jumps abruptly from a subset of the interval $[-0.1,0.1]$ to a superset of the interval $(-10,10)$ as $c$ goes from just above $1-\alpha$ to just below $1-\alpha$.
In particular the map $c \mapsto \credball{\mu_{c}}{\alpha}{}$ is discontinuous at $c=1-\alpha$ in any reasonable topology on continuous functions.
\end{example}

\begin{example} [Highest-posterior density region discontinuity]
Define the parameterized families $\set{ \mu_{i,c} }_{c\in[0,1]}$, for $i \in \set{1,2}$, of distributions by their densities
\begin{align}
\rho_{1,c}(\theta) &\propto 2 + c \, \sin(\theta), \qquad \theta \in [0,2 \pi], \text{ and }\\
\rho_{2,c}(\theta) &\propto 2 - c \, \sin(\theta), \qquad \theta \in [0,2 \pi].
\end{align}
It is clear that 
$\mu_{1,c}, \mu_{2,c}$ converge to the uniform measure in any reasonable topology, as $c$ goes to zero; and
at level $1-\alpha = \frac{1}{2}$,
the supports of
the level $1-\alpha$ highest-posterior density regions for $\mu_{1,c}$ and $\mu_{2,c}$
are $[0,\pi]$ and $[\pi,2\pi]$ respectively, for all $c > 0$.
In particular, these observations imply that the \emph{measures} $\mu_{1,c}, \mu_{2,c}$ converge to the same point as $c \to 0$, but the associated credible regions do not.
\end{example}

Although both these examples are very simple, the same phenomena occur for many other distributions and families of credible regions. To construct other distributions that give discontinuities for the same families of credible regions, one could, e.g., take any mixture of a well-behaved distribution with one of these pathological examples. On the other hand, Example \ref{ExCredBallDisc} gives a similar discontinuity for posterior-quantile credible regions (though at a different level). 

\section{Finding Matching Priors} \label{SecFindingMP}

As in the work of \citet{muller2016coverage}, our results show matching priors \textit{exist} but not how to \textit{find} them. It is an open problem 
to identify interesting general families of credible regions that reliably yield matching priors described by simple formulas or efficient exact computations.
In lieu of this, we might seek an approximation algorithm with a guarantee such as,
for any user-specified error $\epsilon>0$, 
the computed prior $\pi_{\epsilon}$ %
yields credible regions at level $1-\alpha$ that are also confidence sets at level (at least) $1-\alpha - \epsilon$.
We say such a prior is \emph{approximately matching}.
Note that the error $\epsilon$ should \emph{not} depend on the amount of data, in contrast to the literature on asymptotically matching priors \citep{datta2012probability}.

Finding such an algorithm is a problem left open by this work. 
As a potential step towards an approximation algorithm,
a key observation is that our existence proof relies on finding the fixed point of a map (see \cref{EqNMMap}). As suggested by \citet{muller2016coverage}, it is natural to investigate whether or not iterates of this map converge to a fixed-point. In the remainder of this section, we pursue this question for the Bernoulli model presented in \cref{ExBernoulliMatching}.

Recall from \cref{ExBernoulliMatching} that, in the Bernoulli model, the uniform prior on $\{\alpha,1-\alpha\}$ is a matching prior at level $1-\alpha$ for the set estimator $\tau$ given by $\tau(0,\cdot) = [0,1-\alpha)$ and $\tau(1,\cdot) = (\alpha,1]$.  (See \cref{ThmNewBernEx1} for the proof.)
What relationship does this exact matching prior have to fixed points of our map? Towards an answer, we discretize the parameter space, $[0,1]$, of the Bernoulli model, making it straightforward to both iterate the map and to compute the coverage. 
Taking the slope $\beta=100$ in order to produce relaxed credible balls with sharp edges,
we find set estimators with coverage close to the nominal level ($\approx 0.949$ versus $0.95$),
after only 25 iterations on a uniform 500-point discretization.
Using interpolation to produce densities, 
\cref{fig:iterations} visualizes the sequence of priors corresponding to iterates of our map.
Numerically, the final, approximately matching prior is close to our hand-crafted exact matching prior.  

Figure \ref{fig:iterations} is far from a proof of convergence: in principle this sequence might achieve coverage of $0.9499$ but never reach, e.g., $0.949998$. While we do not have proof of convergence in general, for this specific example, we can check that a natural simple sequence of priors and credible regions inspired by this picture \textit{both} (i) converge to a limit that is an exact matching prior \textit{and} (ii) have coverage converging to the nominal level $1-\alpha$; see \cref{ThmNewBernEx2} (\cref{SecAppFindingMP}).

Since it is straightforward to numerically estimate the coverage of a proposed credible region in the regime of \textit{finite} sample spaces and \textit{low-dimensional} parameter spaces, the same approach of guessing a matching prior, potentially improving it via iteration, and then checking if it works may be useful for other problems in this regime.
In \cref{SecAppBin}, we repeat the numerical part of this investigation in the binomial model with $n\le 5$ independent trials ($n=1$ being the Bernoulli model). For a range of slopes $\beta$, we quickly arrive at approximately matching priors. %
For $n=2$ trials, the limiting behaviour (as the slope $\beta$ increases) is similar to the $n=1$ case and there is numerical agreement with a hand-crafted exact matching prior. 
The limiting behaviour appears to shift, however, for $n\ge3$, possibly due to slower convergence and more challenging numerical issues. 
While the algorithm we study here may point in a fruitful direction, 
new algorithmic ideas may be needed to scale  to more complex problems.

\begin{figure}[t]
\centering
\includegraphics[height=3cm,trim=10pt 45pt 30pt 10pt,clip]{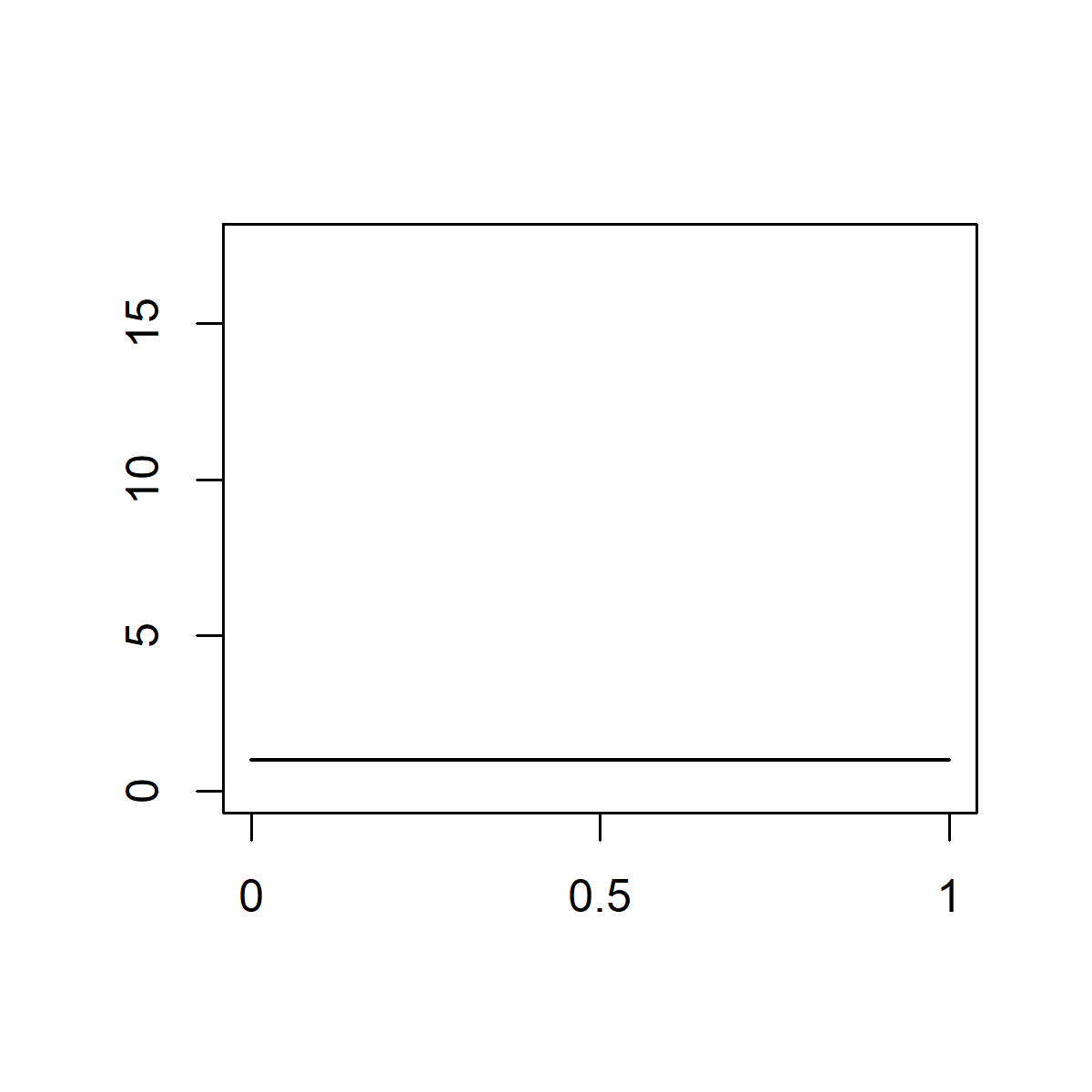}%
\includegraphics[height=3cm,trim=47pt 45pt 30pt 10pt,clip]{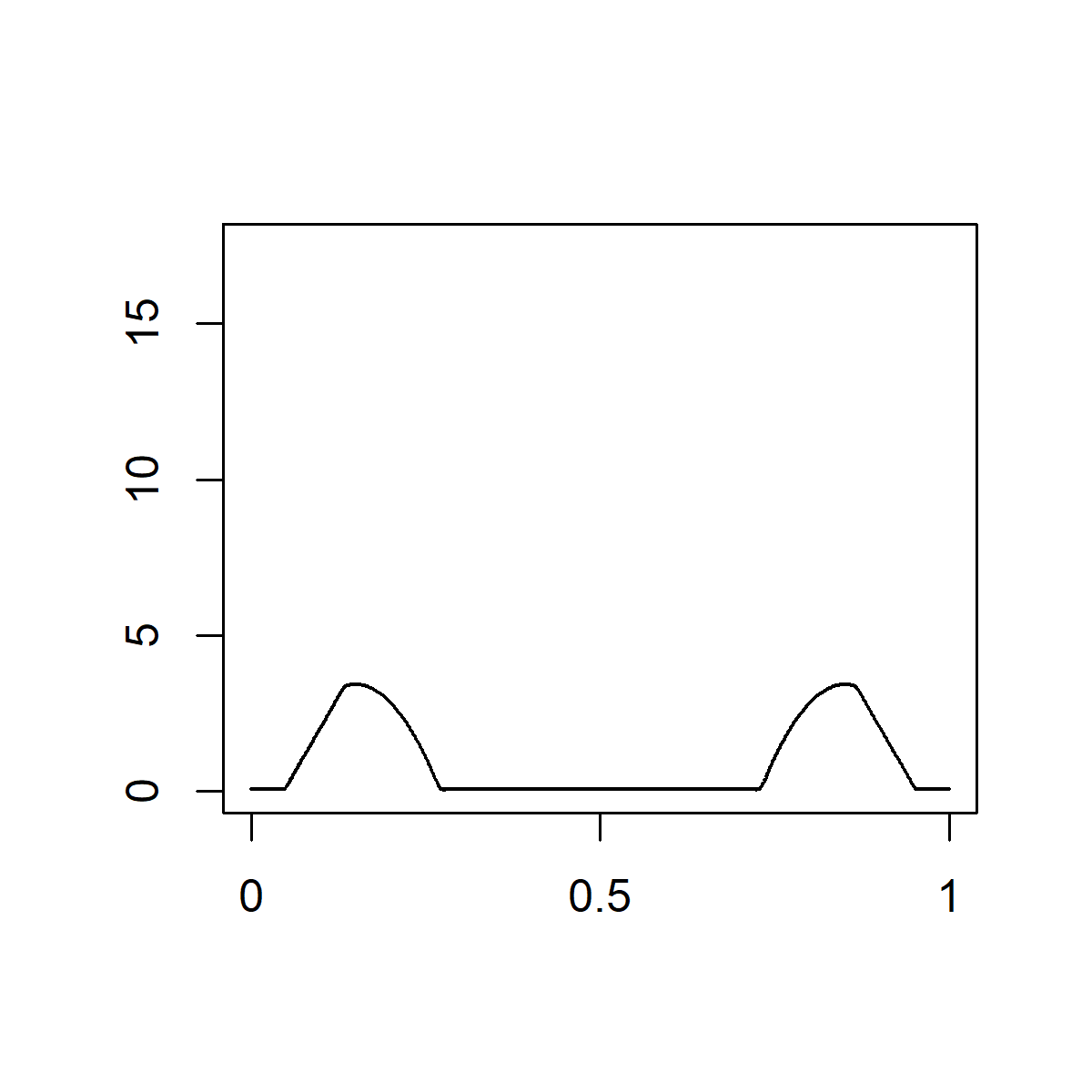}%
\includegraphics[height=3cm,trim=47pt 45pt 30pt 10pt,clip]{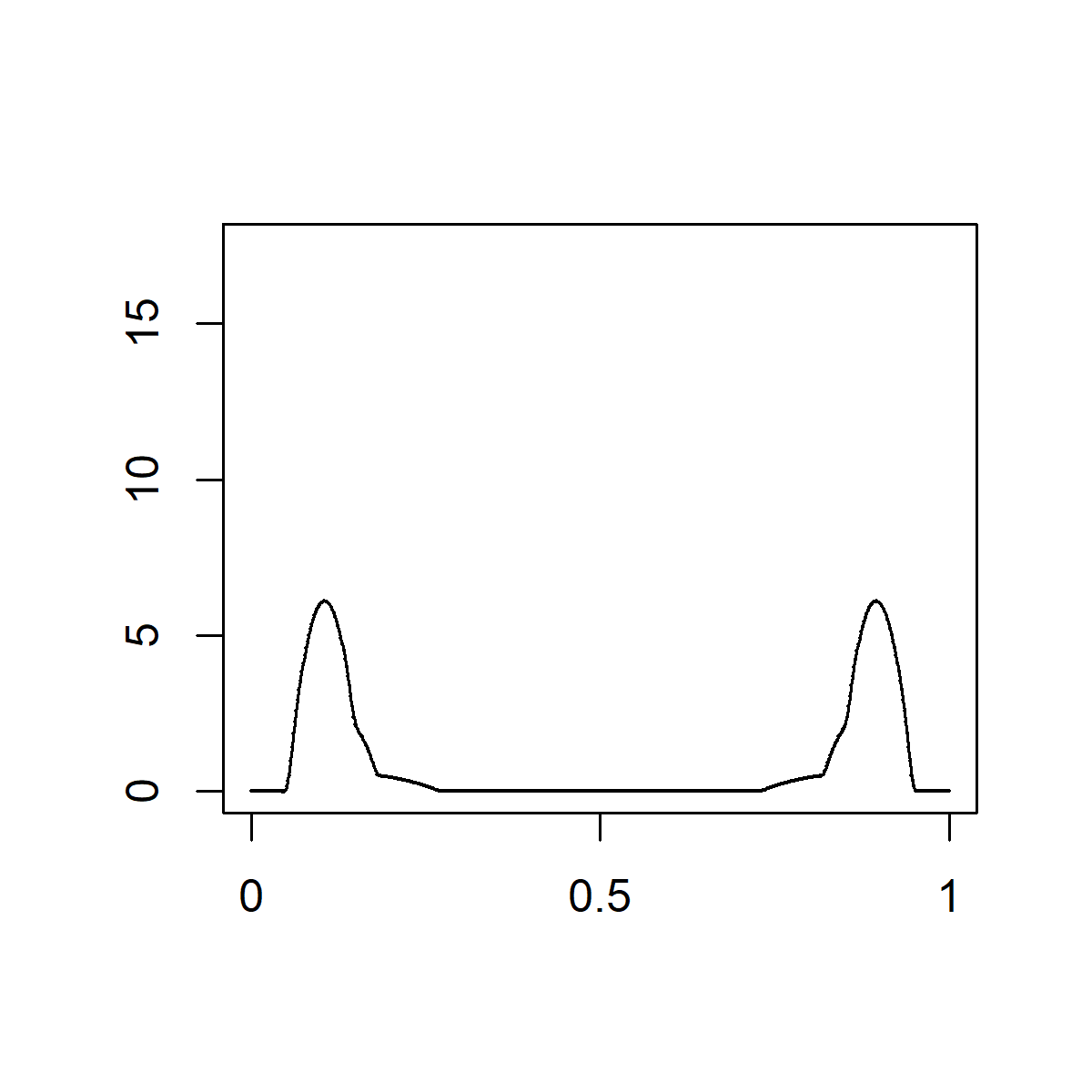}%
\includegraphics[height=3cm,trim=47pt 45pt 30pt 10pt,clip]{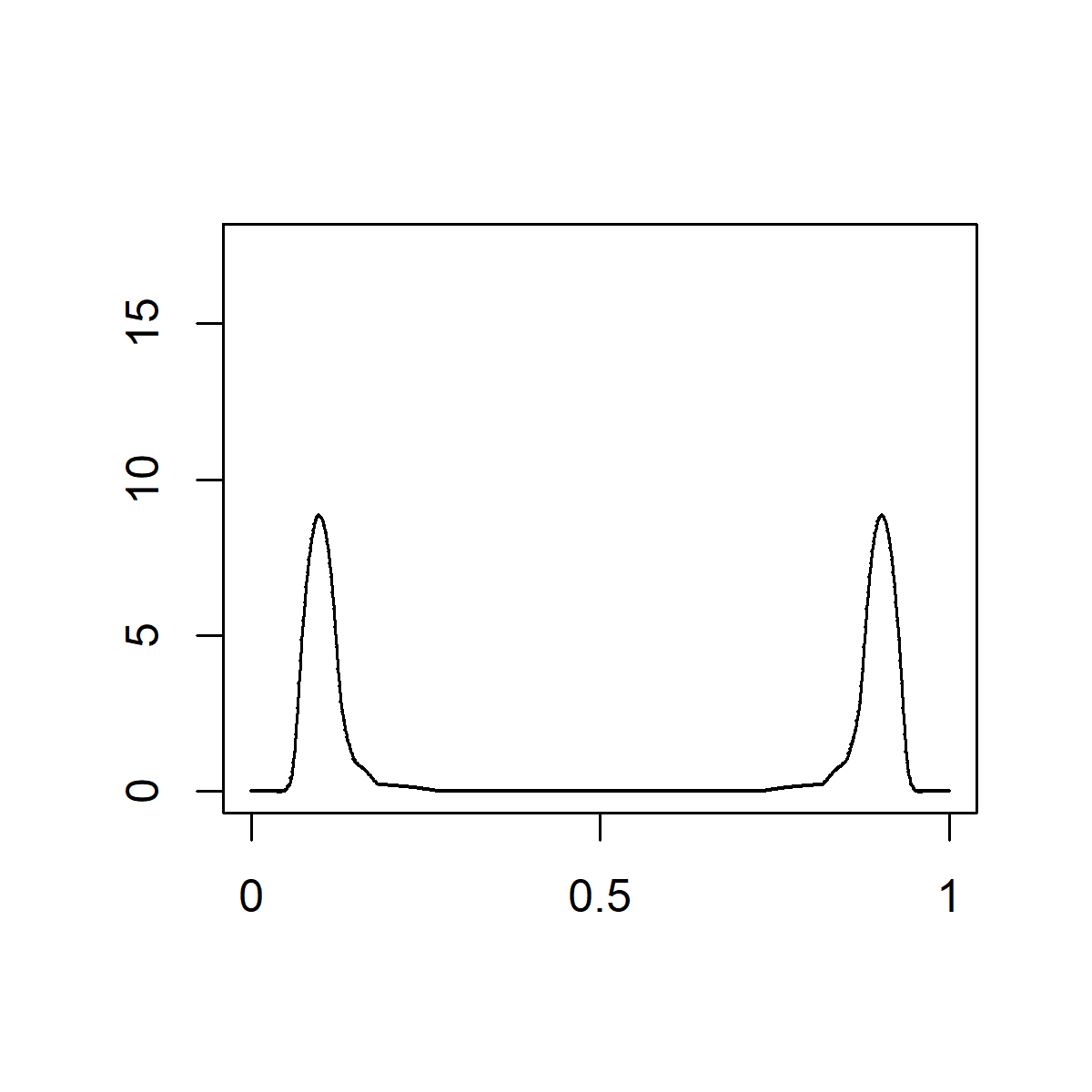}
\includegraphics[height=3cm,trim=47pt 45pt 30pt 10pt,clip]{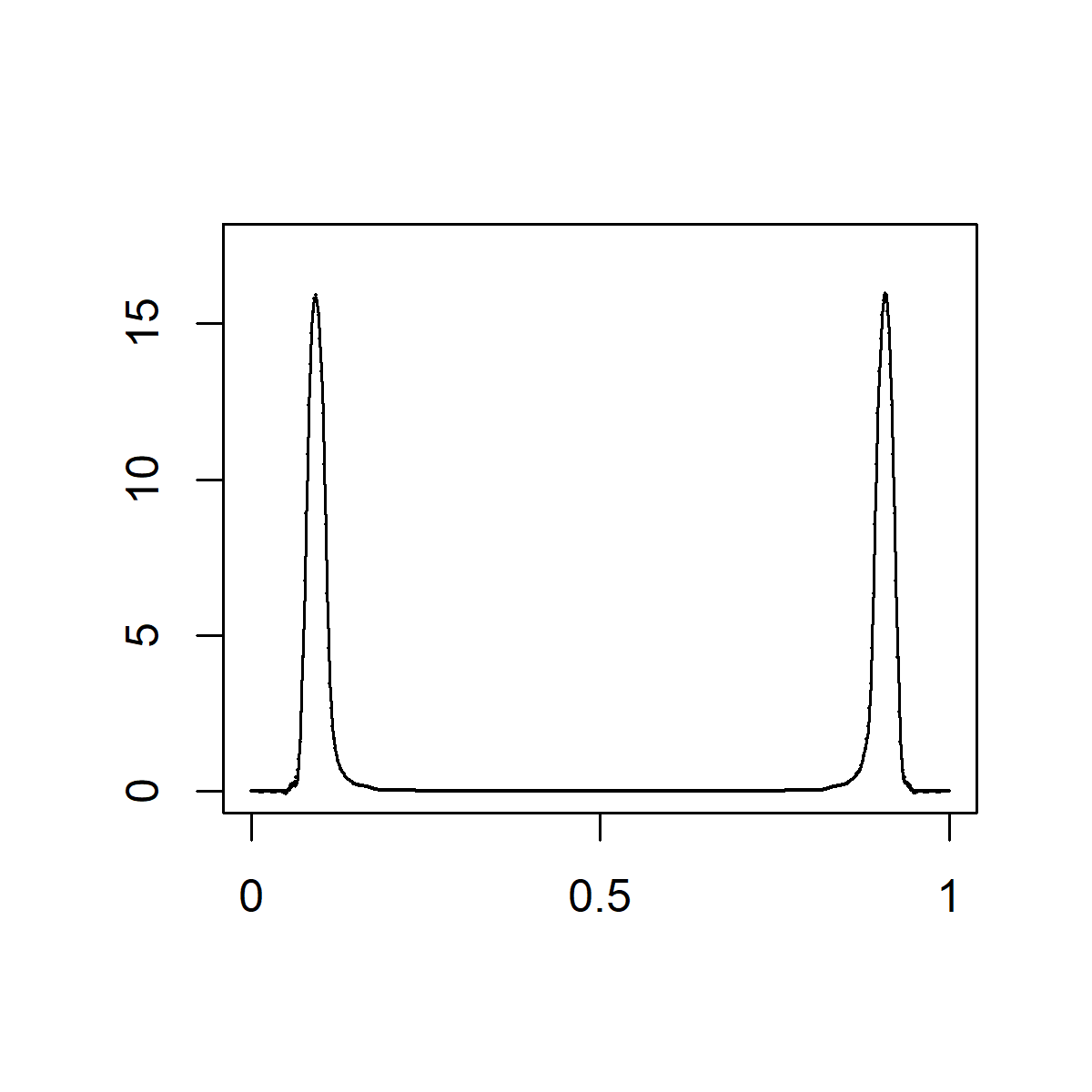}

\caption{Densities of estimated matching priors for discretized Bernoulli model 
after $0,1,3,6$, and $25$ iterations of $Q_{S}$
for relaxed credible balls (\cref{DefSimpleCredInt}, $\alpha = 0.05$, $\beta=100$).
Final coverage $1-\alpha' \approx 0.949$.}\label{fig:iterations}

\end{figure}
 
To summarize, for the Bernoulli example, the problem of computing a matching prior seems as well-behaved as we could hope. A natural heuristic algorithm identifies an approximately matching prior after a small number of iterations, it is straightforward to check that the associated credible regions have good coverage, the estimated prior is, visually, close to an exact matching prior, and some natural sequences of priors and credible regions inspired by this figure have coverage that converges to the nominal level (\cref{ThmNewBernEx2}). 
How generically we should expect to find such behaviour is a theoretical problem left open by this work.

\section{Discussion}\label{secdiscussion}

Our work has two main messages: the first is that models with matching priors are common. Informally, we have shown that they exist for large nonparametric families of models, not merely for special parametric ones. The second message is that the existence of a matching prior may depend quite heavily on the specific family of credible regions. Both of these messages are different from those one might glean from other work wherein matching priors are found to be rare for particular families of credible regions \citep{sweeting2008predictive}

Our work does not provide strong guidance on how to find matching priors, but \cref{SecFindingMP} suggests that a very natural algorithm works ``well." A key theoretical problem left unsolved by this paper is finding efficient algorithms with concrete convergence / approximation guarantees.

While we give \textit{sufficient} conditions for the existence of matching priors, we do not provide matching \textit{necessary} conditions. We leave open the question of which conditions can be relaxed. We draw particular attention to the assumption that the parameter space is compact.  
Non-compactness seems to play an important role in negative results \citep{sweeting2008predictive}.

\makebib

\newpage

\renewcommand{\gamma}{\oldgamma}

\section{Guide to supplementary material}

Our supplementary material is divided roughly into two parts. The first part, Sections \ref{sec:existenceproof} to \ref{SecAppBernProofs}, contains proofs to theorems in the main body and also describes extensions to those main results. 
The second part, Sections \ref{AppRoutineAnalysis} to \ref{SecAppFindingMP}, contain proofs of deferred technical details. 

Since the first part of the appendix includes both our main proof techniques \emph{and} details on how our arguments can be extended to situations not directly covered in the main paper, we expect parts of it to be of interest to most readers. We give a quick guide to what is included in these sections.

\cref{sec:existenceproof} contains the main technical details of the paper, including the proof of \cref{ThmMainExistenceResult}. The proof relies on nonstandard analysis, transferring the result of \citet{muller2016coverage} from finite (parameter) spaces to hyperfinite ones, and then using regularity to obtain the result for compact spaces.

We start \cref{sec:existenceproof} by introducing and discussing 
\cref{WeakThmMainExistenceResult}, a substantial strengthening of our main result, \cref{ThmMainExistenceResult}. We then introduce basic notions from nonstandard analysis and define nonstandard counterparts of the basic statistical structures. For hyperfinite parameter spaces, we establish the existence of nonstandard matching priors under moderate regularity conditions.
When the parameter space is compact, every nonstandard prior is associated with a standard prior, namely its so-called push-down. In order to yield a standard result,
we show that the *Wasserstein distance between any nonstandard prior and its push-down is infinitesimal when the parameter space is compact.
We then use this result to show that, under moderate regularity conditions,
the push-down of nonstandard matching prior is a standard matching prior.

\cref{secmorenote,SecSimpleExamples,SecAppBernProofs} contain proofs for extensions of the applications in \cref{SecApplHead}. The main result contained in \cref{secmorenote,SecSimpleExamples} is \cref{ThmMatchingPriorSimple}, which says that we can obtain matching priors for variants of the usual credible balls and highest-posterior density (HPD) regions that satisfy our regularity conditions. We view this result as an extension of the main result of \cref{simplexist}, which only discusses the existence of matching priors for variants of the credible ball regions. \cref{secmorenote} contains a precise statement of \cref{ThmMatchingPriorSimple}, and also gives several lemmas that allow one to more easily verify its main assumptions. \cref{SecSimpleExamples} is the longest section in the supplement. It defines and analyzes explicit constructions for two families of credible regions: the family that is introduced informally in \cref{SecApplHead} to study credible balls, and a new family that is used in \cref{secmorenote} to study HPD regions. The section is long largely because we must verify that these constructions do in fact give credible families, and that they are continuous in the appropriate way.

\cref{SecAppBernProofs} contains a proof of the main result in \cref{SecAppBern}. We recall the following basic problem that motivated \cref{SecAppBern}: the most straightforward way to check our continuity condition involves checking that the map taking data and a prior to a posterior is uniformly continuous. Unfortunately, this map is \emph{not} uniformly continuous, even for many simple textbook models.  \cref{SecAppBernProofs} shows how to slightly tweak \cref{WeakThmMainExistenceResult} so that it applies even if this map is not uniformly continuous. Although we carry out the analysis specifically for the Bernoulli model, the same strategy can be applied to other models whose log-likelihood functions have isolated discontinuities. 
In \cref{AppRoutineAnalysis,AppNSASec,AppSec5Proofs,secbernproof}, we present deferred proofs. %
In \cref{SecAppExamples}, we present exact matching priors for the binomial and Gaussian location model.
Finally, in \cref{sec8app}, we present proofs and experiments referenced in \cref{SecFindingMP}.

\section{Existence of matching priors}
\label{sec:existenceproof}

\subsection{Overview}

In this section, we prove \cref{ThmMainExistenceResult}, our main result on the existence of matching priors. We make heavy use of \emph{nonstandard analysis} \citep{AR65}, an area of mathematical logic with tools that provide a bridge between discrete and continuous sets.
In particular, using \emph{saturation}, one can construct nonstandard models
possessing \emph{hyperfinite sets}, i.e., infinite sets which possess all the first-order logic properties of finite sets.
By considering a hyperfinite cover of the original parameter space, we can transfer the finite result by \citet{muller2016coverage} to this cover, under some moderate regularity conditions. Deriving a standard result from this nonstandard construction is the bulk of the work.

Our proof proceeds in stages:
In \cref{SecWeakUnif}, we introduce significantly weaker conditions than \cref{assumptionjc}.

In \cref{secnotation}, we introduce some notation related to nonstandard analysis.

In \cref{SecExNSMatch}, we show that \emph{nonstandard} matching priors exist under these weaker conditions (\cref{hyperconfidence}).
  We note that this result does not require the parameter space $\Theta$ to be compact.

In \cref{SecWassPuss}, we show that nonstandard distributions can be ``pushed down" to standard distributions that are infinitesimally close, so long as the relevant parameter space is compact (\cref{pdclose}).

In \cref{SecMainThmRes}, we combine these results to prove the most general version of our main theorem (\cref{WeakThmMainExistenceResult}), which implies \cref{ThmMainExistenceResult}.

We have deferred most proofs in this section to \cref{AppNSASec}.
For those who are interested in applications of nonstandard analysis to statistics and probability, \cref{AppNSASec} includes detailed arguments. 

\subsection{Weakening uniform assumptions}\label{SecWeakUnif}

In this section, we weaken \cref{assumptionjc} considerably.
This improvement is critical for obtaining the results in \cref{SecAppBern}, which say that matching priors exist for the Bernoulli model \emph{despite} the fact that its likelihood is not bounded away from zero.
Fix a statistical model $\SModel$, level $1-\alpha \in (0,1)$ and family $\tau=\CRegions{\tau}{\pi}{\Priors}$ of $1-\alpha$ credible regions with rejection probability function $\frset{}{}{}$.
For every $\theta\in \Theta$ and $\pi \in \Priors$, define
\[ \label{EqDefZMap}
z_{\theta}(\pi)=\int_{X} [\frset{\pi}{\theta}{x}-\alpha] P_{\theta}(\dee x).
\]
We also define $z^{+}_{\theta}(\pi)=\max\{0, z_{\theta}(\pi)\}$. 

We need to move between probability measures defined on subspaces to their extensions on the whole space.
For every measurable subset $S\subseteq \Theta$, endowed with the trace $\sigma$-algebra $\mathcal F$ from $\Theta$,
and every probability measure $\pi \in \PM{S}$, 
define $\extn{\pi}{\Theta}$ to be the unique extension in $\Priors$ 
satisfying $\extn{\pi}{\Theta} = \pi$ on $\mathcal F$.
By a slight abuse of notation, for every subset $M\subseteq \PM{\Theta}$, 
define $\Restrict{M}{S}=\{\pi\in \PM{S} : \extn{\pi}{\Theta}\in M\}$ 
to be those measures on $S$ whose extension to $\Theta$ belongs to $M$.
These definitions extend to spaces other than $\Theta$ as expected.
In particular, we define these two notions in the same way for internal probability measures.
Finally, for finite sets $S \subseteq \Theta$, 
we sometimes abuse notation and identify elements of $\PM{S}$ with their probability mass functions.   

Following \cite{muller2016coverage}, for finite $S$
we define $\Q{S} \, : \, \PM{S} \to \PM{S}$ by the formula 
\[ \label{EqNMMap}
\Q{S}(\pi)(s) = \frac{\pi(s) + z^{+}_{s}(\pi)}{\sum_{t \in S} \{ \pi(t) + z^{+}_{t}(\pi) \}}, \quad \pi \in \PM{S},\ s \in S.
\]

\begin{assumption}\label{assumptionjcAlt}
Let $\Priors$ be equipped with Wasserstein metric,
and let $\frset{}{}{}$ be the rejection probability function for the family $\tau$ in $\SModel$.
There exists a set $\FS \subseteq \PowerSet(\Theta)$ consisting of finite subsets of $\Theta$, a nonempty set $\Main \subseteq \Priors$, and an open set $\oMain\subseteq \Priors$ satisfying the following conditions: 
\begin{enumerate}
\item \label{cfinite}
For all finite subsets $A \subseteq \Theta$, there exists $B\in \FS$ such that $A\subseteq B$.
\item \label{ccont}
For every $\theta\in \Theta$, the function $\frset{\cdot}{\theta}{\cdot} \: X\times \oMain \to [0,1]$ is jointly continuous.
\item \label{copen}
There exists an open set $U\subseteq \Priors$ such that $\Main\subseteq U\subseteq \closure{U} \subseteq \oMain$.       
\item \label{cimage}
For all $S \in \FS$, $\Q{S}(\Restrict{\oMain}{S}) \subseteq \Restrict{\Main}{S}$.
\end{enumerate}
\end{assumption}
Clearly, \cref{assumptionjcAlt} is no stronger than \cref{assumptionjc}, as we can always take $\Main= \PM{\Theta}$.
If $\oMain$ is, in particular, the $r$-fattening of $\Main$ (w.r.t.\ Wasserstein metric) for some constant $r>0$,
then \cref{copen} of \cref{assumptionjcAlt} holds, e.g., when $U$ is the $r/2$-fattening of $\Main$. The results in the section imply that \cref{assumptionjcAlt} is strictly weaker than \cref{assumptionjc}, 
since the Bernoulli model does not satisfy \cref{assumptionjc}.

\cref{assumptionjcAlt} was introduced to circumvent a technical obstruction that
 does not appear in the finite setting of  \cite{muller2016coverage} but is 
 very common for models with continuous parameter spaces: the likelihood may not be \textit{bounded away from 0} even if none of the likelihoods are \textit{exactly} 0. For example, this occurs for the Bernoulli model $q(\theta,x) = \theta^{x} (1-\theta)^{1-x}$ in the neighbourhoods of 0 and 1. The map $\Q{S}$ often behaves quite badly for measures that have large support in these ``bad" neighbourhoods, and this causes our proof strategy to fail. The idea behind \cref{assumptionjcAlt} is that it is sometimes enough to study $\Q{S}$ only on a subset of priors that do not concentrate on such ``bad" neighbourhoods, as long as the subset is fixed by $\Q{S}$. Informally, \cref{assumptionjcAlt} says that there exists a large subset of $\PM{\Theta}$ which 
 (i) avoids the ``bad" neighbourhoods enough for the map $\frset{\cdot}{\theta}{\cdot}$ to be jointly continuous and
 (ii) is stable under $\Q{S}$.

Let $\FPM{\Theta}$ be the collection of all priors with finite support.

\begin{mycondition}{CC}[Continuous coverage]\label{assumptionztv}
A family $\tau= \CRegions{\tau}{\pi}{\Priors}$ has \defn{continuous coverage} for a subset $A\subseteq \Priors$
when, for all $\theta\in \Theta$, the map
  $z_{\theta} \: \Priors\to [-1,1]$, 
  is continuous with respect to total-variation distance on $\FPM{\Theta}\cap A$.
\end{mycondition}

We abbreviate this condition by writing ``\cref{assumptionztv} holds for $\tau,A$''. 
We now establish \cref{assumptionztv} from \cref{assumptionjcAlt}. 

\begin{lemma}\label{zcontweak}
  Let $\SModel$ be a statistical model where $\Theta$ is compact and let $\tau$ be a family of credible regions on $\SModel$ with associated rejection probability function $\frset{}{}{}$. 
  Suppose \cref{ccont} of \cref{assumptionjcAlt} holds for some open subset $\oMain \subseteq \Priors$ with respect to Wasserstein metric.
  Then, for every $\theta\in \Theta$, the function $z_{\theta} \: \Priors\mapsto [-1,1]$ is continuous on $\oMain$ with respect to Wasserstein metric.
\end{lemma}
\begin{proof}
  Fix $\theta\in \Theta$.
  Pick $\epsilon>0$. Then there exists a compact set $K\subseteq X$ such that $P_{\theta}(K)>1-\frac{\epsilon}{3}$.
  It follows that
  \[ \label{IneqL341}
  \abs[\Big]{
    \int_{K}[\frset{\pi}{\theta}{x}-\alpha] P_{\theta}(\dee x)
    - \int_{X}[\frset{\pi}{\theta}{x}-\alpha] P_{\theta}(\dee x)
  }
  < \frac{\epsilon}{3}
  \]
  for all $\pi\in \Priors$.
  Pick $\pi_0\in \oMain$. 
  As $\oMain$ is open, pick $B$ to be an open ball centered at $\pi_0$ such that $\closure{B}\subseteq \oMain$. 
  We use $r_0$ denote the radius of $B$. 
  As $\Priors$ is compact under the Wasserstein metric, $\closure{B}$ is compact under the Wasserstein metric.  
  As $K\times \closure{B}$ is compact, by \cref{assumptionjcAlt},
  $\frset{\cdot}{\theta}{\cdot}$ is uniformly continuous on $K\times \closure{B}$,
  and so there exists $0<\delta<r_0$ such that $|\frset{\pi}{\theta}{x}-\frset{\pi_0}{\theta}{x}|<\frac{\epsilon}{3}$
  for all $x\in K$ and all $\pi\in \Priors$ with $W(\pi,\pi_0)<\delta$.
  Hence we have
  \[ \label{IneqL342}
  \abs[\Big] {
    \int_{K}[\frset{\pi}{\theta}{x}-\alpha]P_{\theta}(\dee x)
    - \int_{K}[\frset{\pi_0}{\theta}{x}-\alpha]P_{\theta}(\dee x)
  }
  < \frac{\epsilon}{3}
  \]
  for all $\pi\in \Priors$ with $W(\pi,\pi_{0})<\delta$.
  It follows from Inequalities~\eqref{IneqL341} and \eqref{IneqL342} that $z_{\theta}(\cdot)$ is continuous on $\oMain$ with respect to Wasserstein metric.
\end{proof}

As the Wasserstein metric induces a coarser topology than the total-variation metric, the following corollary is an immediate consequence of \cref{zcontweak}.

\begin{corollary}\label{jcimplyztv}
Let $\Theta$, $\tau$, and $\frset{}{}{}$ be as in \cref{zcontweak},
and suppose \cref{ccont} of \cref{assumptionjcAlt} holds for some open subset $\oMain \subseteq \Priors$ with respect to Wasserstein metric.
Then \cref{assumptionztv} holds for $\tau, \oMain$.
\end{corollary}

We now state a stronger version of \cref{ThmMainExistenceResult}, replacing \cref{assumptionjc} by  \cref{assumptionjcAlt}.

\begin{theorem}[Main result]\label{WeakThmMainExistenceResult}
  Fix a statistical model $\SModel$, level $1-\alpha \in (0,1)$, and family $\tau = \CRegions{\tau}{\pi}{\Priors}$
  of credible regions with rejection probability function $\frset{}{}{}$.
  Suppose $\Theta$ is compact and \cref{assumptionjcAlt} holds for $\FS$, $\Main$, and $\oMain$.
  Then there exists a matching prior in $\oMain$ for $\tau$. 
\end{theorem}

\subsection{Notation from nonstandard analysis}\label{secnotation}

We briefly introduce notation from nonstandard analysis.
For those who are not familiar with nonstandard analysis, \citet[][App.~A]{DR2017} provide a review tailored to their statistical application.
\citet{NSAA97,NDV} provide thorough introductions.

We use $\NSE{}$ to denote the nonstandard extension map taking elements, sets, functions, relations, etc., to their nonstandard counterparts.
In particular, $\HReals$ and $\NSE{\Nats}$ denote the nonstandard extensions of the reals and natural numbers, respectively.
Given a topological space $(Y,T)$ and a subset $X \subseteq \NSE{Y}$,
let $\NS{\NSE{X}} \subseteq \NSE{X}$  denote the subset of near-standard elements (defined by the monadic structure induced by $T$)
and let $\ST : \NS{\NSE{Y}} \to Y$ denote the standard part map taking near-standard elements to their standard parts.
In both cases, the notation elides the underlying space $Y$ and the topology $T$, because the space and topology will always be clear from context.
As an abbreviation,
we write $^\circ x$ for $\ST(x)$
for atomic elements $x$.
For functions $f$, we write $^\circ f$ for the composition $x \mapsto \ST(f(x))$.

An internal probability space is a triple $(\Omega,\cF,P)$ such that
\begin{enumerate}
\item $\Omega$ is an internal set;
\item $\cF$ is an internal subalgebra of $\PowerSet{(\Omega)}$; and
\item $P: \cF\to \NSE{[0,1]}$ is an internal function such that $P(\emptyset)=0$, $P(\Omega)=1$ and $P(A\cup B)=P(A)+P(B)-P(A\cap B)$ whenever $A,B\in \cF$.
\end{enumerate}
Note that an internal probability space is hyperfinitely additive. 
Given an internal probability space $(\Omega,\cF,P)$,
let $(\Omega,\LoebAlgebraX{\cF}{P},\Loeb{P})$ denote the corresponding Loeb space, i.e., the completion of the unique extension of $P$ to the (usually external) $\sigma$-algebra generated by $\cF$.
Write $\Loeb{\cF}$ for $\LoebAlgebraX{\cF}{P}$ when the context is clear.
For two internal sets $A$ and $B$, let $\IntFuncs{A}{B}$ denote the set of all internal functions from $A$ to $B$.

\subsection{Existence of nonstandard matching priors} \label{SecExNSMatch}

The basic idea is to extend the statistical model and family of credible regions to the nonstandard universe,
and then restrict our attention to a hyperfinite parameter space $\TTheta$ satisfying $\Theta \subseteq \TTheta \subseteq \sTheta$.
The existence of such a set follows from saturation.
On this hyperfinite parameter space, we may apply a transfer of the result of \citeauthor{muller2016coverage}.
This approach produces an internal probability measure that is, informally, a nonstandard matching prior.
The bulk of the work, which we carry out in later subsections,
is to use regularity conditions to construct a corresponding standard matching prior for the standard model.

We begin by quoting the result of \citet{muller2016coverage},
expressed in our notation.
 
\begin{theorem}%
\label{muresultOrig}
  Fix a statistical model $\SModel$,
  level $1-\alpha \in (0,1)$, and
  family $\tau = \CRegions{\tau}{\pi}{\Priors}$ of $1-\alpha$ credible regions.
  Suppose $\Theta$ is finite and, for every $\theta \in \Theta$,
  $z_{\theta} \: \Priors\to [-1,1]$ defined for $\tau = \CRegions{\tau}{\pi}{\Priors}$
  is continuous with respect to total-variation distance on $\FPM{\Theta}$.
  Then there exists a matching prior for $\tau$.
\end{theorem}

Inspecting the proof of this result, one observes that its authors, in fact, establish a formally stronger result:

\begin{theorem}%
\label{muresult}
  Fix a statistical model $\SModel$,
  level $1-\alpha \in (0,1)$, and
  family $\tau = \CRegions{\tau}{\pi}{\Priors}$ of $1-\alpha$ credible regions.
  Suppose $\Theta$ is finite and let $A$ be a nonempty subset of $\Priors$ such that $A$ is compact when viewed as a subset of $\Reals^{|\Theta|}$.
  Suppose $\Q{\Theta}(A)\subseteq A$ and \cref{assumptionztv} holds for $\tau, A$. Then there exists a matching prior for $\tau$ and, moreover, it belongs to $A$.
\end{theorem}

Note that this strengthens the result in two ways: 
(1) we need only have continuity on the smaller set $A$; and 
(2) we are guaranteed that there exists a matching prior in the smaller set $A$.

In order to state the main theorem in this section, we must introduce several nonstandard structures associated with a statistical model $\SModel$.
First, the set of priors is $\Priors$, and so the set of *priors is $\sPriors$.
The model $P$ is a measurable function from $\Theta$ to $\PM{X}$, and so
its extension $\NSE{P}$ is a *measurable function from $\sTheta$ to $\NSE{\PM{X}}$.
By the transfer principle, for every $\theta \in \sTheta$, $\NSE{P}(\theta)$ is an internal probability measure on the *measurable space $\NSE{(X,\BorelSets X)} = (\NSE{X},\NSE{\BorelSets{X}})$.
For all $\theta \in \sTheta$, we define $\NSE{P}_{\theta} = \NSE{P}(\theta)$.
By convention, we drop the star from the notation $\NSE{\!\!\int}$
for the nonstandard extension of the integral operator
$(f,\mu) \mapsto \int f \dee \mu$. 
For $N\in \NSE{\Nats}$, we write $\IVSpace{\HReals}{N}$ for the set of internal functions from $\{0,1,\dotsc, N-1\}$ to $\HReals$.

\begin{theorem}\label{hyperconfidence}
Fix a statistical model $\SModel$,
level $1-\alpha \in (0,1)$, and
family $\tau = \CRegions{\tau}{\pi}{\Priors}$ of $1-\alpha$ credible regions with
rejection probability function $\frset{}{}{}$.
Suppose \cref{assumptionztv} holds for $\tau, \Main$ for some $\Main\subseteq \Priors$. 
Then, for every hyperfinite set $T \subseteq \sTheta$ such that $\NSE{\Q{T}}(\Restrict{\NSE{\Main}}{T})\subseteq \Restrict{\NSE{\Main}}{T}$ and $\Restrict{\NSE{\Main}}{T}$ is a nonempty *compact set when viewed as a subset of $\IVSpace{\HReals}{|T|}$, 
there exists a *prior $\Pi \in \NSE{\Main}$, with $\Pi(T) = 1$, such that the following statement holds: 
\[\label{saeouNDOEU}
\forall \theta \in T,\ \NSE{z_{\theta}}(\Pi) \leq 0.
\]
\end{theorem}

Our proof relies on restricting a family of credible regions to a \scare{submodel}:
For a function $f: Y \to Z$, write $\Restrict{f}{Y'}$ for the restriction of $f$ to $Y' \subseteq Y$.  A \defn{submodel} of $\SModel$ 
is any statistical model of the form $\SModelalt{\fTheta}$ for $\fTheta \in \BorelSets{\Theta}$. Supermodels are defined as expected.

\begin{definition} \label{DefSubModel}
  Let $\SModelalt{\fTheta}$ be a submodel of $\SModel$,
  and let $\tau = \CRegions{\tau}{\pi}{G}$ be a family of $1-\alpha$ credible regions in $\SModel$.
  Then \defn{the restriction of $\tau$ to $\SModelalt{\fTheta}$} is
  the family $\tau' = \CRegions{\tau'}{\pi}{\Restrict{G}{\fTheta}}$
  and, for all $x \in X$, $u \in [0,1]$, and $\pi \in \Restrict{G}{\fTheta}$,
  \[\label{aoeuduinhd2}
  \tau'_{\smash{\pi}}(x,u) = \tau_{\extn{\pi}{\Theta}}(x,u) \cap \fTheta.
  \]
\end{definition}

The extension and restriction of *credible regions can be defined similarly. 
We now prove the simple (standard) fact that if \cref{assumptionztv} holds for a family,
then it holds for any restriction to a submodel.

\begin{lemma}\label{restrictionpreserves}
  Let $\SModelalt{\fTheta}$ be a submodel of $\SModel$,
  let $\tau = \CRegions{\tau}{\pi}{G}$ be a family of $1-\alpha$ credible regions in $\SModel$
  with rejection probability function $\frset{}{}{}$,
  and let $\tau' = \CRegions{\tau}{\pi}{\Restrict{G}{\fTheta}}$ be the restriction of $\tau$ to $\SModelalt{\fTheta}$.
  Then $\tau'$ is a family of $1-\alpha$ credible regions in $\SModelalt{\fTheta}$
  and, moreover, its rejection probability function $\frsetalt{}{}{}$ satisfies
  $\frsetalt{\pi}{\theta}{x} = \frset{\extn{\pi}{\Theta}}{\theta}{x}$
  for all $\theta \in \fTheta$, $x \in X$, and $\pi \in \Restrict{G}{\fTheta}$.
  Finally, if \cref{assumptionztv} holds for the pair $\tau$, $\Main$ in $\SModel$, then it holds for the pair $\tau'$, $\Restrict{\Main}{\fTheta}$ in $\SModelalt{\fTheta}$. 
\end{lemma}
\begin{proof}
See \cref{Appdef43}. 
\end{proof}

We are now in a position to give a proof sketch whose details could easily be filled in by a reader versed in nonstandard analysis.
For the benefit of readers less familiar with nonstandard analysis, we provide a proof in \cref{Appdef43} where the details are made explicit.

\begin{proof}[sketch of \cref{hyperconfidence}]
  Let $T \subseteq \sTheta$ be nonempty hyperfinite such that $\NSE{\Q{T}}(\Restrict{\NSE{\Main}}{T})\subseteq \Restrict{\NSE{\Main}}{T}$ and $\Restrict{\NSE{\Main}}{T}$ is nonempty and *compact.
  It follows from $T$ being nonempty and hyperfinite that $T \in \NSE{\BorelSets{\Theta}}$ and $\sSModelalt{T}$ is a *submodel of $\sSModel$.
  Let $\NSE{\tau}$ be the extension of $\tau$ and let $\xi = \CRegions{\NSE{\tau}}{\pi}{\NSE{\PM{T}}}$
  be the *restriction of $\NSE{\tau}$ to the *submodel $\sSModelalt{T}$.
  By assumption, \cref{assumptionztv} holds for $\tau, \Main$.
  By the transfer principle, the transfer of \cref{assumptionztv} holds for $\NSE{\tau}, \NSE{\Main}$ in $\sSModel$. 
  Recall that $\Restrict{\NSE{\Main}}{T}=\{\Pi\in \NSE{\PM{T}} : \extn{\Pi}{\NSE{\Theta}}\in \NSE{\Main}\}$.
  The transfer of \cref{restrictionpreserves} implies that the *restriction $\xi$ is a *family of credible regions and
  that the transfer of \cref{assumptionztv} holds for $\xi, \Restrict{\NSE{\Main}}{T}$ in $\sSModelalt{T}$.

  Thus, by transfer of \cref{muresult}, taking $A=\Restrict{\NSE{\Main}}{T}$,
  there exists a *matching *prior $\Pi' \in \NSE{\PM{T}}$ for $\xi$.
  The *prior $\Pi'$ extends to a *prior $\Pi$ on $\NSE{\Theta}$ with $\Pi(T) = 1$.
  By transfer of the definition of matching priors
  and the rejection probability function identity of \cref{restrictionpreserves},
  this implies \cref{saeouNDOEU}, completing the proof.
\end{proof}

\subsection{Push-down measures and Wasserstein metric} \label{SecWassPuss}

In this section, we start by discussing how to construct standard probability measures from internal (hyperfinitely additive) probability measures. There exists a rich literature on this subject (see, e.g., \citep{anderson87,renderpush,aldazrepresent,rossinfinite}).
Such a standard probability measure is usually constructed via the standard part map, and is called the \defn{push-down} of the internal probability measure.
We then show that the *Wasserstein metric between an internal probability measure and the nonstandard extension of its push-down is infinitesimal.
We begin with the definition of a push-down measure:

\begin{definition}\label{defnpushdown}
  Let $Y$ be a Hausdorff space with Borel $\sigma$-algebra $\BorelSets Y$,
  let $\nu$ be an internal probability measure defined on $(\NSE{Y},\NSE{\BorelSets Y})$,
  and let
  \[
  \cC=\{C\subseteq Y \st \ST^{-1}(C)\in \Loeb{\NSE{\BorelSets Y}} \}.
  \]
  The push-down of $\nu$, denoted $\pd{\nu}$, is the countably additive measure on $\cC$ given by
  $\pd{\nu}(C)=\Loeb{\nu}(\ST^{-1}(C))$.
\end{definition}

For general Hausdorff Borel measurable space $(Y,\BorelSets Y)$, the push-down measure $\pd{\nu}$ may not be a probability measure.
In fact, if $\Loeb{\nu}(\NS{\NSE{Y}})=0$ then $\pd{\nu}$ is the null measure on $(Y,\BorelSets Y)$.
However, when $Y$ is compact, the following theorem guarantees that $\pd{\nu}$ is a probability measure on $(Y,\BorelSets Y)$.

\begin{theorem}%
\label{pushdown}
  Let $Y$ be a compact Hausdorff space equipped with Borel $\sigma$-algebra $\BorelSets {Y}$,
  and let $\nu$ be an internal probability measure defined on $(\NSE{Y}, \NSE{\BorelSets{Y}})$.
  Then the push-down measure $\pd{\nu}$ of $\nu$ is the completion of a regular Borel probability measure on $(Y,\BorelSets Y)$.
\end{theorem}

It is natural to hope that the ``distance" between an internal probability measure and the nonstandard extension of its push-down measure is small.
The following example shows that the *total-variation distance between these two internal probability measures may be large.

\begin{example}
  Let $\nu$ be an internal probability measure on $(\NSE{[0,1]},\NSE{\BorelSets {[0,1]}})$ concentrating on some infinitesimal $\epsilon>0$.
  Then $\pd{\nu}$ is a standard probability measure on $([0,1],\BorelSets {[0,1]})$ concentrating on $\set{0}$.
  Thus, the *total variation distance between $\nu$ and $\NSE{\pd{\nu}}$ is $1$.
\end{example}

However, under moderate conditions, the *Wasserstein metric between an internal probability measure and its push-down is small. The details of the proof of this section's main result are somewhat technical and so are deferred to \cref{appsecdef44}, but may be of independent interest. The main result needed in this paper is:

\begin{theorem}\label{pdclose}
  Suppose $Y$ is a compact metric space with Borel $\sigma$-algebra $\BorelSets Y$.
  Let $\Pi\in \NSE{\PM{Y}}$.
  Then $\NSE{W}(\Pi, \NSE{\pd{\Pi}})\approx 0$.
\end{theorem}
\begin{proof}
See \cref{appsecdef44}.
\end{proof}

Subsequent to the completion of this work, \citet{starmeasure} extended \cref{cpartition,pdclose} to hold for bounded $\sigma$-compact metric spaces. 

\subsection{Matching priors for compact parameter spaces} \label{SecMainThmRes}
In this section, we establish a standard result using push-down techniques under regularity conditions. Note that the parameter space, $\Theta$, is compact. 
Let $\Priors$ be equipped with Wasserstein metric.
As $\Theta$ is compact, the Wasserstein metric on $\Priors$ metrizes weak convergence. Thus,  $\Priors$ equipped with Wasserstein metric is a compact space by Prokhorov's theorem.

Fix $\alpha \in (0,1)$ and a family $\tau = (\tau_{\pi})_{\pi \in \Priors}$ 
of $1-\alpha$ credible regions with rejection probability function $\frset{}{}{}$. 
The following proposition is the main technical result:

\begin{proposition}\label{phipush}
  Suppose \cref{ccont} of \cref{assumptionjcAlt} holds for some open subset $\oMain \subseteq \Priors$ with respect to Wasserstein metric.
  Let $\Pi\in \sPriors$ and let $\pd{\Pi}$ denote its push-down.
  Suppose $\pd{\Pi}\in \oMain$.
  Then, for all $\theta\in \Theta$,
  \[
  \int \NSE{\frset{\Pi}{\theta}{x}}\NSE{\Model}_{\theta}(\dee x)\approx \int \NSE{\frset{\NSE{\pd{\Pi}}}{\theta}{x}}\NSE{\Model}_{\theta}(\dee x).
  \]
\end{proposition}
\begin{proof}
See \cref{appsecdef45} for the proof as well as additional technical lemmas.
\end{proof}

We are now ready to prove our main theorem, which ensures the existence of matching-prior for compact parameter spaces under sufficient conditions:

\begin{proof} [of \cref{WeakThmMainExistenceResult}]
  Let $T\in \NSE{\FS}$ be a hyperfinite subset of $\NSE{\Theta}$ that contains $\Theta$ as a subset.
  The existence of such $T$ is guaranteed by the saturation principle.
  By \cref{copen} of \cref{assumptionjcAlt}, there exists a nonempty open set $U\subseteq \Priors$ such that $\Main\subseteq U\subseteq \closure{U} \subseteq \oMain$.       
  As $\FPM{\Theta}$ is a dense subset of $\Priors$ under the Wasserstein distance, we know that 
  \[
  \closure{U}\cap \FPM{\Theta}\neq \emptyset.
  \]
  Let $\ssetPM{T}{\NSE{\Theta}}=\{\Pi\in \sPriors : \Pi(T)=1\}$. Note $\ssetPM{T}{\sTheta} \subseteq \NSE{\FPM{\Theta}}$.
  Then, 
  $\NSE{\closure{U}}\cap \ssetPM{T}{\NSE{\Theta}}$ is a nonempty *compact subset of $\ssetPM{T}{\NSE{\Theta}}$ under the *Wasserstein metric. 
  Thus, $\Restrict{\NSE{\closure{U}}}{T}$ is a nonempty *compact subset of $\NSE{\PM{T}}$ under the *Wasserstein metric, 
  which implies that $\Restrict{\NSE{\closure{U}}}{T}$ is *compact when viewed as a subset of $\IVSpace{\HReals}{|T|}$.
  By the transfer principle,  \cref{cimage} of \cref{assumptionjcAlt} implies
  $\NSE{\Q{T}}(\Restrict{\NSE{\closure{U}}}{T})\subseteq \Restrict{\NSE{\Main}}{T}\subseteq \Restrict{\NSE{\closure{U}}}{T}$.
  By \cref{jcimplyztv}, we know that \cref{assumptionztv} holds for $\tau,\closure{U}$.
  Thus, by \cref{jcimplyztv} and \cref{hyperconfidence}, there is a internal prior $\Pi \in \NSE{\closure{U}}$ with $\Pi(T)=1$ such that
  \[
  \forall t\in T, \ \int \NSE{\frset{\Pi}{t}{x}}\NSE{\Model}_{t}(\dee x)\leq \alpha .
  \]
  As $\Theta\subseteq T$, we have $\int \NSE{\frset{\Pi}{\theta}{x}}\NSE{\Model}_{\theta}(\dee x)\leq \alpha$ for all $\theta\in \Theta$.
  Let $\pd{\Pi}$ be the push-down of $\Pi$.
  As $\Theta$ is compact, $\pd{\Pi}$ defines a probability measure on $\Theta$.
  By \cref{pdclose}, we have $\NSE{W}(\Pi, \NSE{\pd{\Pi}})\approx 0$.
  As $\closure{U}$ is compact, by \cref{gencompact}, we know that $\pd{\Pi}\in \NSE{\closure{U}}\subseteq \oMain$.  
  By \cref{phipush}, we have
  \[
  \int \NSE{\frset{\Pi}{\theta}{x}}\NSE{\Model}_{\theta}(\dee x)\approx \int \NSE{\frset{\NSE{\pd{\Pi}}}{\theta}{x}}\NSE{\Model}_{\theta}(\dee x),
  \]
  for all $\theta\in \Theta$.
  As $\Theta\subseteq T$, we know that $\int \NSE{\frset{\NSE{\pd{\Pi}}}{\theta}{x}}\NSE{\Model}_{\theta}(\dee x)\leq \alpha$ for all $\theta\in \Theta$.
  By the transfer principle, we have $\int \frset{\pd{\Pi}}{\theta}{x}\Model_{\theta}(\dee x)\leq \alpha$ for all $\theta\in \Theta$, completing the proof.
\end{proof}

\section{Sufficient conditions yielding exact matching priors}
\label{secmorenote}

\subsection{Overview}

We return to the problem discussed in Section \ref{SecApplHead}: how to prove the existence of  matching priors for families of credible regions that approximate standard families. We begin by stating an analogue of \cref{ThmMatchingPriorSimpleCred}  for highest-posterior density regions (\cref{ThmMatchingPriorSimple} in \cref{secmpcbhpd}). We then prove a similar results under weaker continuity requirements (\cref{sec:weakerconditions}). Finally, our results on highest-posterior density regions require the existence of something that we call a system of prior perturbation; we give a very general construction of such a system in \cref{SubsecExtExist}.

\subsection{Extension to highest-posterior density regions}\label{secmpcbhpd}

We begin by defining highest-posterior density (HPD) regions formally and introducing additional assumptions necessary to establish an analogue of \cref{ThmMatchingPriorSimpleCred} for HPD regions. Let $\PD{\Theta} \subseteq \Priors$ denote the set of absolutely continuous distributions.

\begin{definition} [Highest posterior density regions] \label{DefUsualHPD} 
  Fix a level $1-\alpha \in (0,1)$
  and absolutely continuous distribution $\mu \in \PD{\Theta}$ with density $\rho$.
  Let $R(d) = \set{ \theta \in \Theta \st \rho(\theta) \geq d }$ be the region with density at least $d$
  and let
  \[ \label{hpdLdefn}
  \hpdL{\mu} \defas \sup \set[\big]{ d \ge 0 \st \mu\parens[\big]{R(d)} \geq 1- \alpha }.
  \]
  The \defn{$1-\alpha$ $\mu$-highest posterior density (HPD) region} is the subset $\smash{R(\hpdL{\mu})}$ of $\Theta$ with characteristic function $\hpdreg{\mu}{\alpha}{}$.
  A family $\CRegions{\tau}{\pi}{F}$ of $1-\alpha$ credible regions
  is a \defn{family of $1-\alpha$ HPD regions} if,
  for all $\pi \in F$ and
  $\marg{\pi}$-almost all $x \in X$,
  the posterior $\Post{x}{\pi}$ is an element of $\PD{\Theta}$
  and, for almost all $u \in [0,1]$,
  $\charfunc{\tau_{\pi}(x,u)} = \hpdreg{\Post{x}{\pi}}{\alpha}{}$, i.e., 
  $\tau_{\pi}(x,u)$ is the $1-\alpha$ $\Post{x}{\pi}$-HPD region.
\end{definition}

Note that $\hpdL{\mu}$ is not affected by changes to the density $\rho$ on a Lebesgue measure zero set. On the other hand,
the characteristic function $\hpdreg{\mu}{\alpha}{}$ 
\emph{is} affected by changes to the density $\rho$ on a Lebesgue measure zero set.
In order to specify well-defined frequentist
properties of HPD regions, it is necessary to assume that we have fixed a particular posterior density for every combination of prior and sample. 
For many typical problems, there is a canonical choice. 
For example, there is often a unique continuous version of the density.

One obstruction to studying HPD regions is that most priors do not yield absolutely continuous posteriors, and thus no HPD region can be defined.
To sidestep this issue, we allow the statistician to specify, for every prior $\pi$ and tolerance $\gamma >0$, 
an absolutely continuous measure $\pi_{\gamma}$ that is ``close'' to $\pi$.
In order to give the statistician as much flexibility as possible, we allow this ``perturbed'' prior to live on a larger parameter space in some supermodel.

We recall some notation and basic facts about norms. For $1 \leq p \leq \infty$, let $\norm{\cdot}_{p}$ denote the usual $L^{p}$ norm on $\Reals^n$ and $\Reals^n \to \Reals$ with respect to Lebesgue measure.  
We sometimes abuse notation and write $\norm{\cdot}_{p}$ to denote the metric induced by the $L^{p}$ norm. For $p=2$, we drop the subscript in the notation. %

We recall a fundamental result about the Wasserstein distance. Given metric spaces $(\Omega, d)$, $(\Omega', d')$,
recall that the \defn{Lipschitz norm} of a function $f \: \Omega \to \Omega'$ is
  \[
  \lipnorm{f} \defas \sup_{x \neq y} \frac{ d'(f(x),f(y)) }{ d(x,y) }.
  \]
For $C > 0$, let $\Lip{\Omega}{C} \defas \set{ f \: \Omega \to \Reals \st \lipnorm{f} \leq C }$ 
denote the set of real-valued $C$-Lipschitz functions,
  and let $\Lip{\Omega}{} \defas \bigcup_{C > 0} \Lip{\Omega}{C}$ denote the set of all real-valued Lipschitz functions.
Finally, let $(\Omega, \BorelSets \Omega, d)$ be a metric Borel measurable space,
  and let $\mu,\nu \in \PM{\Omega}$. If $\mu$ and $\nu$ have bounded support, then a classical result by Kantorovich and Rubinstein shows that
\[
  W_{1}(\mu,\nu)=\sup \set[\Big]{ \abs[\Big]{ \int f\dee \mu-\int f \dee \nu}  \st f\in \Lip{Y}{1} }.
\]
For this paper, the metric underlying $W_1$ is the usual Euclidean metric on $\Reals^{d}$.

\begin{definition}\label{peturbsys}
Let $\SModel$ be a submodel of $\SModelfullalt$.
A \defn{system of prior perturbations on $\Theta$ in $\SModelfullalt$}
assigns to every prior distribution $\pi \in \Priors$ and real $\gamma > 0$
a  \defn{perturbed} prior $\pt{\pi}{\gamma} \in \AltPriors$
such that $W_1(\pi,\pt{\pi}{\gamma}) < \gamma$. 
Moreover, there is assumed to be some function $D : (0,\infty) \to (0,\infty)$ such that,
for all pairs of priors $\mu,\nu \in \Priors$ and reals $\gamma > 0$,
we have 
\[ \label{IneqPerturbSysWass}
W_1(\pt{\mu}{\gamma},\pt{\nu}{\gamma}) \le D(\gamma)\, W_1(\mu,\nu).
\]
\end{definition}

A system of prior perturbations exists trivially: we may simply take $\fTheta = \Theta$ and $\Model'=\Model$ and then $\pt{\pi}{\gamma} = \pi$ for all $\pi$ and $\gamma$. 
This trivial system suffices for approximating credible balls.
In order to approximate HPD regions, however, we require the following assumption to hold of the system of prior perturbations:

\begin{assumption}\label{assumptionPerturb}
For every prior $\pi \in \Priors$ and real $\gamma > 0$, the perturbed prior $\pt{\pi}{\gamma}$ is absolutely continuous with Lipschitz density function $p_{\pi,\gamma}$. Moreover, there is assumed to be some function $D' : (0,\infty) \to (0,\infty)$ such that,
for all pairs of priors $\mu,\nu \in \Priors$ and reals $\gamma > 0$,
we have 
\[ \label{IneqPerturbSysSup}
\Linfnorm{}{p_{\mu,\gamma} - p_{\nu,\gamma} } \leq D'(\gamma) W_{1}(\mu,\nu)
\]
and
\[\label{Perturbupbound}
\Linfnorm{}{p_{\mu,\gamma}} \leq D'(\gamma).
\] 
\end{assumption}

We observe that, to find a system of perturbed priors that satisfies 
\cref{assumptionPerturb}, 
it is only necessary to construct a system of perturbed priors for $\gamma$ in some small interval $(0, \gamma_{0}]$; we can then take $\pi_{\gamma} = \pi_{\gamma_{0}}$ for all $\gamma > \gamma_{0}$.

For the case of absolute continuity with respect to Lebesgue measure, 
\emph{every} model satisfying \cref{assumptionde} with a compact parameter space $\Theta \subseteq \Reals^{d}$ can be embedded in a supermodel with a compact parameter space $\fTheta \supseteq \Theta$, $\fTheta \subseteq \Reals^{d}$, that both satisfies \cref{assumptionde}
and yields a system of prior perturbations satisfying \cref{assumptionPerturb}. See \cref{SubsecExtExist} for details, including an explicit construction.

 We now state a generalization of our main application that holds also for HPD regions:

\begin{theorem}\label{ThmMatchingPriorSimple}
Fix $\alpha \in (0,1)$ and 
let $B \in \set{\tBall,\tHPD}$ be the family of credible regions to be extended,
where $B=\tBall$ specifies credible balls (\cref{DefUsualInterval}) and $B=\tHPD$ specifies HPD regions (\cref{DefUsualHPD}).
  Let $\SModel$ be a submodel of a statistical model $\SModelfullalt$ that satisfies \cref{assumptionde},
  and suppose $\Theta, \fTheta$ are compact subsets of $\Reals^{d}$.
  Fix any system of prior perturbations and, in the case $B=\tHPD$, assume the system meets \cref{assumptionPerturb}.
  Assume also that we fix versions of posterior distributions and posterior densities, when the latter exist.
  Let $\epsilon \in (0,\alpha)$.
  There exists
  a family of $1-\alpha$ credible regions $\tau = \CRegions{\tau}{\pi}{\Priors}$,
  real $\gamma \in (0, \epsilon)$, 
  and prior $\pi_0 \in \Priors$
  such that 
  \begin{enumerate}
  \item $\pi_0$ is a matching prior for $\tau$ in $\SModel$;
  \item In the case of $B=\tBall$, for all $\pi \in \Priors$ and $x \in X$,
    the support of $\tau_{\pi}$ at $x$
    is contained in the $\epsilon$-fattening of $\{\theta\in \Reals^{d}: \norm{\theta-M(\Post{x}{\pt{\pi}{\gamma}})} \leq \cbLv{\Post{x}{\pt{\pi}{\gamma}}}{\alpha-\epsilon}\}$
    and
    the support of $\credset{I}{\Post{x}{\pt{\pi}{\gamma}}}{\alpha}{}$
    is contained in
    the $\epsilon$-fattening of the support of $\tau_{\pi}$ at $x$; and
    \item In the case of $B=\tHPD$, for all $\pi \in \Priors$ and $x \in X$, let $\rho$ denote the density function for $\Post{x}{\pt{\pi}{\gamma}}$.
    Let $\Theta_{\tau}, \Theta_{\alpha}$ be the supports of $\tau_{\pi}$ and $\hpdreg{\Post{x}{\pt{\pi}{\gamma}}}{\alpha}{}$, respectively.
    Then, we have 
    \[\label{HPDContainment}
    \essinf_{\theta \in \Theta_{\tau}} \rho(\theta) \geq  \essinf_{\theta \in \Theta_{\alpha-\epsilon}} \rho(\theta)   - \epsilon \lipnorm{\rho}.
    \]
    
  \end{enumerate}
\end{theorem}

Part (2) relates the \emph{support} of the credible ball regions, 
while Part (3) relates the \emph{densities} of the points in the HPD regions.
The first of these is easy to interpret: 
the supports of our approximate credible ball regions are almost the same as the usual credible balls in a fairly strong sense.
The density relationship in Part (3) is a natural analogue for HPD regions:
HPD regions are defined by choosing those points in $\Theta$ that have the highest posterior density,
and so we compare the threshold below which points are excluded.
Note that the second containment relationship in Part (2) does in fact hold for HPD regions for many textbook scenarios (though perhaps not for $\pi_0$). 
In particular, if the posterior density is reasonably far from being flat in high-density regions, 
Inequality~\eqref{HPDContainment} immediately implies a containment relationship similar to the second one in Part (2).
  We present further discussion on roadblocks to finding matching priors for HPD regions in \cref{secrelation}.

\subsection{Weaker conditions for approximate credible balls and highest-posterior density regions}
\label{sec:weakerconditions}

\cref{ThmMatchingPriorSimple}
requires the model densities to be uniformly Lipschitz and the log densities to be bounded (\cref{assumptionde}). 
Here we show that these same results hold under the assumption that the 
posterior map $(x,\pi) \to \Post{x}{\pi}$ is sufficiently continuous;
we again elide dependencies on the statistical model $\SModel$ (and supermodels).
Recall that $\Theta$ is a compact subset of $\Reals^{d}$ of positive Lebesgue measure and $\Theta \subseteq \fTheta$.
Assume $\fTheta \subseteq \Reals^d$ is also compact.

\begin{assumption}[Lipschitz posterior map] \label{AssumptionModelSmoothness}
  The map
  \[
  (x, \pi) \mapsto \Post{x}{\pi} \: (X \times \AltPriors, \dX \otimes W_{1} ) \to (\AltPriors, W_{1})
  \]
  is Lipschitz.
\end{assumption}

We now show that \cref{assumptionde} implies \cref{AssumptionModelSmoothness}.

\begin{lemma}\label{implication4}
  Suppose $\Theta$ is compact.  Then \cref{assumptionde} implies \cref{AssumptionModelSmoothness}. 
\end{lemma} 

\begin{proof}
See \cref{AppConsAssump1}
\end{proof}

In order to approximate HPD regions, 
we require additional regularity of the posterior map on the set of absolutely continuous priors, $\PD{\Theta}$.
For $\mu,\nu \in \PD{\Theta}$, write $\abver_{\mu}$ for an arbitrary version of its probability density, 
and let 
  \[
  \dsup(\mu,\nu) = \Linfnorm{}{ \abver_{\mu} - \abver_{\nu} }
  \]
be the metric on absolutely continuous probability distributions corresponding to the $L^\infty$-norm metric on their underlying density functions.

\newcommand{\BR}{E}
\begin{assumption}[$L^\infty$-norm continuity] \label{AssumptionModelSupNorm} 
 For every subset $\BR \subseteq \PD{\fTheta}$ for which 
 \[\label{eqesssup}
 \sup_{\pi \in \BR} \, 
     \Linfnorm{}{ \abver_{\pi} } 
    < \infty,
 \]
  the map
  \[ \label{esssupconclusion}
  (x, \pi) \mapsto \Post{x}{\pi} \: (X \times \BR, \dX \otimes \dsup ) \to (\PD{\fTheta}, \dsup)
  \]
  from $X \times \BR$ to $\PD{\fTheta}$ is Lipschitz continuous.  Moreover, for $\pi\in \BR$, if $\abver_{\pi}$ is Lipschitz then $p_{\pi,x}$ is Lipschitz, where $p_{\pi,x}$ is the density function of $\Post{x}{\pi}$. 
\end{assumption}

\begin{lemma}\label{implication5}
Suppose $\Theta$ is compact. 
Then \cref{assumptionde} implies \cref{AssumptionModelSupNorm}.
\end{lemma}

\begin{proof}
See \cref{AppConsAssump1}.
\end{proof}

In \cref{simplexist} and \cref{secmpcbhpd}, we claimed existence of
matching priors for families that approximate credible balls and HPD regions under \cref{assumptionde}.
In fact, we can establish these result under the weaker hypotheses, \cref{AssumptionModelSmoothness} and \cref{AssumptionModelSupNorm}.

\begin{theorem}\label{ThmMatchingPriorsExist}
  \cref{ThmMatchingPriorSimple} holds with \cref{AssumptionModelSmoothness} in place of \cref{assumptionde} for the case $B=\tBall$
  and \cref{AssumptionModelSmoothness} and \cref{AssumptionModelSupNorm} in place of \cref{assumptionde} for the case $B = \tHPD$.
\end{theorem}

Clearly \cref{ThmMatchingPriorSimple} follows immediately from \cref{ThmMatchingPriorsExist}, \cref{implication4}, and \cref{implication5}.
We prove \cref{ThmMatchingPriorsExist} in \cref{SecSimpleExamples}.
Then \cref{ThmMatchingPriorSimpleCred} is seen to follow immediately from \cref{ThmMatchingPriorSimple} when one adopts the trivial system of perturbed priors whereby $\pt{\pi}{\gamma} = \pi$ for all priors $\pi$ and $\gamma > 0$.

\subsection{Existence of extensions and systems of perturbations} \label{SubsecExtExist}

In this section, we show that any model $\SModel$ satisfying \cref{assumptionde} 
can be extended to a supermodel $\SModelfullalt$ satisfying \cref{assumptionde} and admitting a 
system of prior perturbations satisfying \cref{assumptionPerturb}. 
The proofs in this section are constructive in the sense that they give explicit formulas that describe the supermodel (see  \cref{EqExplicitSupermodel}) and system of prior perturbations (see \cref{EqExplicitFamilyPerturbations}).

\begin{lemma}\label{LemmaModelExtensionsExist} 
Let $\SModel$ satisfy \cref{assumptionde}, 
and let $\fTheta \supseteq \Theta$ be any compact set. 
Then there exists a supermodel $\SModelfullalt$ of $\SModel$ that satisfies \cref{assumptionde}.
\end{lemma}

\begin{proof}
See \cref{AppSupExist}. 
\end{proof}

Next, we show that it is possible to construct a system of prior  perturbations. We require some notation related to probability. For a distribution $\mu$, we write $X \sim \mu$ if the random variable $X$ has distribution $\mu$. For two distributions $\mu, \nu$ on the same vector space, we write $\mu \ast \nu$ for the usual convolution. That is, if $X \sim \mu$ and $Y \sim \nu$ are independent, then $X+Y \sim \mu \ast \nu$.

\begin{lemma} \label{LemmaPriorPertsExist}
Let $\SModel$ satisfy \cref{assumptionde}, with $\Theta$ compact. Then there exists a supermodel $\SModelfullalt$ of $\SModel$ 
that satisfies \cref{assumptionde} and admits a system of prior perturbations that satisfy \cref{assumptionPerturb}.
\end{lemma}

\begin{proof}
See \cref{AppSupExist}. 
\end{proof}

\section{Families of coverage-inducing priors} \label{SecSimpleExamples}

\subsection{Overview}

Our main result, \cref{ThmMainExistenceResult}, establishes the existence of matching priors for certain families of credible regions.
Among other assumptions in this theorem, we require that these families of credible regions satisfy certain continuity conditions that are not satisfied by the usual credible ball and HPD regions. 
In this section, we modify the usual credible balls and HPD regions and show that these modifications yield families of credible regions that satisfy the main continuity assumptions of \cref{ThmMainExistenceResult}.
We also show that our modified credible balls and modified HPD regions are very similar to the usual credible balls and HPD regions in a formal sense.
We conclude by proving \cref{ThmMatchingPriorsExist} (and then proving \cref{ThmMatchingPriorSimple} as a corollary). The explicit construction of the credible ball is included in \cref{SecConstructions} from the main paper, but we repeat much of this material here for ease of reference.

We give a short guide to our approach in the following subsections.
We begin with a general construction and theorem in \cref{subsecpert}.
The general construction is a recipe for perturbing a generic family of credible regions whose rejection probability functions are uniformly Lipschitz.
The general result, \cref{PropContinuity1}, says that the resulting family of credible regions satisfies \cref{assumptionjc} as long as the original family's rejection probability functions, viewed as functions of the data and prior, are sufficiently continuous.
In \cref{secfamily}, we construct two families of credible regions that we view as \scare{relaxed} versions of the usual credible ball and HPD regions and show that these families have uniformly Lipschitz rejection probability functions, satisfying the first condition introduced in \cref{subsecpert}.
In \cref{subsecpbpd}, we show that these relaxed families are also continuous functions of the associated data and prior, thus satisfying the second condition introduced in \cref{subsecpert}.

The results up to this point are sufficient to show that our new families of credible regions have matching priors.
In \cref{SubsecSupport}, we show that our relaxed credible ball and HPD regions are also similar to the usual credible ball and HPD regions.
These results allow us to establish our main result, \cref{ThmMatchingPriorsExist}, in \cref{SubsecProofMainThm}; \cref{ThmMatchingPriorSimpleCred} is an immediate corollary.
 
For the remainder of \cref{SecSimpleExamples}, we fix a model $\SModel$ and a supermodel $\SModelfullalt$ 
that both satisfy \cref{AssumptionModelSmoothness},
and let $C$ denote the associated Lipschitz constant from that assumption. 
We fix a system $(\pi_\gamma)$ of perturbed priors (\cref{peturbsys}),
and, abusing notation slightly, write $\Post{x}{\pi_\gamma}$ for $(P')^{x}_{\pi_\gamma}$.
Up to this point, we have been working with rejection probability functions $\varphi$.
In this section, we use acceptance probability functions $\psi = 1- \varphi$, 
as they lead to more natural definitions and proofs.

\subsection{Perturbed credible regions}\label{subsecpert}

In this section, we study credible regions with uniformly Lipschitz acceptance probability functions. We show how to modify such credible regions 
in order to obtain families of credible regions that satisfy \cref{assumptionjc} under some regularity conditions.

The basic idea behind our construction is to smoothly average the usual acceptance probability function across some small range of levels (see \cref{tfrsetdefn}), then perform a small post-averaging correction. 
In particular, for $\eta \in (0,1)$, let $\pb{\eta} \: \NNReals \to \Reals$ be the density of some probability distribution on $\NNReals$ such that the support of $\pb{\eta}$ is $[\eta,1]$ and $\pb{\eta}$ is continuous. For $c > 0$, let $\pb{\eta}^{(c)}(x) = c^{-1} \pb{\eta}(\frac{x}{c})$ be the probability density after the change of variable transformation $x \mapsto c x$. 
We use this density to define a random inflation to the level of an acceptance probability function.

We now introduce a base family of credible regions. 
Fix $\gamma > 0$ and let $\PPriors = \{{\pt{\pi}{\gamma}} : \pi \in \Priors\}$ be the set of perturbed priors. 
Fix $\beta > 0$.
For every $\alpha \in (0,1)$,
let $\rtau^{\alpha} = \CRegionsv{\rtau^\alpha}{\pi}{\pi}{\Priors}$ be a family of $1-\alpha$ credible regions
with acceptance probability function $\bapf[\alpha]{\beta}{}{}{}$
such that $\bapf[\alpha]{\beta}{\pt{\pi}{\gamma}}{\theta}{x}$ is $\beta$-Lipschitz in $\theta \in \Theta$ for every $\pi \in \Priors$ and $x \in X$. Let $\rtau = (\rtau^\alpha)_{\alpha \in (0,1)}$.

For $\alpha,\eta \in (0,1)$ and $\delta \in (0,\alpha)$, 
the \defn{$(\gamma,\eta,\delta)$-perturbation of $\bapf[\alpha]{\beta}{}{}{}$} is the function 
$(\theta,x,\pi) \mapsto \papf{\pi}{\theta}{x} : \Theta \times X \times \Priors \to [0,1]$ given by
\[ \label{tfrsetdefn}
\papf{\pi}{\theta}{x}  
= \int_{[\delta \eta,\delta]} \bapf[\alpha-z]{\beta}{\pt{\pi}{\gamma}}{\theta}{x} \,\pb{\eta}^{(\delta)}(z) \dee z.
\]

Fix $\alpha,\eta,\delta$ as above.
As we show, 
while $\papf{\cdot}{\cdot}{\cdot}$ is not itself an acceptance probability function for family of $1-\alpha$ credible regions,
it can be modified to be one. The following result is immediate from inspection.

\begin{lemma}\label{perlipschitz}
The function $\papf{\pi}{\cdot}{x}$ is $\beta$-Lipschitz for
all $x \in X$ and $\pi \in \Priors$.
\end{lemma}

The next result shows that, for sufficiently small $\gamma$, 
the function $\papf{\pi}{\cdot}{x}$, when viewed as an acceptance probability function,
has more than $1-\alpha$ credibility.
Recall that $C$ is the Lipschitz constant in \cref{AssumptionModelSmoothness}. 

\begin{lemma}\label{techlemma}
Assume $\gamma$ is small enough to satisfy 
\[\label{gammasmall}
F' \equiv \delta \eta - \beta C \gamma > 0.
\]
Let $x \in X$ and $\pi \in \Priors$. Then 
$
\Post{x}{\pi}(\papf{\pi}{\cdot}{x}) \geq 1 - \alpha + F' > 1- \alpha.
$
\end{lemma}
\begin{proof}
By \cref{peturbsys} and \cref{AssumptionModelSmoothness}, $W_{1}(\Post{x}{\pi},\Post{x}{\pt{\pi}{\gamma}}) \leq C \gamma$.

Then, for $z \in [\delta \eta,\delta]$,
\[
\Post{x}{\pi}(\bapf[\alpha-z]{\beta}{\pt{\pi}{\gamma}}{\cdot}{x}) 
&\ge \Post{x}{\pt{\pi}{\gamma}}(\bapf[\alpha-z]{\beta}{\pt{\pi}{\gamma}}{\cdot}{x}) - \beta C \gamma \\
&= 1 - \alpha + z  - \beta C \gamma \\
&\ge 1 - \alpha + \delta \eta  - \beta C \gamma = 1- \alpha + F'.
\]
\end{proof}

For $x \in X$ and $\pi \in \Priors$,
define
\[
\corr{\pi}{x} (r) = \Post{x}{\pi}(\max(0, \papf{\pi}{\cdot}{x} - r)), \qquad r \in [0,1],
\]
and
\[
R(x,\pi) = \sup \set{ r \in [0,1] \st \corr{\pi}{x}(r) \geq  1 - \alpha }.
\]
The next result shows that $\corr{\pi}{x}$ and $R(x,\pi)$ take finite values.
\begin{lemma} \label{minorlem}
Assume \cref{gammasmall} holds.
Let $x \in X$ and $\pi \in \Priors$.
Then 
$\corr{\pi}{x} (0) \ge 1 - \alpha + F' > 1 - \alpha$ and 
\[\label{Rbound}
R(x,\pi) \in [0,\alpha].
\]
\end{lemma}

\begin{proof}
Because $\papf{\pi}{\cdot}{x} \le 1$, we have $R(x,\pi) \le \alpha$.
In light of \cref{techlemma}, we know that 
$\corr{\pi}{x} (0)
=\Post{x}{\pi}(\papf{\pi}{\cdot}{x}) \ge  1- \alpha + F' > 1 - \alpha$.
Thus $R(x,\pi) \ge 0$.
\end{proof}

We are now in a position to give the main definition of this section.

\begin{definition} [Perturbed credible regions] \label{DefContFamilyCredible}
Fix $\gamma,\beta > 0$
and let $\rtau,\alpha,\eta,\delta$ be defined as above, and 
assume \cref{gammasmall} holds.
A \defn{family of $(\gamma,\delta,\eta)$-perturbed $1-\alpha$ credible regions based on $\rtau$} 
is one whose acceptance probability function $\gfinalset{}{}$ satisfies
\[ \label{credform}
\gfinalset{\pi}{x} = \max(0, \papf{\pi}{\cdot}{x} - R(x,\pi))
\]
for all $x \in X$ and $\pi \in \Priors$. 
\end{definition}

It is immediate that
\[\label{StartingHighLemma}
\gfinalset{\pi}{x}
\leq
\papf{\pi}{\cdot}{x}.
\]

\cref{DefContFamilyCredible} specifies a family of credible regions: 

\begin{lemma} \label{LemmaTheComplicatedThingIsACredibleSet}
Assume \cref{gammasmall} holds.
A \defn{family of $(\gamma,\delta,\eta)$-perturbed $1-\alpha$ credible regions based on $\rtau$} 
is a family of $1-\alpha$ credible regions.
\end{lemma}

\begin{proof}
See \cref{AppSec51Proofs}.
\end{proof}

Let $\CFuncs(Y,Y')$ denote the collection of continuous functions from $Y \subseteq \Reals^{d}$ to $Y' \subseteq \Reals^{d'}$.
We give this set its usual sup-norm metric structure: 
\[
\supnorm{f - g} \defas \sup_{x \in Y} \, \norm{f(x) - g(x)}.
\]

We now state the main result for this section, showing that perturbed families of credible regions meet \cref{assumptionjc} provided the perturbation is continuous.

\begin{proposition}\label{PropContinuity1}
Assume
\cref{gammasmall} holds and
$(x,\pi) \mapsto \papf{\pi}{\cdot}{x}$
is a continuous map from 
$(X\times \Priors, \dX \otimes W_1)$ 
to $(\CFuncs(\Theta,[0,1]), \supnorm{\cdot})$.
Then 
\[
(x,\pi) \mapsto \gfinalset{\pi}{x} \: (X\times\Priors, \dX \otimes W_1) \to  (\CFuncs(\Theta, [0,1]), \supnorm{\cdot})
\]
is continuous and, in particular, \cref{assumptionjc} holds for the family 
$\tau = \CRegions{\tau}{\pi}{\Priors}$ of
$1-\alpha$ credible regions with acceptance probability function
$\gfinalset[\cdot]{\cdot}{\cdot}$.
\end{proposition}

\begin{proof}
See \cref{AppSec51Proofs}.
\end{proof}

The following corollary is an immediate consequence of \cref{PropContinuity1} and our main theorem \cref{ThmMainExistenceResult}:

\begin{corollary}\label{maincor}
Under the assumptions of \cref{PropContinuity1},
there exists a matching prior for $\tau$.
\end{corollary}

\subsection{Construction of uniformly Lipschitz credible regions}\label{secfamily}  

In the previous section, we showed how families of credible regions with uniformly Lipschitz acceptance probability functions can be modified slightly to obtain families of credible regions satisfying \cref{assumptionjc}, provided the perturbation we perform is continuous.

In this section, we construct perturbed variants of the usual credible balls and high posterior density (HPD) regions.  In Section \ref{subsecpbpd}, we show that these perturbed families of credible regions meet the Lipschitz hypotheses of Section \ref{subsecpert}, yielding perturbed families of credible regions with matching priors.  In Section \ref{SubsecSupport}, 
\cref{ballcontain}, \cref{relaxedhpddensprop}, \cref{basicsupportresult}, and \cref{ballsupportresult} show that these perturbed families are ``similar''  to the usual definitions of \cref{DefUsualInterval} and \cref{DefUsualHPD}. Some of the definitions are repeated from \cref{SecConstructions} for ease of reading.

We begin with relaxed notions of credible balls and HPD regions, whose acceptance probability functions are Lipschitz continuous. These serve as building blocks in our construction of ``perturbed'' versions of credible balls and HPD regions that
satisfy our main theorem's continuity hypothesis:

Recalling the relaxed versions of credible balls given in \cref{DefSimpleCredInt}.
For our purposes here, we note that we can take $\Theta$ to be $\fTheta$ in that definition.

In the presence of posteriors with atoms, $1-\alpha$ credible balls are not always $1-\alpha$ credible regions.
In contrast, relaxed credible balls always have the right level:
\begin{lemma} \label{rcbisacredset}
  Fix a level $1-\alpha \in (0,1)$,
  \defn{slope} $\beta \in (0,\infty)$,
  and distribution $\mu\in \fPriors$ with mean $M(\mu)$.
  Then $\cbLR{\mu}$ is well-defined and
  a $1-\alpha$ relaxed $\mu$-credible ball with slope $\beta$ has $1-\alpha$ $\mu$-credibility. Moreover, $\ffB{r}{}$ is $\beta$-Lipschitz and so the acceptance probability function is as well.
\end{lemma}

\begin{proof}
  Let $H \: (0,\infty) \to [0,1]$ be the map $r \mapsto  \int_{\fTheta} \ffB{r}{} \,\dee\mu$.
  Assuming $0 \leq r \leq s < \infty$,
  \begin{align*}
    \abs{H(r) - H(s)}
    &\leq \int_{\fTheta} \abs[\Big]{ (r - \beta \norm{ \theta - M(\mu) }) - (s - \beta \norm{ \theta - M(\mu) }) } \mu(\dee \theta) \\
    &= \int_{\fTheta} \abs{r - s} \,\mu(\dee \theta)
    = \abs{r-s}.
  \end{align*}
  Thus, $H$ is 1-Lipschitz, hence continuous. 
  Clearly $\lim_{r \to 0} H(r) = 0$ and $\lim_{r \to \infty} H(r) = 1$ and so there exists a least $r > 0$ such that $H(r) = 1-\alpha \in (0,1)$. 
  Hence, $\cbLR{\mu}$ is well-defined.
  Moreover,
  \begin{align*}
    \int_{\fTheta} \rcb{\mu}{\alpha}{\beta} (\theta) \,\mu(\dee \theta)
    &= H(\cbLR{\mu}) \\
    &= \inf \set{ H(r) \st H(r) \geq 1 - \alpha }
    = 1 - \alpha.
  \end{align*}
Finally, it is straightforward to verify that \cref{EqLevelsBlah} is $\beta$-Lipschitz.
\end{proof}

We now define relaxed HPD regions. Note that the following regions are only defined for absolutely continuous distributions with Lipschitz density function. 

\begin{definition} [Relaxed highest posterior density regions] \label{DefSimpleHPDCredInt}
Fix a level $1-\alpha \in (0,1)$, \defn{slope} $\beta \in (0,\infty)$, and an absolutely continuous distribution $\mu \in \PD{\fTheta}$ with density $\rho$ and $\lipnorm{\rho} < \infty$. Let $\hat{\beta} = \frac{\beta}{\max(1, \lipnorm{\rho})}$. For $d \in [0,\infty)$, let
\begin{equation}\label{EqLevelsHPDBlah}
\ffHPD{d}{\theta} \defas
\begin{cases}
1, & \rho(\theta) \geq d, \\
\hat{\beta} (\rho(\theta) - (d- \hat{\beta}^{-1})), & \rho(\theta) \in (d - \hat{\beta}^{-1}, d) \\
0, & \rho(\theta) \leq d - \hat{\beta}^{-1}.
\end{cases} 
\end{equation}
and let
\[ 
\hpdLR{\mu}
= \sup \set[\Big]{ d \geq 0 \st \int_{\fTheta} \ffHPD{d}{} \, \dee \mu \geq 1 - \alpha }. %
\]
Then the \defn{relaxed $\mu$-highest posterior density (HPD) region with slope $\beta$}
is any random set with \mpf\
$\rhpd{\mu}{\alpha}{\beta} \defas \ffHPD{\hpdLR{\mu}}{}$.
\end{definition}

The proof of the following lemma is similar to the proof of \cref{rcbisacredset}. 

\begin{lemma}\label{rhpdisacredset}
Fix a level $1-\alpha \in (0,1)$, slope $\beta\in (0,\infty)$ and distribution $\mu\in \PD{\fTheta}$ with Lipschitz density function.
Then $\hpdLR{\mu}$ is well-defined and
a $1-\alpha$ relaxed $\mu$-HPD region with slope $\beta$ 
has $1-\alpha$ $\mu$-credibility.
Furthermore, 
$\ffHPD{d}{}$ (and thus the acceptance probability function) is $\beta$-Lipschitz.
\end{lemma}
\begin{proof}
 Let $H \: (0,\infty) \to [0,1]$ be the map $d \mapsto  \int_{\fTheta} \ffHPD{d}{} \,\dee\mu$.
  Consider $0 \leq d_0 \leq d_1 < \infty$ with $d_{0} > d_{1} - \hat{\beta}$, let $A_1=\{\theta\in \fTheta: \rho(\theta)\in [d_0,d_1)\}$, $A_2=\{\theta\in \fTheta: \rho(\theta)\in (d_1-\hat{\beta}^{-1},d_0)\}$ and $A_3=\{\theta\in \fTheta: \rho(\theta)\in (d_0-\hat{\beta}^{-1},d_1-\hat{\beta}^{-1}]\}$. We have
  \begin{align*}
    H(d_0) - H(d_1)&=\int_{A_1}(1-(\rho(\theta)-(d_{1}-\hat{\beta}^{-1})))\mu(\dee\theta)\\
    &\qquad +\int_{A_2}(d_1-d_0)\mu(\dee \theta)+\int_{A_3}\rho(\theta)-(d_{0}-\hat{\beta}^{-1})\mu(\dee \theta)
  \end{align*}
  As $\mu$ is absolutely continuous, it is easy to see that $H(d_0)-H(d_1)\to 0$ as $d_1-d_0\to 0$, hence $H$ is continuous.
  Clearly $\lim_{d \to 0} H(d)=1$ and $\lim_{d \to \infty} H(d) = 0$ and so the supremum in \cref{DefSimpleHPDCredInt} is achieved.
  Hence, $\hpdLR{\mu}$ is well-defined.
  Moreover,
  \begin{align*}
    \int_{\fTheta} \rhpd{\mu}{\alpha}{\beta} (\theta) \,\mu(\dee \theta)= H(\cbLR{\mu})=1-\alpha.
  \end{align*}
Finally, it is straightforward to verify that $\ffHPD{d}{}$ as defined by \cref{EqLevelsHPDBlah} is $\hat{\beta} \lipnorm{\rho}$-Lipschitz and thus $\beta$-Lipschitz.
\end{proof}

\subsection{Perturbations of relaxed credible balls and highest-posterior density regions}\label{subsecpbpd}

In this section, we apply perturbation developed in \cref{subsecpert} to relaxed credible balls and relaxed HPD regions. 
We show that the maps sending a (data, prior) pair to the acceptance function of perturbed relaxed credible balls and perturbed relaxed HPD regions are continuous in the sense of \cref{assumptionjc}. 

Fix $\gamma > 0$ and let $\PPriors = \{{\pt{\pi}{\gamma}} : \pi \in \Priors\}$ be the set of perturbed priors. 
Fix $\beta > 0$.
For every $\alpha\in (0,1)$, let $\tau^{\alpha} = \CRegionsv{\tau^\alpha}{\pi}{\pi}{\PPriors}$ be a family of $1-\alpha$ credible regions (for perturbed priors)
with acceptance function $\rbcs{\Post{x}{\pt{\pi}{\gamma}}}{\alpha}{\beta}{B}$ for $B \in \set{\tBall,\tHPD}$, 
i.e., $\tau^{\alpha}$ is a family of $1-\alpha$ relaxed credible balls or HPD regions.
Let $\tau= \tau^{\alpha}_{\alpha \in (0,1)}$.
Recall that $\Post{x}{\pi}$ denotes some particular fixed version of the posterior distribution associated with $\pi \in \Priors$ and a sample $x \in X$. 
Moreover, for absolutely continuous posterior distributions, we fix some version of the posterior density.

By \cref{rcbisacredset} and \cref{rhpdisacredset}, 
for every $\alpha \in (0,1)$, 
$\rbcs{\Post{x}{\pt{\pi}{\gamma}}}{\alpha}{\beta}{B}$ is $\beta$-Lipschitz for $B \in \set{\tBall,\tHPD}$. 
We may therefore consider the perturbed version of this family, as defined in \cref{subsecpert}:
Let $\alpha,\eta \in (0,1)$ and $\delta \in (0,\alpha)$, and 
assume \cref{gammasmall} holds.
Write 
$\finalsetgen{B;}{\alpha}{\beta}{\gamma}{\delta}{\eta}{\pi}{x}{\cdot}$
for the acceptance probability function underlying the family of $(\gamma,\delta,\eta)$-perturbed $1-\alpha$ credible regions based on $\tau$.

In order to show that there exist matching priors for these new (perturbed) credible regions,
we establish the hypotheses of \cref{PropContinuity1}. 
We begin with a result that relates to the hypotheses for perturbed relaxed credible balls.

\begin{lemma} \label{LemmaSmoothDensitiesMeans}
Assume \cref{gammasmall} holds.
Then the function $(\mu,x) \mapsto M(\Post{x}{\pt{\mu}{\gamma}})$ is Lipschitz continuous.
\end{lemma}
\begin{proof}
Since the identity function is 1-Lipschitz, we have, by the definition of Wasserstein distance and then \cref{AssumptionModelSmoothness},
\begin{align*}
\norm{ M(\Post{x}{\pt{\mu}{\gamma}}) - M(\Post{y}{\pt{\nu}{\gamma}}) }_2 &\leq W_{1}(\Post{x}{\pt{\mu}{\gamma}},\Post{y}{\pt{\nu}{\gamma}}) \\
&\leq C(W_{1}(\pt{\mu}{\gamma},\pt{\nu}{\gamma}) + \dX(x,y) ).
\end{align*}
By \cref{peturbsys},
$W_{1}(\pt{\mu}{\gamma},\pt{\nu}{\gamma}) \le D(\gamma)\,W_{1}(\mu,\nu)$.
\end{proof}

For HPD regions, we rely on \cref{assumptionPerturb}.
Recall that, under \cref{assumptionPerturb}, for all $\pi \in \Priors$ and $\gamma > 0$, we have $\pt{\pi}{\gamma} \in \PD{\fTheta}$, and so $\Post{x}{\pt{\pi}{\gamma}} \in \PD{\fTheta}$ holds for $\marg{\pt{\pi}{\gamma}}$-almost-all $x$ due to the model being dominated.
Moreover, by \cref{assumptionPerturb} and \cref{AssumptionModelSupNorm}, $\Post{x}{\pt{\pi}{\gamma}}$ has a Lipschitz density function.

\begin{lemma} \label{LemmaSmoothDensities}
Suppose \cref{assumptionPerturb}, \cref{AssumptionModelSupNorm}, and \cref{gammasmall} hold.
There exists a real number $C''_\gamma > 0$ such that,
for any prior distributions $\mu, \nu \in \mathcal{M}(\Theta)$ and samples $x,y \in X$, 
with associated posterior distributions $\Post{x}{\pt{\mu}{\gamma}}$ and $\Post{y}{\pt{\nu}{\gamma}}$
admitting densities $p_{x,\pt{\mu}{\gamma}}$ and $p_{y,\pt{\nu}{\gamma}}$, %
\[ \label{postdensclose}
\supnorm{ p_{x,\pt{\mu}{\gamma}} - p_{y,\pt{\nu}{\gamma}} }
      \leq C''_{\gamma}\, (W_{1}(\mu,\nu) + \dX(x,y) ). 
\]
\end{lemma}
\begin{proof}
Let $E = \PPriors$ be the set of perturbed priors on $\fTheta$.
By \cref{Perturbupbound} in \cref{assumptionPerturb}, $E$ satisfies \cref{eqesssup} in \cref{AssumptionModelSupNorm}. 
The result then follows from \cref{IneqPerturbSysSup} in \cref{assumptionPerturb} 
and \cref{esssupconclusion} in \cref{AssumptionModelSupNorm}. 
\end{proof}

We now establish the requisite continuity of the perturbed credible regions. 

\begin{lemma} \label{LemmaContinuityTildeRAnnoying}
Let $B \in \set{\tBall,\tHPD}$.
For $B = \tHPD$, suppose \cref{assumptionPerturb} and \cref{AssumptionModelSupNorm} hold.
The map
$(x,\pi) \mapsto \candid{\pi}{x}$
is a continuous map from  $(X\times \Priors, \dX \otimes W_1)$ to $(\CFuncs(\Theta,[0,1]), \supnorm{\cdot})$.
\end{lemma}

\begin{proof}
See \cref{appsecdef53}. 
\end{proof}

We now may establish the main theorem about perturbed credible balls and HPD regions.

\begin{proposition}[Perturbed credible regions] \label{ThmContinuity1}
Fix a model $\SModel$ and a supermodel $\SModelfullalt$ 
that both satisfy \cref{AssumptionModelSmoothness}, and suppose $\Theta$ is compact.
Let $\alpha,\eta \in (0,1)$, $\delta \in (0,\alpha)$, $\gamma > 0$, and $\beta > 0$
satisfy \cref{gammasmall}. 
Pick $B \in \set{\tBall,\tHPD}$ and 
let $\tau = \CRegions{\tau}{\pi}{\Priors}$
be the family of $1-\alpha$ credible regions with acceptance probability function
$\finalset{\cdot}{\cdot}$.
For $B=H$, 
suppose further that \cref{assumptionPerturb} and \cref{AssumptionModelSupNorm} hold.
Then $\tau$ has a matching prior.
\end{proposition}
\begin{proof}
By hypothesis, 
\cref{gammasmall} holds.
By \cref{LemmaContinuityTildeRAnnoying,PropContinuity1},
the map
\[
(x,\pi) \mapsto \finalset{\pi}{x} \: (X\times\Priors, \dX \otimes W_1) \to  (\CFuncs(\Theta, [0,1]), \supnorm{\cdot})
\]
is continuous and, in particular, \cref{assumptionjc} holds. %
The result then follows from \cref{maincor}.
\end{proof}

\subsection{Similarity of relaxed and usual credible regions} \label{SubsecSupport}

In this section, we show that the new credible regions we have introduced are similar to their usual versions in a certain formal sense. All proofs are deferred to \cref{appsecdef54}.

We begin by relating the support of a perturbed family to that of the original family:

\begin{lemma}\label{basicsupportresultlemma} 
Assume \cref{gammasmall} holds. Let $x \in X$ and $\pi \in \Priors$.
Assume that the families of credible regions $\rtau^\alpha$, for $\alpha \in (0,1)$, are nested in the sense
that 
\[ \label{nestedass}
\bapf[\alpha]{\beta}{\pt{\pi}{\gamma}}{\cdot}{x}
\le 
\bapf[\alpha']{\beta}{\pt{\pi}{\gamma}}{\cdot}{x}
\]
for every $\alpha' \le \alpha$.
Then the support of $\gfinalset{\pi}{x} $
is contained in the 
support of $\bapf[\alpha-\epsilon]{\beta}{\pt{\pi}{\gamma}}{\theta}{x}$
for all $\epsilon \in [\delta,\alpha)$.
\end{lemma}
\begin{proof}
See \cref{appsecdef54}.
\end{proof}

Next, we relate the supports of the typical and relaxed credible balls: 

\begin{lemma}\label{ballcontain}
  Fix a level $1-\alpha \in (0,1)$,
  slope $\beta\in (0,\infty)$,
  and distribution $\mu\in \fPriors$.
  Let $E_0=\{\theta\in \Reals^{d}: \norm{\theta-M(\mu)} \leq \cbL{\mu}\}$ where $\cbL{\mu}$ is defined in \cref{cbLdefn}.
  Then the support of $\credball{\mu}{\alpha}{}$ is contained in the support of $\rcb{\mu}{\alpha}{\beta}$,
  and the support of
  $\rcb{\mu}{\alpha}{\beta}$
  is contained in the $\beta^{-1}$-fattening of $E_0$.
\end{lemma}

\begin{proof}
See \cref{appsecdef54}.
\end{proof}

\begin{lemma}\label{basicsupportresult} 
Assume \cref{gammasmall} holds.
Let $x \in X$ and $\pi \in \Priors$.
The support of $\finalsetgen{\tBall}{\alpha}{\beta}{\gamma}{\delta}{\eta}{\pi}{x}{\cdot}$
is contained in the $\beta^{-1}$-fattening of
\[
\{\theta\in \Reals^{d}: \norm{\theta-M(\Post{x}{\pt{\pi}{\gamma}})} \leq \cbLv{\Post{x}{\pt{\pi}{\gamma}}}{\alpha-\epsilon}\}
\]
for all $\epsilon \in [\delta,\alpha)$.
\end{lemma}
\begin{proof}
It is straightforward to verify that the relaxed credible balls are indeed nested as required to meet the hypotheses of \cref{basicsupportresultlemma}. The result then follows from \cref{ballcontain}.
\end{proof}

\begin{lemma}\label{ballsupportresult} 
Assume \cref{gammasmall} holds.
Let $x \in X$ and $\pi \in \Priors$.
The support of
$\credset{I}{\Post{x}{\pt{\pi}{\gamma}}}{\alpha}{}$
is contained in the $\beta^{-1}$-fattening of the support of
$\finalsetgen{\tBall}{\alpha}{\beta}{\gamma}{\delta}{\eta}{\pi}{x}{\cdot}$.
\end{lemma}

\begin{proof}
See \cref{appsecdef54}.
\end{proof}

\cref{basicsupportresult} and \cref{ballsupportresult} show that,
in a strong sense, our approximations to credible balls are quite close to the usual ones.

Although the support of the relaxed HPD region may not be contained in a slight fattening of the original HPD region, the smallest \emph{density} at some point in the support of the acceptance function of the relaxed HPD region is never much smaller than the density at every point in the support of the acceptance function of the original HPD region:

\newcommand{\onething}{\mathbf{1}}
\begin{lemma}\label{relaxedhpddensprop}
Fix an absolutely continuous distribution $\mu \in \PD{\Theta}$ with density $\rho$.
Let $A$ and $A_{\beta}$ denote the respective supports of the usual and relaxed HPD regions, 
$\hpdreg{\mu}{\alpha}{}$ and $\rhpd{\mu}{\alpha}{\beta}$.
Then
\[
\essinf_{\theta \in A} \rho(\theta) - \frac{\lipnorm{\rho}}{\beta} \leq \essinf_{\theta \in A_{\beta}} \rho(\theta) \leq \essinf_{\theta \in A} \rho(\theta) + \frac{\lipnorm{\rho}}{\beta}.
\]
\end{lemma}

\begin{proof}
See \cref{appsecdef54}.
\end{proof}

\subsection{Proof of the main theorem}\label{SubsecProofMainThm}

\begin{proof}[of \cref{ThmMatchingPriorsExist}] 
We begin by pointing out that there exist constants $\eta$, $\delta$, $\gamma$, $\beta >0$ that satisfy both the requirements in the statement of \cref{ThmContinuity1},
and the additional requirement that
\[ \label{IneqAddReqLate}
\max(\beta^{-1},\eta,\delta, \gamma) < \epsilon < \alpha.
\] To see this, we introduce an explicit parameterized family of constants
\[ \label{EqExpParam}
\eta(a) = a, \ \delta(a) = a, \ \gamma(a) = a^{4}, \ \beta(a) = a^{-1}, \quad a > 0.
\]
Then it is straightforward to see that there exists some constant $A_{0} = A_{0}(\epsilon, \alpha, C) > 0$
such that, for all $a \in (0,A_{0})$, the constants $(\eta(a), \delta(a), \gamma(a), \beta(a))$ satisfy the requirements of \cref{ThmContinuity1} and \cref{IneqAddReqLate}. 

Consider the statement of \cref{ThmMatchingPriorSimple}, with \cref{AssumptionModelSmoothness} in place of \cref{assumptionde} when $B=\tBall$, and 
\cref{AssumptionModelSmoothness} and \cref{AssumptionModelSupNorm} in place of \cref{assumptionde} when $B = \tHPD$.
Let $\tau = \CRegions{\tau}{\pi}{\Priors}$
be the family of $1-\alpha$ credible regions with acceptance probability function
$\finalset{\cdot}{\cdot}$.
By \cref{ThmContinuity1}, $\tau$ has a matching prior $\pi_0$,
establishing Part (1) of \cref{ThmMatchingPriorSimple}.

The first containment relationship in Part (2) follows from \cref{basicsupportresult} and the fact that $\delta \le  \epsilon < \alpha$ and $\beta^{-1} \le \epsilon$.
The second containment relationship in Part (2) follows from \cref{ballsupportresult} and the fact that $\beta^{-1} < \epsilon$. 

We now prove Part (3) of \cref{ThmMatchingPriorSimple}. 
Pick $\pi \in \Priors$ and $x \in X$.
Let $\Theta^{\beta}_{\alpha}$ be the support of $\rhpd{\Post{x}{\pi_{\gamma}}}{\alpha}{\beta}$ and let $\rho$ be the support of $\Post{x}{\pi_{\gamma}}$. By \cref{DefSimpleHPDCredInt}, $\Theta^{\beta}_{\alpha}\subseteq \Theta^{\beta}_{\alpha-\epsilon}$. As $\delta \eta<\epsilon$, by \cref{tfrsetdefn}, $\Theta_{\tau}\subseteq \Theta^{\beta}_{\alpha-\epsilon}$.
Thus, we have
\[
\essinf_{\theta\in \Theta_{\tau}} \rho(\theta)\geq \essinf_{\theta\in \Theta^{\beta}_{\alpha-\epsilon}} \rho(\theta).
\]
By \cref{relaxedhpddensprop}, we have
\[
\|\essinf_{\theta\in \Theta_{\alpha-\epsilon}}\rho(\theta)-\essinf_{\theta\in \Theta^{\beta}_{\alpha-\epsilon}}\rho(\theta)\|\leq \frac{\lipnorm{\rho}}{\beta}.
\]
Then, we have
\[
\essinf_{\theta\in \Theta_{\tau}} \rho(\theta)\geq \essinf_{\theta\in \Theta_{\alpha-\epsilon}}\rho(\theta)-\frac{\lipnorm{\rho}}{\beta},
\]
completing the proof. 
\end{proof}

\section{Application to the Bernoulli model: proofs}\label{SecAppBernProofs}

In this section, we give a proof of  \cref{ThmBernoulliMain} that follows the proof of \cref{ThmMatchingPriorsExist} quite closely. The essential difficulty is checking part 4 of \cref{assumptionjcAlt} - we need to see that a sufficiently rich family of measures is preserved under the map $\Q{S}$. 

Let $\lambda$ denote Lebesgue measure on $[0,1]$.
Fix $\FS$ to be the collection of finite subsets of $[0,1]$ that have size $|S| > 1000$ and satisfy
\[ \label{EqDefFFin}
\frac{|S \cap I | / |S|}{\lambda(I)} \in (0.99, 1.01)
\]
for all intervals $I \subseteq [0,1]$ of length $\lambda(I) > 0.01$.
For $a,b > 0$, define 
\[
\AGD{a}{b} = \{ \pi \in \Priors \, : \, \pi([0,a] \cup [1-a,1]) \leq 1 - b \}.
\]
Using notations from \cref{sec:existenceproof}, we use $\Restrict{\AGD{a}{b}}{S}$ to denote the restriction of $\AGD{a}{b}$ on $S$. 

\begin{lemma} \label{LemmaZInjBern}
There exist universal constants $A,B >0$ so that for all $0< a <A, 0 < b <B$ and all $S \in \FS$, 
\[
\Q{S}(\PM{S}) \subseteq \Restrict{\AGD{a}{b}}{S}
\]
if the associated family of credible regions falls into one of the following cases:
\begin{enumerate}
\item The family of credible regions is given by \cref{DefUsualInterval} and $0 < \alpha < 0.2$;
\emph{or}
\item The family of credible regions is given by \cref{DefSimpleCredInt} and $0 < \alpha + 2 \beta^{-1} < 0.2$; 
\emph{or}
\item The family of credible regions is given by \cref{DefContFamilyCredible} with underlying family of credible regions given \cref{DefUsualInterval}, the parameters satisfy \cref{gammasmall}, and 
$0 < \alpha + \gamma +  2 \beta^{-1} < 0.2$.
\end{enumerate}

\end{lemma}
\begin{proof} 
See \cref{secbernproof}.
\end{proof}

We next check that the map sending data and a prior to a posterior is appropriately continuous; the proof is a very small modification of the proof of \cref{implication4}.

\begin{lemma} \label{LemmaContPriorPosteriorBern}
Fix $a,b > 0$.
  Then the map
  \[
  (x, \pi) \mapsto \Post{x}{\pi} \: (X \times \AGD{a}{b}, \dX \otimes W_{1} ) \to (\Priors, W_{1})
  \]
  is Lipschitz.
\end{lemma}
\begin{proof}
See \cref{secbernproof}.
\end{proof}

Finally, we need the following straightforward bound:

\begin{lemma} \label{LemmaTrivContAGD}
Fix $a,b>0$ and $0 < \epsilon < \frac{ab}{2}$.
Then the $\epsilon$-fattening of $\AGD{a}{b}$ is contained in $\AGD{\frac{a}{2}}{b - \frac{2 \epsilon}{a}}$.
\end{lemma}
\begin{proof}
See \cref{secbernproof}.
\end{proof}

We can now prove the main result of this section:

\begin{proof} [of \cref{ThmBernoulliMain}]

The proof closely mimics the proof of \cref{ThmMatchingPriorsExist} assuming a trivial system of prior perturbations (i.e., $\pt{\pi}{\gamma}=\pi$).
The final steps of this proof are contained in \cref{SubsecProofMainThm}, though we need to make a few small changes in some other details.
Rather than rewriting the entirety of a fairly long argument, we simply list the small changes that must be made to the proof of \cref{ThmMatchingPriorsExist}:

\begin{enumerate}
\item Assume $\epsilon < \min\{\alpha, \frac{(0.2-\alpha)}{3}\}$.  We note that this bound on $\epsilon$ does not appear in the statement of  \cref{ThmBernoulliMain}. However, if the conclusion of that theorem holds for some $\epsilon'$, it holds immediately for all $\epsilon > \epsilon'$, and so this bound can be omitted. 
Recall that, in the proof of \cref{ThmMatchingPriorsExist}, we choose $\beta$ to satisfy $\beta^{-1} < \epsilon$.

\item \cref{ThmBernoulliMain} contains only claims about credible balls, so all references to HPD regions are not needed.

\item In order to show that matching priors exist, the proof of \cref{ThmMatchingPriorsExist} invokes \cref{ThmMainExistenceResult} (this occurs exactly once, in the proof of \cref{jcimplyztv}).
This should be replaced by a reference to  \cref{WeakThmMainExistenceResult}, 
which requires us to also establish the existence of sets 
$\FS \subseteq \PowerSet(\Theta)$, $\Main,\oMain \subseteq \Priors$, $\oMain$ open,
satisfying the properties outlined in \cref{assumptionjcAlt}.

We choose $\FS$ as defined above.
For constants $a,b,r_{\Main}>0$ that we specify below,
we choose $\Main$ to be $\AGD{a}{b}$ 
and choose $\oMain$ to be $\Main_{r_{\Main}}$, i.e., the $r_{\Main}$-fattening of $\Main$.
\cref{cfinite} of \cref{assumptionjcAlt} is clearly satisfied.
\cref{copen} of \cref{assumptionjcAlt} is satisfied by the $r_{\Main}/2$-fattening of $\Main$.  
\cref{cimage} of \cref{assumptionjcAlt} is satisfied for all $a,b > 0$ sufficiently small by part 3 of \cref{LemmaZInjBern}.
We note that \cref{LemmaZInjBern} requires that $\alpha + 3 \epsilon \leq 0.2$, which is satisfied for our range of $\epsilon$.
Finally, \cref{ccont} of \cref{assumptionjcAlt} is exactly \cref{assumptionjc} on a smaller set, so it is established in the proof of \cref{ThmMatchingPriorsExist} (with the replacements outlined here).

\item All references to \cref{AssumptionModelSmoothness} should be replaced by references to \cref{LemmaContPriorPosteriorBern}.
By \cref{LemmaTrivContAGD},  \cref{LemmaContPriorPosteriorBern} gives the same conclusion as \cref{AssumptionModelSmoothness} for priors in $M_{r_{M}}$ for $r_{M}$ sufficiently small (and in particular for our choice of $r_{M}$ below).

\item In the statement (and proof) of \cref{PropContinuity1} and \cref{LemmaContinuityTildeRAnnoying}, replace $\PM{\Theta}$ by $M_{r_{M}}$ and replace every occurrence of \cref{AssumptionModelSmoothness} by \cref{LemmaContPriorPosteriorBern}.

\item All references to \cref{assumptionjc} should be replaced by a reference to \cref{ccont} of \cref{assumptionjcAlt}, which gives the same conclusion for priors in $M_{r_{M}}$.
Note that \cref{assumptionjc} is in fact proved under the conditions of \cref{ThmMatchingPriorsExist}; after making the substitutions contained in this list, the same argument establishes \cref{ccont} of \cref{assumptionjcAlt}.

\end{enumerate}

Finally, we must specify the constants $a,b, r_{M}$.
We choose $a,b > 0$ to be any constant allowed in the statement of \cref{LemmaZInjBern}.
Next, we choose $r_{M} > 0$ to be any constant so that the $r_{M}$-fattening of $\AGD{a}{b}$ is contained in $\AGD{a'}{b'}$ for some pair $a',b' > 0$; such an $r_{M} > 0$ always exists by \cref{LemmaTrivContAGD}. 
\end{proof}

We note that the proof of \cref{ThmBernoulliMain} used very few details of the Bernoulli model - only the properties discussed in these lemmas. This is not an accident, and indeed the same basic strategy can be applied for other models with unbounded log-likelihoods. The main difficulty comes in checking the analogue of \cref{LemmaZInjBern}, which says that the map $\Q{S}$ preserves a class of measures $\AGD{a}{b}$ that is not too concentrated near the singularities of the log-likelihood. 

\section
[Proofs deferred from \cref*{secmorenote}]
{Proofs deferred from \cref{secmorenote}}
\label{AppRoutineAnalysis}

\subsection{Overview}

This section contains several proofs deferred from Section \ref{secmorenote}. Section \ref{AppConsAssump1} contains proofs of \cref{implication4} and \cref{implication5}, both of which concern consequences of \cref{assumptionde}. Section \ref{AppSupExist} contains proofs of \cref{LemmaModelExtensionsExist} and \cref{LemmaPriorPertsExist}, both of which concern the existence of supermodels.

\subsection
[Consequences of \cref*{assumptionde}]
{Consequences of \cref{assumptionde}}
\label{AppConsAssump1}

\begin{proof} [of \cref{implication4}]

  We first show that the map
  \[\label{sahoeushduehdue}
  (x, \pi) \mapsto \Post{x}{\pi} \: (X \times \Priors, \dX \otimes W_{1} ) \to (\Priors, W_{1})
  \]
  is continuous. Pick $x\in X$ and $\pi\in \Priors$.
  Pick a sequence $\{x_n: n\in \Nats\}\subseteq X$ and a sequence $\{\pi_n: n\in \Nats\}\subseteq \Priors$
  such that $x_n \to x$ and $W_{1}(\pi_n,\pi) \to 0$. It suffices to show that $W_{1}(\Post{x_n}{\pi_n},\Post{x}{\pi}) \to 0$.

Let $\cd$ be as in \cref{assumptionde}. For  $B \in \BorelSets{\Theta}$, define $\Phi(B) = \int_B \cd(\theta,x) \pi(\dee \theta)$ and for $n \in \mathbb{N}$ define $\Phi_{n}(B) = \int_B \cd(\theta,x_{n}) \pi_{n}(\dee \theta)$. Recall that by \cref{assumptionde}, $\inf_{x \in X, \, \theta \in \Theta} \cd(\theta,x)>0$  and so $\Phi(\Theta) > 0$; thus, by \cref{stEOSUThs}, $\Post{x}{\pi}(B) = \Phi(B)/\Phi(\Theta)$. 

A subset $A$ of a Borel probability space $(Y,\BorelSets Y,P)$ is a $P$-continuity set if $P(\boundary A)=0$, where $\boundary A$ denotes the boundary of $A$. Let $B$ be a $\Post{x}{\pi}$-continuity set. Since $\Phi(\Theta) > 0$, it is also a $\Phi$-continuity set; by \cref{assumptionde}, $\inf_{x \in X, \, \theta \in \Theta} \cd(\theta,x)>0$, so $B$ is also a $\pi$-continuity set.

Pick $\epsilon>0$. By the compactness of $\Theta$ and uniform continuity of $\cd$, there exists $N_1\in \Nats$ 
  such that $\abs{ \cd(\theta,x_n)-\cd(\theta,x) } < \epsilon$ for all $\theta\in \Theta$ and all $n>N_1$.
  By the weak convergence of $\pi_n \to \pi$ and $\pi$-continuity of $B$, 
  there exists $N_2\in \Nats$ such that $|\int_{B} \cd(\theta, x)\pi_n(\dee \theta)-\int_{B}\cd(\theta, x)\pi(\dee \theta)|<\epsilon$ for all $n>N_2$.
  Let $N=\max\set{N_1,N_2}$. For every $n>N$, we have
  \begin{align*}
    &\abs[\Big]{\int_{B} \cd(\theta, x_n)\pi_n(\dee \theta)-\int_{B}\cd(\theta, x)\pi(\dee \theta)} \\
    &=\abs[\Big]{\int_{B} \cd(\theta, x_n)\pi_n(\dee \theta)-\int_{B} \cd(\theta, x)\pi_n(\dee \theta)+\int_{B} \cd(\theta, x)\pi_n(\dee \theta)-\int_{B}\cd(\theta, x)\pi(\dee \theta) } \\
    &\leq \abs[\Big]{\int_{B} \cd(\theta, x_n)\pi_n(\dee \theta)-\int_{B} \cd(\theta, x)\pi_n(\dee \theta)} \\
    & \qquad \qquad + \abs[\Big]{\int_{B} \cd(\theta, x)\pi_n(\dee \theta)-\int_{B} \cd(\theta, x)\pi(\dee \theta)} < 2\epsilon.
  \end{align*}
  Thus, we have $\int_B \cd(\theta,x_n) \pi_n(\dee \theta) \to \int_{B} \cd(\theta,x) \pi(\dee \theta) = \Phi(B)$. Since the Wasserstein metric metrizes weak convergence of probability measures on a compact set, this implies that  the map given by \cref{sahoeushduehdue} is continuous.

  It is well-known that $\PD{\Theta}$ is a dense subset of $\Priors$ under the Wasserstein metric.
  To complete the proof, it suffices to show that the map $(x, \pi) \mapsto \Post{x}{\pi}$ is Lipschitz continuous on $X \times \PD{\Theta}$.
  Let $R = \sup_{\theta_{1},\theta_{2} \in \Theta} \norm{\theta_{1}-\theta_{2}}$.
  Fix priors $\mu, \nu \in \PD{\Theta}$ with densities $g_{\mu}, g_{\nu}$.
  For data $x,y \in X$, the associated posterior distributions can be assumed to have densities given by the formulas
  \begin{align*}
    p_{x,\mu} (\theta) &= Z_{\mu,x} \, g_{\mu}(\theta) \,\cd(\theta,x) \qquad \text{and} \\
    p_{y,\nu} (\theta) &= Z_{\nu,y} \, g_{\nu}(\theta) \,\cd(\theta,y).
  \end{align*}
  Fix a 1-Lipschitz function $h \: \Theta \to \Reals$.
  There exists a constant $c$ such that $\supnorm{h-c} \leq R$ and so we assume $c=0$ for the following argument, without loss of generality.
  Let $M \equiv \sup_{\theta\in \Theta, x\in X} \abs {\log \cd(\theta,X) }$.
  By assumption, $M < \infty$.
  Recall that $\Theta$ is compact and that $\cd$ is $\mC$-Lipschitz.
  
  By 
  the triangle inequality and the fact that $\cd$ is bounded and $C'$- Lipschitz,
  \begin{align} 
    & \abs[\Big]{ \int_{\Theta} (g_{\mu}(\theta) \cd(\theta,x)h(\theta) - g_{\nu}(\theta) \cd(\theta,y)) h(\theta) \dee \theta }  \nonumber \\
    &\leq \abs[\Big]{ \int_{\Theta} (g_{\mu}(\theta) - g_{\nu}(\theta)) h(\theta) \cd(\theta,x) \dee \theta } + \abs[\Big]{ \int_{\Theta} g_{\nu}(\theta) h(\theta) ( \cd(\theta,x) - \cd(\theta,y)) d \theta } \nonumber\\
    &\leq \lipnorm{ h(\cdot) \cd(\cdot,x)} W_{1}(\mu,\nu) + \supnorm{h} \int_{\Theta} g_{\nu}(\theta)( \cd(\theta,x) - \cd(\theta,y)) d \theta \nonumber\\
    &\leq (\lipnorm{h(\cdot)} \supnorm{\cd(\cdot,x)} + \supnorm{h(\cdot)} \lipnorm{\cd(\cdot,x)})\,W_{1}(\mu,\nu) + R\,\mC\, \dX(x,y)\\
    &\leq (e^M+\mC\,R)\,W_{1}(\mu,\nu) + R\,\mC\, \dX(x,y).\label{IneqGenWassChecking}
  \end{align}

  Dropping the constants $M, \mC, R$, and taking $h=1$, we have
  \[
  \abs{ Z_{\mu,x}^{-1} - Z_{\nu,y}^{-1} } = O( W_{1}(\mu,\nu) ) + O(\dX(x,y)).
  \]
  But then the bound $Z_{\mu,x} \le e^M$ implies
  \[
  \abs{ Z_{\mu,x} - Z_{\nu,y} } = O(W_{1}(\mu,\nu)) + O(\dX(x,y)).
  \]
  Combining this with \cref{IneqGenWassChecking} 
  completes the proof.
\end{proof}

\begin{proof} [of \cref{implication5}]

Pick a set $\BR \subseteq \PD{\fTheta}$ satisfying \cref{eqesssup} and let \[N = \sup_{\pi \in \BR} \, \Linfnorm{}{ \abver_{\pi} }  <\infty.\]
Fix $\epsilon > 0$, $x_1,x_2 \in X$, and $\mu_1,\mu_2 \in \BR$ with densities $\abver_{1}(\theta),\abver_{2}(\theta)$, respectively. For $i,j\in \set{1,2}$,
let $g_{i,j}(\theta)=\abver_{i}(\theta)\,\cd(\theta,x_{j})$
and $Z_{i}^{-1} = \norm{g_{i,i}}_1$, so that $Z_{i} g_{i,i}$ is a density function for $\Post{x_{i}}{\pi_{i}}$.

Fix $i\in \{1,2\}$.
By \cref{assumptionde}, there exist $0<c_1<c_2<\infty$ such that $c_1\leq \cd(\theta,x)\leq c_2$ for every $x\in X, \theta\in \Theta$.
Thus, we have
\[
Z_{i}^{-1}=\int_{\Theta} g_{i,i}(\theta)\dee \theta \leq c_2 \int_{\Theta} \abver_{i}(\theta) \dee \theta=c_2.
\]
Similarly, $Z_{i}^{-1}\geq c_1$. 
Thus
$Z_i\in [c_{2}^{-1},c_{1}^{-1}]$.
Then, for $G' = \int_{\Theta} \dee \theta$,
\[
\abs{ Z_{1}^{-1}-Z_{2}^{-1} }
&\leq \norm{g_{1,1} - g_{2,2}}_{1} \\
&\leq \norm{g_{1,1} - g_{1,2}}_{1} + \norm{g_{1,2}-g_{2,2}}_{1} \\
&\leq \Linfnorm{}{ \abver_{1} } \,\norm{ \cd(\cdot,x_{1}) - \cd(\cdot,x_{2})}_{1} 
           + \Linfnorm{}{ \abver_{1} - \abver_{2} } \, \norm{\cd(\cdot,x_{2})}_{1} \\
&\leq  N\, G' \mC \, \dX(x_{1},x_{2})  + c_2 \, G' \dsup(\mu_1,\mu_2) ,\label{zlipeq}
\]
where the last inequality used 
\cref{assumptionde}
to bound $\norm{ \cd(\cdot,x_{1}) - \cd(\cdot,x_{2})}_{1}$ and $\norm{\cd(\cdot,x_{2})}_{1}$.
Thus, 
there exists $\delta > 0$, such that
$\dX(x_1,x_2) + \dsup(\mu_{1},\mu_{2}) < \delta$ implies
\[ \label{IneqDiffInvNorm}
\abs{Z_{1}^{-1}-Z_{2}^{-1}}<\epsilon.
\]
As $Z_{i}^{-1}\in [c_{2}^{-1},c_{1}^{-1}]$ for $i\in \{1,2\}$, 
the same statement holds for $\abs{Z_{1}-Z_{2}}<\epsilon$.

We have
\[ \label{IneqLipBd38}
\dsup(\Post{x_1}{\mu_1},\Post{x_2}{\mu_2})
&= \Linfnorm{}{ Z_{1}g_{1,1}-Z_{2}g_{2,2} } \leq T_1 + T_2 + T_3,
\]
where
\[
  T_1 &= \Linfnorm{}{Z_{1}g_{1,1}-Z_{2}g_{1,1}} 
                 \leq c_{1}^{-2} \abs {Z_{1}^{-1} - Z_{2}^{-1}} \, \Linfnorm{}{g_{1,1}}  
                 \leq c_{1}^{-2}  N\,  c_{2} \, \abs{Z_{1}^{-1} - Z_{2}^{-1}},
\\T_2 &= \Linfnorm{}{Z_{2}g_{1,1}-Z_{2}g_{2,1}} \leq c_{1}^{-1}\, c_{2} \, \dsup(\mu_1,\mu_2), \text{ and}
\\T_3 &= \Linfnorm{}{Z_{2}g_{2,1}-Z_{2}g_{2,2}} 
                \leq c_{1}^{-1} \,N \, \Linfnorm{}{\cd(\cdot,x_1)-\cd(\cdot,x_2)}
                \leq c_{1}^{-1} \,N \, \mC\, \dX(x_{1},x_{2}),
\]
where the first and last inequality use \cref{assumptionde}. 
Thus, applying Inequality~\eqref{IneqDiffInvNorm} to bound the difference $|Z_{1}^{-1} - Z_{2}^{-1}|$ appearing in our bound on $T_{1}$,
there exists $\delta > 0$, such that
$\dX(x_1,x_2) + \dsup(\mu_{1},\mu_{2}) < \delta$ implies
$T_1,T_2,T_3<\epsilon/3$.
By Inequalities \cref{zlipeq} and \cref{IneqLipBd38}, we have shown that the map
  \[
  (x, \pi) \mapsto \Post{x}{\pi} \: (X \times \BR, \dX \otimes \dsup ) \to (\PD{\fTheta}, \dsup)
  \]
from $X \times \BR$ to $\PD{\fTheta}$ is Lipschitz continuous.

Let $\pi\in \BR$ with Lipschitz density function $\abver_{\pi}$. As the product of two bounded Lipschitz functions is Lipschitz, $p_{\pi,x}$ is Lipschitz, completing the proof. 
\end{proof}

\subsection{Existence of supermodels} \label{AppSupExist}

Before proving \cref{LemmaModelExtensionsExist}, we first quote the following classic result in general topology: 

\begin{theorem}%
\label{mcextension}
Let $S$ be a metric space,  let $E \subseteq S$, and let $f: E \to \Reals$ be $M$-Lipschitz.
There exists an $M$-Lipschitz extension $\hat{f}: S \to \Reals$ of $f$ to $S$. 
Furthermore, $\hat{f}$ satisfies
\[ \label{McshaneObs1}
\inf_{x\in S} \hat{f}(x) = \inf_{x\in E} f(x)
\qquad\text{and}\qquad
\sup_{x\in E} f(x) < \infty \implies \sup_{x\in S'} \hat{f}(x)  < \infty
\]
for all compact sets $S'$ satisfying $E \subseteq S' \subseteq S$.
\end{theorem}

\begin{remark}
We note that \cref{McshaneObs1} does not appear in the statement of \citep[][Thm.~1]{mcshane34}. 
However, by inspecting the (very explicit) construction of $\hat{f}$ in the original proof,
one can immediately see that $\hat{f}$ satisfies
\[ \label{McshaneObs1Orig}
\sup_{x\in S} \hat{f}(x) = \sup_{x\in E} f(x)
\qquad\text{and}\qquad
\inf_{x\in E} f(x) > -\infty \implies \inf_{x\in S'} \hat{f}(x) > -\infty
\]
for all compact set $S'$ satisfying $E \subseteq S' \subseteq S$. 
In order to obtain the desired inequalities (\cref{McshaneObs1}), 
we may apply the original theorem to find an $M$-Lipschitz extension $\cd$ of $-f$,
which then satisfies \cref{McshaneObs1Orig}. 
Then $-q$ is an $M$-Lipschitz extension of $f$ satisfying \cref{McshaneObs1}.

Note also that this construction gives an explicit and fairly simple formula for the construction of $\hat{f}$.
\end{remark}

\begin{proof} [of \cref{LemmaModelExtensionsExist}]
Let $\cd$ be the function whose existence is guaranteed by \cref{assumptionde} for $\SModel$.
Consider a measurable function $p \: \fTheta \times X \to \Reals$ satisfying: 
\begin{enumerate}
\item for all $\theta \in \fTheta$, $p(\theta,\cdot)$ is a probability density with respect to $\nu$; and
\item for all $\theta \in \Theta$ and $x \in X$, $p(\theta,x) = \cd(\theta,x)$.
\end{enumerate}
For $\theta \in \fTheta$, define $P'_\theta$ to be the probability measure on $X$ with density $p(\theta,\cdot)$ with respect to $\nu$. We then obtain a statistical model $\SModelfullalt$, which by (2) above is a supermodel of $\SModel$.
We now find such a function $p$ so that, moreover, the associated model $\SModelfullalt$ satisfies \cref{assumptionde}.

By \cref{mcextension}, there exists a map $\hat{\cd} \, : \, \Reals^{d} \times X \to \Reals$ 
such that $\hat{\cd}$ is $\mC$-Lipschitz,
\[ \label{EqExtMain}
\forall \theta \in \Theta,\ \forall x \in X,\ 
\hat{\cd}(\theta,x) = \cd(\theta,x),
\] 
\[ \label{IneqExtLower}
\inf_{\theta \in \fTheta, \, x \in X} \hat{\cd}(\theta,x) = \inf_{\theta \in \Theta, \, x \in X} \cd(\theta,x),
\]
and, since $\fTheta$ is compact,
\[\label{IneqExtUpper}
\sup_{\theta \in \fTheta, \, x \in X} \hat{\cd}(\theta,x) < \infty.
\]
\cref{IneqExtLower} and \cref{IneqExtUpper} imply that $\log \hat{\cd}$ is bounded on $\fTheta \times X$.

Note that $\hat{\cd}(\theta,\cdot)$ is, in general, not guaranteed to be a probability density with respect to $\nu$.
To fix this, define $Z(\theta) = \int \hat{\cd}(\theta,x) \nu(\dee x)$ 
for all $\theta \in \fTheta$. 
Note that $Z$ is Lipschitz because $\hat{\cd}$ is and $\nu(X) < \infty$ by \cref{assumptionde}.
By \cref{IneqExtLower} and \cref{IneqExtUpper}, there exist constants $c_{1}, c_{2}$ so that 
\[
0 < c_{1} \equiv \inf_{\theta \in \fTheta} Z(\theta) \leq \sup_{\theta \in \fTheta} Z(\theta) \equiv c_{2} < \infty.
\]
In particular, $\log Z$ is bounded.
Defining
\[ \label{EqExplicitSupermodel}
p(\theta,x) = Z(\theta)^{-1} \hat{\cd}(\theta,x),
\] 
we obtain a probability density $p(\theta,\cdot)$ with respect to $\nu$ for every $\theta \in \fTheta$, thus (1) is satisfied.
Note that, for $\theta \in \Theta$, $\cd(\theta,\cdot)$ is a probability density and so $Z(\theta) = 1$.
Thus, by \cref{EqExtMain}, for all $\theta \in \Theta$,
\[
p(\theta,x) = Z(\theta)^{-1} \hat{\cd}(\theta,x) = \hat{\cd}(\theta,x) = \cd(\theta,x),
\]
and so (2) holds. It remains to verify that $p$ has the properties required to witness that $\SModelfullalt$ satisfies \cref{assumptionde}.

We first verify that $p$ is Lipschitz. 
As $Z$ is bounded away from zero,
for all $x_{1}, x_{2} \in X$ and $\theta_{1},\theta_{2} \in \fTheta$,
\begin{gather*}
\abs{p(\theta_{1},x_{1}) - p(\theta_{2},x_{2})} 
= \frac{ \abs{ Z(\theta_{2})(\hat{\cd}(\theta_{1},x_{1}) - \hat{\cd}(\theta_{2},x_{2})) + \hat{\cd}(\theta_{2},x_{2}) (Z(\theta_{2}) - Z(\theta_{1}))}}
        {Z(\theta_{1}) \, Z(\theta_{2})} \\
\leq c_{1}^{-1}|\hat{\cd}(\theta_{1},x_{1}) - \hat{\cd}(\theta_{2},x_{2})| + c_{1}^{-2} \sup_{\theta \in \fTheta, \, x \in X} \hat{\cd}(\theta,x) |Z(\theta_{1}) - Z(\theta_{2})|.
\end{gather*}
Recall that $\hat{\cd}$ is bounded on $\fTheta \times X$
and that $\hat{\cd}$ and $Z$ are Lipschitz.
It thus follows that $p$ is Lipschitz.
Finally, since $\log \hat{\cd}$ and $\log Z$ are both bounded, their difference, $\log p$, is too. This completes the proof.
\end{proof}

\begin{proof}[of \cref{LemmaPriorPertsExist}]
Let $\fTheta = \Theta_{3}$ be the 3-fattening of $\Theta$; since $\Theta$ is compact, $\fTheta$ is as well. 
Let $\SModelfullalt$ be the supermodel of $\SModel$ whose existence is guaranteed by \cref{LemmaModelExtensionsExist}, 
and let $p$ be the density whose existence is guaranteed by \cref{assumptionde} for $\SModelfullalt$.

Fix $\gamma > 0$. Without loss of generality,  we may assume that $0 < \gamma < 1$.
Let $\bg : \Reals^d \to \NNReals$ 
be the density of some probability distribution on $\Reals^d$  with mean zero such that
\begin{enumerate}
\item the support of $\bg$ is contained in the unit ball; and
\item $\bg$ is Lipschitz (and thus bounded). 
\end{enumerate}
For $x > 0$, let $\bg^{(c)}(x) \defas c^{-d} \bg(\frac{x}{c})$, i.e., 
the probability density after the change of variable transformation $x \mapsto c x$. 

Finally, we define our family of prior perturbations by the convolution formula 
\[ \label{EqExplicitFamilyPerturbations}
\pi_{\gamma} =  \bg^{(\gamma)} \ast \pi.
\]

We now check that this is indeed a family of prior perturbations, and that it satisfies \cref{assumptionPerturb}.

Fix $\mu, \nu \in \Priors$ and $\epsilon > 0$. Let $\theta_{1} \sim \mu$, $\theta_{2} \sim \nu$ be coupled so that 
\[
\EE[\norm{\theta_{1} - \theta_{2}}_{2}] \leq W_{1}(\mu,\nu) + \epsilon;
\]
such a coupling exists by the definition of the Wasserstein distance. 
Let $\xi \sim \bg^{(\gamma)}$ be independent of these random variables.
Then $\theta_{1}' = \theta_{1} + \xi \sim \pt{\mu}{\gamma}$ and $\theta_{2}' = \theta_{2} + \xi \sim \pt{\nu}{\gamma}$.
Thus $(\theta_{1}',\theta_{2}')$ defines a coupling of $\pt{\mu}{\gamma}$ and $\pt{\nu}{\gamma}$, 
and then
\[
W_{1}(\mu, \pt{\mu}{\gamma}) 
\leq \EE[\norm{\theta_{1} - \theta_{1}'}_{2}] 
= \EE[\norm{\xi}_{2}] \leq \gamma
\]
and 
\[
W_{1}(\pt{\mu}{\gamma}, \pt{\nu}{\gamma}) 
\leq \EE[\norm{\theta_{1}' - \theta_{2}'}_{2}] 
= \EE[\norm{\theta_{1} - \theta_{2}}_{2}] 
\leq W_{1}(\mu,\nu) + \epsilon.
\]
Since $\epsilon > 0$ was arbitrary, 
these two inequalities show that \cref{EqExplicitFamilyPerturbations} 
does define a system of prior perturbations with $D(\gamma)=1$.

It remains to show that this family of prior perturbations satisfies \cref{assumptionPerturb}.
It is straightforward to verify that $\pt{\pi}{\gamma}$ is Lipschitz since $\bg^{(\gamma)}$ is. 
Note that $\pt{\mu}{\gamma}$ and $\pt{\nu}{\gamma}$ admit densities,
which we denote by $p_{\mu, \gamma}$ and $p_{\mu, \gamma}$.
For a measurable function $u$ from $\Reals^n$ to $\Reals^n$, let $\Linfnorm{T}{u}$ denote the $L^\infty$-norm with respect to Lebesgue measure restricted to $T \subseteq \smash{\Reals^n}$.
So we have,
\begin{equation*}
\Linfnorm{\fTheta}{p_{\mu,\gamma}}
= \Linfnorm[\Big]{\fTheta}{ \int_{\Theta} \bg^{(\gamma)}(\cdot - z) \mu(\dee z) } 
\leq \Linfnorm{}{\bg^{(\gamma)}} 
= \gamma^{-d} \Linfnorm{}{\bg},
\end{equation*} 
so in particular the $L^\infty$ norm of any perturbed prior is uniformly bounded by $\gamma^{-d}\Linfnorm{}{\bg}$. 
Next, for every $ \theta \in \fTheta$,  
\begin{align*}
\abs {p_{\mu,\gamma}(\theta) - p_{\nu,\gamma}(\theta) } 
&= \abs[\Big]{\int \bg^{(\gamma)}(\theta-z) (\mu(\dee z) - \nu(\dee z))} \\
&\leq \lipnorm{\bg^{(\gamma)}} \, W_{1}(\mu, \nu) 
= \gamma^{-d}\lipnorm{\bg} \, W_{1}(\mu,\nu).
\end{align*}
Taking $D'(\gamma) = \gamma^{-d} \max \set{ \lipnorm{\bg}, \Linfnorm{}{\bg} }$ completes the proof.
\end{proof}

\section
[Proofs deferred from \cref*{sec:existenceproof}]
{Proofs deferred from \cref{sec:existenceproof}}
\label{AppNSASec}

\subsection{Overview}

This section includes proofs deferred from \cref{sec:existenceproof}, with subsections labelled in the same order as that section.

\subsection
[Proofs deferred from \cref*{SecExNSMatch}]
{Proofs deferred from \cref{SecExNSMatch}}
\label{Appdef43}

\begin{proof} [of \cref{restrictionpreserves}]
  We first verify conditions (1) and (2) from \cref{credfamdefn}.
  Regarding (1), \cref{aoeuduinhd2} is equivalent to the statement
  \[\label{uhtdute}
  \forall \theta \in\Theta,\
  \charfunc{\tau'_{\pi}(x,u)}(\theta) \charfunc{\fTheta}(\theta) = \charfunc{\tau_{\extn{\pi}{\Theta}}(x,u)}(\theta) \charfunc{\fTheta}(\theta).
  \]
  Because $\fTheta \in \BorelSets{\Theta}$, the right hand side is product measurable and
  so its restriction to $\fTheta$ is product measurable on the trace $\sigma$-algebra.

  Regarding (2), let $\pi \in \Restrict{G}{\fTheta}$. Then $\extn{\pi}{\Theta}(\fTheta) = 1$.
  Thus, by \cref{stEOSUThs}, for $\margalt{\pi}$-almost all $x \in X$, $\Postalt{x}{\pi} = \Post{x}{\extn{\pi}{\Theta}}$ on $\BorelSets{\fTheta}$ and so
  \[
  \Postalt{x}{\pi}(\tau'_{\pi}(x,u))
  = \Post{x}{\extn{\pi}{\Theta}}(\tau'_{\pi}(x,u))
  = \Post{x}{\extn{\pi}{\Theta}}(\tau_{\extn{\pi}{\Theta}}(x,u)\cap \fTheta)
  = \Post{x}{\extn{\pi}{\Theta}}(\tau_{\extn{\pi}{\Theta}}(x,u)).
  \]
  This implies (2).

Next, we verify the form of the rejection probability function for $\tau'$.  Let $\frset{}{}{},\frsetalt{}{}{}$ be the rejection probability functions for $\tau,\tau'$, respectively.
  For all $\theta \in \fTheta$, $x \in X$, and $\pi \in \Restrict{G}{\fTheta}$,
  \[\label{aeoundhuane}
  \frsetalt{\pi}{\theta}{x}
  &= 1 - \int_{[0,1]} \charfunc{\tau'_{\pi}(x,u)}(\theta) \, \dee u
  \\&= 1 - \int_{[0,1]} \charfunc{\tau_{\extn{\pi}{\Theta}}(x,u)}(\theta) \, \dee u
  = \frset{\extn{\pi}{\Theta}}{\theta}{x},
  \]
  where the first equality is by definition and the second follows from \cref{uhtdute}.

  Finally, we check that \cref{assumptionztv} holds for our appropriate subset and submodel.  Suppose \cref{assumptionztv} holds for $\tau = \CRegions{\tau}{\pi}{\Priors}, \Main$ in $\SModel$.
  Let $z'$ be defined as in \cref{EqDefZMap}, but for $\tau'$ in $\SModelalt{\fTheta}$.
  By \cref{aeoundhuane} and the definition of $z,z'$,
  we have $z'_{\theta}(\pi) = z_{\theta}(\extn{\pi}{\Theta})$ for all $\pi \in \PM{\fTheta}$.
  Let $\theta \in \fTheta$, let $B$ be an open set in $[-1,1]$,
  and define $\rho : \PM{\fTheta} \to \Priors$ to be $\rho{(\pi)} = \extn{\pi}{\Theta}$ for $\pi \in \PM{\fTheta}$.
  Then the inverse image of $B$ in $\FPM{\fTheta}\cap \Restrict{\Main}{\fTheta}$ is
  \[
  (z'_{\theta})^{-1}[B] \cap \FPM{\fTheta}\cap \Restrict{\Main}{\fTheta}
  &=\set{ \pi \in \FPM{\fTheta}\cap \Restrict{\Main}{\fTheta} : z'_{\theta}(\pi) \in B }
  \\&= \rho^{-1}[\set{ \pi \in \FPM{\Theta}\cap \Restrict{\Main}{\fTheta} : z_{\theta}(\pi) \in B }]
  \]
  By \cref{assumptionztv}, the right hand side is the inverse image of an open set.
  Because $\rho$ is continuous with respect to total variation, the result follows.
\end{proof}

We now prove the main result. 

\begin{proof}[of \cref{hyperconfidence}]
  Let $\Main$ be a subset of $\PM{\Theta}$ such that \cref{assumptionztv} holds for $\tau,\Main$. 
  Let $\fS$ be the set of all nonempty, finite subsets of $\Theta$.
  For $T \in \fS$,
  let $\cD_{T}$ be the set of all functions $X \times [0,1] \to \BorelSets{T}$;
  let $\cC_{T}$ be the set of all functions $\tau : \PM{T} \to \cD_{T}$; and
  let $\cR_{T}$ be the set of all functions $\frset{}{}{}: T \times X \times \Priors \to [0,1]$.

  We now define the key relations.
  \begin{enumerate}
  \item
    Let $\cH$ be the set of all quartuples $(T,\tau',\frsetalt{}{}{}, M')$, for $T \in \fS$, $\tau' \in \cC_{T}$, $\frsetalt{}{}{} \in \cR_{T}$ and $M'\in \PowerSet(\PM{T})$ such that
    $(\tau'(\pi))_{\pi \in \PM{T}}$ is a family of credible regions in the submodel $\SModelalt{T}$,
    $\frsetalt{}{}{}$ is the rejection probability function for $\tau'$, 
    $\Q{T}(M')\subseteq M'$,
    $M'$ is a nonempty compact set of $\Reals^{|T|}$, and
    \cref{assumptionztv} holds for $\tau', M'$.

  \item
    Let $\cP$ be the set of all $(T,\tau',\frsetalt{}{}{}, M', \pi)$,
    for $T \in \fS$, $\tau' \in \cC_{T}$, $\frsetalt{}{}{} \in \cR_{T}$, and $\pi \in \PM{T}$, such that $(T,\tau',\frsetalt{}{}{}, M') \in \cH$
    and $\pi$ is a matching prior for $\tau'$ in $\SModelalt{T}$.
  \end{enumerate}

  We are now in a place to state a special case of  \cref{muresult}.
  In particular, \cref{muresult} implies that the following statement holds:
  \[
  \forall (T,\tau',\frsetalt{}{}{},M') \in \cH,\ \exists \pi \in \PM{T},\ (T,\tau',\frsetalt{}{}{},M',\pi) \in \cP.
  \]
  We may then invoke the transfer principle to arrive at the following nonstandard statement:\footnote{
    Technically, we must verify that the sets/functions that we have constructed and intend to
    extend---namely, $\cH,\cP,\Priors,\BorelSets{\Theta}$, etc.---have finite ``rank'' relative to a base set,
    in this case $Z=\Theta \cup X \cup \Reals$,
    where an element is of rank 0 if it is an element of $Z$; rank 1 if it is otherwise in $Z \cup \PowerSet{(Z)}$;
    rank 2 if it is otherwise in $Z \cup \PowerSet{(Z)} \cup \PowerSet{(Z \cup \PowerSet{(Z)})}$; etc.
    Viewing functions as relations, and coding tuples as relations, the objects we have defined above all have small rank.

    We have also taken liberties in our logical statements to make them more readable. Technically, variables cannot play the role of function or relation symbols.
    Formally, function evaluation is carried out by an evaluation operator associated with the rank of the function space.
    We have also taken liberties around quantification. Formally, quantification is of the form $\forall v \in c$, where $v$ is a variable and $c$ is a constant corresponding to a fixed set in the universe. We have allowed terms in place of constants, e.g., in $\exists \pi \in \PM{T}$,
    $T$ is a variable. We've also allowed tuples $\tuple{a_1,\dots,a_k}$ instead of variables.
    Formally, tuples can by encoded and decoded using function terms. These rewrite rules commute with the process of transferring the sentence into the nonstandard universe, and so we have kept tuples for readability.
  }
  \[\label{transfermuresult}
  \forall (T,\tau',\frsetalt{}{}{},M') \in \NSE{\cH},\ \exists \pi \in \NSE{\PM{T}},\  (T,\tau',\frsetalt{}{}{},M',\pi) \in \NSE{\cP}.
  \]

  By the definition of matching priors
  and the rejection probability function identity of \cref{restrictionpreserves},
  the following statement holds:
  \begin{alignat*}{1}
    \forall (T,\tau',\frsetalt{}{}{},M',\pi) \in \cP,\
    & \exists \pi' \in \Priors,\
    \\ &(\forall B \in \BorelSets{\Theta},\ \pi'(B) = \pi(B \cap T) )
    \\ &\land (\forall \theta \in T,\ \textstyle \int \frset{\pi'}{\theta}{x}\,P_{\theta}(\dee x) \le \alpha).
  \end{alignat*}
  Again, by transfer, we have
  \begin{alignat*}{1}
    \forall (T,\tau',\frsetalt{}{}{},M',\pi) \in \NSE{\cP},\
    & \exists \pi' \in \sPriors,\
    \\ &(\forall B \in \NSE{\BorelSets{\Theta}},\ \pi'(B) = \pi(B \cap T) )
    \\ &\land (\forall \theta \in T,\ \textstyle \int \NSE{\frset{\pi'}{\theta}{x}}\,\NSE{P_{\theta}}(\dee x) \le \alpha).
  \end{alignat*}
  
  We now pick a hyperfinite set $T_0\subseteq \NSE{\Theta}$ such that $\NSE{\Q{T_0}}(\Restrict{\NSE{\Main}}{T_0})\subseteq \Restrict{\NSE{\Main}}{T_0}$ and $\Restrict{\NSE{\Main}}{T_0}$ is a nonempty *compact. By hypothesis and \cref{restrictionpreserves},
  \cref{assumptionztv} holds for $\xi, \Restrict{\NSE{\Main}}{T_0}$, where $\xi$ is the *restriction of $\NSE{\tau}$. Let $\frsetalt{}{}{}$ denote the *rejection probability function associated with $\xi$.
  We know that $(T_0,\xi,\frsetalt{}{}{},\Restrict{\NSE{\Main}}{T_0})\in \NSE{\cH}$ is true.
  Thus, we obtain the desired result by
  combining this fact with the
  previous logical sentence, \cref{transfermuresult}.
\end{proof}

\subsection
[Proofs deferred from \cref*{SecWassPuss}]
{Proofs deferred from \cref{SecWassPuss}}
\label{appsecdef44}

This section contains the proof of \cref{pdclose}. Before giving the main argument, we recall several useful results.
First, recall that all finite Borel measures on metric spaces are Radon. 
Thus, $P_{\theta}$ is Radon for every $\theta\in \Theta$.

\begin{lemma}%
\label{apushdown}
Let $Y$ be a Hausdorff topological space with Borel $\sigma$-algebra $\BorelSets Y$.
Suppose $\mu$ is a Radon probability measure on $Y$. Then $\mu(B)=\Loeb{\NSE{\mu}}(\ST^{-1}(B))$ for every $B\in \BorelSets Y$.
\end{lemma}

Let $\Pi$ be an internal probability measure on $(\NSE{Y},\NSE{\BorelSets Y})$ where $Y$ is compact.
By \cref{pushdown}, the push-down measure $\pd{\Pi}$ of $\Pi$ is a standard probability measure
on $(Y, \BorelSets Y)$.
We first prove the following lemma.
Let $\closure{U}$ denote the closure of a set $U$.

\begin{lemma}\label{pdcontinuity}
  Let $Y$ be a compact metric space with Borel $\sigma$-algebra $\BorelSets Y$.
  Let $\Pi$ be an internal probability measure on $(\NSE{Y},\NSE{\BorelSets Y})$.
  Then we have
  $\Pi(\NSE{U})\approx \pd{\Pi}(U)$ for every $\pd{\Pi}$-continuity set $U$.
\end{lemma}
\begin{proof}
  Pick any $\pd{\Pi}$-continuity set $U$.
  As $U$ is a $\pd{\Pi}$-continuity set and $\pd{\Pi}$ is Radon, by \cref{apushdown}, we have
  \[
  \Loeb{\NSE{\pd{\Pi}}}(\ST^{-1}(U))=\pd{\Pi}(U)=\pd{\Pi}(\closure{U})=\Loeb{\NSE{\pd{\Pi}}}(\ST^{-1}(\closure{U})).
  \]
  By the construction of $\pd{\Pi}$, we also have $\Loeb{\Pi}(\ST^{-1}(U))=\pd{\Pi}(U)$ and $\Loeb{\Pi}(\ST^{-1}(\closure{U}))=\pd{\Pi}(\closure{U})$.
  Thus, $\Loeb{\Pi}(\ST^{-1}(U))=\Loeb{\NSE{\pd{\Pi}}}(\ST^{-1}(U))$ and $\Loeb{\Pi}(\ST^{-1}(\closure{U}))=\Loeb{\NSE{\pd{\Pi}}}(\ST^{-1}(\closure{U}))$.
  As $U$ is open and $X$ is compact, we have $\ST^{-1}(U)\subseteq \NSE{U}\subseteq \ST^{-1}(\closure{U})$.
  As $\Loeb{\Pi}(\ST^{-1}(U))=\pd{\Pi}(U)$ and $\Loeb{\Pi}(\ST^{-1}(\closure{U}))=\pd{\Pi}(U)$,
  we have $\Pi(\NSE{U})\approx \pd{\Pi}(U)$ for all $\pd{\Pi}$-continuity set $U$. 
\end{proof}

\begin{lemma}\label{cpartition}
  Let $Y$ be a compact metric space with Borel $\sigma$-algebra $\BorelSets Y$.
  Let $\Pi$ be an internal probability measure on $(\NSE{Y},\NSE{\BorelSets Y})$.
  For every $n\in \Nats$, there exists a finite partition $\{A_1,\dotsc,A_m\}$ consisting of Borel sets with diameter no greater than $\frac{1}{n}$ such that $\Pi(\NSE{A_i})\approx \NSE{\pd{\Pi}}(\NSE{A_i})$ for all $i\leq m$.
\end{lemma}
\begin{proof}
  Pick $n\in \Nats$.
  For every $y\in Y$, there are uncountably many open balls centered at $y$ with diameter no greater than $\frac{1}{n}$.
  The boundaries of these open balls form a uncountable collection of disjoint sets.
  Thus, for every $y\in Y$, we can pick an open ball $U_y$ containing $y$ such that its diameter is no greater than $\frac{1}{n}$ and it is a $\pd{\Pi}$-continuity set.
  As $Y$ is compact, there is a finite subcollection of $\{U_{y} \st y\in Y\}$ that covers $Y$.
  Denote this finite subcollection by $\mathcal{K}_n=\{U_{y_1},\dotsc, U_{y_m}\}$ for some $m\in \Nats$.
  Pick $i,j\leq m$.
  Note that $\boundary(U_{y_i}\cap U_{y_j})\subseteq \boundary U_{y_i}\cup \boundary U_{y_j}$ and $\boundary(U_{y_i}\cup U_{y_j})\subseteq \boundary U_{y_i}\cup \boundary U_{y_j}$.
  Hence any finite intersection (union) of elements from $\mathcal{K}_n$ is an open $\pd{\Pi}$-continuity set.
  For every $i\leq m$, let
  \[
  \textstyle V_i=U_{y_i}\setminus \bigcup_{j<i}U_{y_j}.
  \]
  Note that $V_i=U_{y_i}\setminus F_i$ where $F_i=\bigcup_{j<i}U_{y_j}\cap U_{y_i}$.
  Moreover, both $U_{y_i}$ and $F_i$ are open $\pd{\Pi}$-continuity sets.
  By \cref{pdcontinuity}, we have $\Pi(\NSE{U_{y_i}})\approx \NSE{\pd{\Pi}}(\NSE{U_{y_i}})$ and $\Pi(\NSE{F_i})\approx \NSE{\pd{\Pi}}(\NSE{F_i})$.
  Thus, we have $\Pi(\NSE{V_i})\approx \NSE{\pd{\Pi}}(\NSE{V_i})$ hence $\{V_1,\dotsc,V_m\}$ is the desired partition.
\end{proof}

Let $Y$ be a metric space and recall that $\NSE{\Lip{Y}{1}}$ is the nonstandard extension of the set of 1-Lipschitz functions on $Y$.
For any $F\in \NSE{\Lip{Y}{1}}$, $F$ is S-continuous, i.e., $F(y_1)\approx F(y_2)$ for $y_1\approx y_2\in \NSE{Y}$.
Note that every $P\in \NSE{\PM{Y}}$ is an internal probability measure on $(\NSE{Y},\NSE{\BorelSets Y})$. 

We are now ready to prove the main result: 

\begin{proof} [of \cref{pdclose}]
  Pick $F\in \NSE{\Lip{Y}{1}}$ and fix some $n\in \Nats$.
  Note that $F(t)$ may be an infinite number for every $t\in \NSE{Y}$.
  Let $\{V^{n}_1,\dotsc,V^{n}_m\}$ be the finite partition of $Y$ as in \cref{cpartition} such that the diameter of $V^{n}_{i}$ is no greater than $\frac{1}{n}$ for every $i\leq m$.
  For any $i\leq m$, pick $y_i\in V^{n}_{i}$.
  As $F\in \NSE{\Lip{Y}{1}}$ and the diameter of every $V^{n}_{i}$ is no greater than $\frac{1}{n}$,
  we have $\abs{F(y_i)-F(y)} \leq \frac{1}{n}$ for all $i\leq m$ and all $y\in \NSE{V^{n}_i}$.
  Hence we have
  \[
  &\abs[\Big]{ \int_{\NSE{Y}} F(y)\ \Pi(\dee y)-\sum_{i\leq m}\int_{\NSE{V^{n}_{i}}} F(y_i)\ \Pi(\dee y) } \\
  &\qquad \leq \sum_{i\leq m}\int_{\NSE{V^{n}_{i}}} \abs{F(y)-F(y_i)}\ \Pi(\dee y)
  \leq \frac{1}{n}.
  \]
  Similarly, we have $\abs{\int_{\NSE{Y}} F(y)\ \NSE{\pd{\Pi}}(\dee y)-\sum_{i\leq m}\int_{\NSE{V^{n}_{i}}} F(y_i)\ \NSE{\pd{\Pi}}(\dee y)} \leq \frac{1}{n}$.

  We now compare $\sum_{i\leq m}\int_{\NSE{V^{n}_{i}}} F(y_i)\ \Pi(\dee y)$ and $\sum_{i\leq m}\int_{\NSE{V^{n}_{i}}} F(y_i)\ \NSE{\pd{\Pi}}(\dee y)$.
  Note that
  \[
  &\abs[\Big]{ \sum_{i\leq m}\int_{\NSE{V^{n}_{i}}} F(y_i)\ \Pi(\dee y)-\sum_{i\leq m}\int_{\NSE{V^{n}_{i}}} F(y_i)\ \NSE{\pd{\Pi}}(\dee y) } \\
  &\qquad=\abs[\Big]{ \sum_{i\leq m}F(y_i)(\Pi({\NSE{V^{n}_{i}}})-\NSE{\pd{\Pi}}(\NSE{V^{n}_{i}})) } \\
  &\qquad=\abs[\Big]{ \sum_{i\leq m}(F(y_1)+k_i)(\Pi({\NSE{V^{n}_{i}}})-\NSE{\pd{\Pi}}(\NSE{V^{n}_{i}})) }
  \]
  where $k_i$ is the difference between $F(y_i)$ and $F(y_1)$.
  As $F\in \NSE{\Lip{Y}{1}}$ and $Y$ is compact, we know that $k_i\in \NS{\NSE{\Reals}}$ for all $i\leq m$.
  Hence, we have
  \[
  &\abs[\Big]{ \sum_{i\leq m}(F(y_1)+k_i)(\Pi({\NSE{V^{n}_{i}}})-\NSE{\pd{\Pi}}(\NSE{V^{n}_{i}})) } \\
  &\quad=\abs[\Big]{ \sum_{i\leq m}F(y_1)(\Pi({\NSE{V^{n}_{i}}})-\NSE{\pd{\Pi}}(\NSE{V^{n}_{i}}))+\sum_{i\leq m}k_i(\Pi({\NSE{V^{n}_{i}}})-\NSE{\pd{\Pi}}(\NSE{V^{n}_{i}})) }\\
  &\quad=\abs[\Big]{\sum_{i\leq m}k_i(\Pi({\NSE{V^{n}_{i}}})-\NSE{\pd{\Pi}}(\NSE{V^{n}_{i}})) } \approx 0.
  \]
  Thus, $\abs{\int_{\NSE{Y}} F(y)\ \Pi(\dee y)-\int_{\NSE{Y}} F(y)\ \NSE{\pd{\Pi}}(\dee y)} \lessapprox {2}/{n}$.
  As $n$ is arbitrary, it follows that 
  \[ 
       \int_{\NSE{Y}} F(y)\ \Pi(\dee y) \approx \int_{\NSE{Y}} F(y)\ \NSE{\pd{\Pi}}(\dee y),
  \]
   which implies that $\NSE{W}(\Pi, \NSE{\pd{\Pi}})\approx 0$.
\end{proof}

\subsection
[Proofs deferred from \cref*{SecMainThmRes}]
{Proofs deferred from \cref{SecMainThmRes}}
\label{appsecdef45}

This section contains a proof of \cref{phipush}.  Before giving the main argument, we recall or prove several useful lemmas, starting with the following nonstandard characterization of compactness.

\begin{lemma}%
\label{gencompact}
  Let $Y$ be a compact Hausdorff space.
  For all $y\in \NSE{Y}$, there exists $x\in Y$ such that $y\in \NSE{U}$ for every open set $U$ containing $x$.
\end{lemma}

The proof is illustrative, and so we reproduce it here for the nonexpert.
\begin{proof}
  Suppose the statement is false for some $y_0\in \NSE{Y}$.
  For every $x\in Y$, there is an open set $U_x$ containing $x$ such that $y_0\not\in \NSE{U}_x$.
  As $Y$ is compact, there is a finite subcollection of $\{U_x \st x\in Y\}$ that covers $Y$.
  Denote such finite subcollection by $\{U_{x_1},\dotsc, U_{x_n}\}$ for some $n\in \Nats$ and $x_1,\dotsc,x_n\in Y$.
  Then we have $\bigcup_{i\leq n}\NSE{U_{x_i}}=\NSE{Y}$.
  Note that $y_0\not\in \NSE{U_{x_i}}$ for all $i\leq n$.
  This is a contradiction.
\end{proof}

Fix $\alpha \in (0,1)$ and a family $\tau = (\tau_{\pi})_{\pi \in \Priors}$ of $1-\alpha$ credible regions with rejection probability function $\frset{}{}{}$.

\begin{lemma}\label{phiscts}
  Suppose \cref{ccont} of \cref{assumptionjcAlt} holds for $\tau, \frset{}{}{}, \oMain$.
  For all $\theta\in \Theta$,
  $\Pi_1,\Pi_2\in \sPriors$, and $x_1, x_2\in \NS{\NSE{X}}$,
  if there exists a $\pi_0\in \oMain$ such that $\NSE{W}(\Pi_1,\NSE{\pi_0})\approx \NSE{W}(\Pi_2,\NSE{\pi_0})\approx 0$ and $x_{1} \approx x_{2}$ then $\NSE{\frset{\Pi_1}{\theta}{x_1}}\approx \NSE{\frset{\Pi_2}{\theta}{x_2}}$.
\end{lemma}
\begin{proof}
  Fix some $\theta\in \Theta$.
  Pick any two $\Pi_1,\Pi_2\in \sPriors$ such that there exists $\pi_0\in \oMain$ with $\NSE{W}(\Pi_1,\NSE{\pi_0})\approx \NSE{W}(\Pi_2,\NSE{\pi_0})\approx 0$ and any two $x_1\approx x_2\in \NS{\NSE{X}}$.
  We pick $x_0\in X$ such that $x_1\approx x_2\approx x_0$.
  Pick any $\epsilon\in \PosReals$.
  By \cref{assumptionjcAlt}, there is some $\delta\in \PosReals$ such that, for all $\pi\in \Priors$,
  the following statements holds:
  \begin{gather}
    \begin{split}
      (\forall x\in X)((W(\pi,\pi_0)<\delta\wedge \dX(x,x_0)<\delta) \hspace*{6.5em}
      \\ \implies \abs{\frset{\pi}{\theta}{x}-\frset{\pi_0}{\theta}{x_0} } < \epsilon). \hspace*{4em}
    \end{split}
  \end{gather}
  By the transfer principle, for all $\Pi\in \sPriors$ and $x\in \NSE{X}$, the following statements holds:
  \begin{gather}
    \begin{split}
      (\NSE{W}(\Pi,\NSE{\pi_{0}})<\delta \wedge \NSE{\dX}(x,x_0)<\delta) \hspace*{3em}
      \\ \implies (\abs{\NSE{\frset{\Pi}{\theta}{x}}-\NSE{\frset{\NSE{\pi_{0}}}{\theta}{x_{0}}}} <\epsilon).
    \end{split}
  \end{gather}
  In particular, we have $\abs{ \NSE{\frset{\Pi_{i}}{\theta}{x_{i}}}-\NSE{\frset{\NSE{\pi}_0}{\theta}{x_0}}} < \epsilon$ for $i\in \set{1,2}$.
  As $\epsilon$ was chosen arbitrarily, we have $\NSE{\frset{\Pi_1}{\theta}{x_1}}\approx \NSE{\frset{\Pi_2}{\theta}{x_2}}$.
\end{proof}

We use the following lemma in nonstandard integration theory.

\begin{lemma}%
\label{bintegral}
  Suppose $(\Omega,\cA,P)$ is an internal probability space and $F \: \Omega\to \NSE{\Reals}$ is an internal measurable function such that $\SP{F}$ exists everywhere.
  Then $\SP{F}$ is integrable with respect to $\Loeb{P}$ and $\int F(x) P(\dee x)\approx \int \SP{F}(x) \Loeb{P}(\dee x)$.
\end{lemma}

We are now ready to prove the main technical result of this section: 

\begin{proof} [of \cref{phipush}]
  Pick any standard $\theta\in \Theta$.
  As the function $\NSE{\frset{}{}{}}$ is bounded, by \cref{bintegral}, we have
  \[
  \int \NSE{\frset{\Pi}{\theta}{x}}\NSE{\Model}_{\theta}(\dee x)
  \approx \int \SP({\NSE{\frset{\Pi}{\theta}{x}}})\Loeb{\NSE{\Model}_{\theta}}(\dee x)
  \]
  and
  \[
  \int \NSE{\frset{\NSE{\pd{\Pi}}}{\theta}{x}}\NSE{\Model}_{\theta}(\dee x)
  \approx \int \SP({\NSE{\frset{\NSE{\pd{\Pi}}}{\theta}{x}}})\Loeb{\NSE{\Model}_{\theta}}(\dee x).
  \]
  For every $x\in \NS{\NSE{X}}$, as $\pd{\Pi}\in \oMain$, by \cref{pdclose,phiscts},
  we have 
  \[
    \NSE{\frset{\Pi}{\theta}{x}} \approx \NSE{\frset{\NSE{\pd{\Pi}}}{\theta}{x}}.
  \]
  As $\theta\in \Theta$, we know that $\Loeb{\NSE{\Model}_{\theta}}(\NS{\NSE{X}})=1$.
  Hence we have
  \[
  \int \SP({\NSE{\frset{\Pi}{\theta}{x}}})\Loeb{\NSE{\Model}_{\theta}}(\dee x)
  =\int \SP({\NSE{\frset{\NSE{\pd{\Pi}}}{\theta}{x}}})\Loeb{\NSE{\Model}_{\theta}}(\dee x).
  \]
  Thus, we have the desired result.
\end{proof}

\section
[Proofs deferred from \cref*{SecSimpleExamples}]
{Proofs deferred from \cref{SecSimpleExamples}}
\label{AppSec5Proofs}

\subsection{Overview}

This section contains routine proofs from \cref{SecSimpleExamples}, with subsections ordered as in that section.

\subsection
[Proofs deferred from \cref*{subsecpert}]
{Proofs deferred from \cref{subsecpert}}
\label{AppSec51Proofs}

This section contains many routine proofs from \cref{subsecpert}, culminating in the proof of its main result \cref{PropContinuity1}. We begin by showing that the credibility function $ \corr{\pi}{x}$ is Lipschitz and monotone decreasing, and then characterize its right derivative:

\begin{lemma} \label{LemmaContinuityOfL} 
Let $x \in X$ and $\pi \in \Priors$. For all $0 < r_{1} < r_{2} < \infty$,
\[
\abs{ \corr{\pi}{x}(r_{2}) - \corr{\pi}{x}(r_{1}) } \leq \abs{ r_{2} - r_{1} }.
\]
In particular, $\corr{\pi}{x}$ is 1-Lipschitz and monotone over this range.
Finally, for $ 0 \leq r \leq 1$ and all $0 < c \le 1-r$, 
\[ 
\frac {\corr{\pi}{x}(r)-\corr{\pi}{x}(r+c)}{c} \geq  \corr{\pi}{x}(r).
\]
\end{lemma}

\begin{proof} 
Because probability measures are monotone, $\corr{\pi}{x}(\cdot)$ is also monotone by construction.
We can immediately calculate
\begin{align*}
\abs{ \corr{\pi}{x}(r_{2}) - \corr{\pi}{x}(r_{1}) }
&= | \Post{x}{\pi}(\max(0, \papf{\pi}{\theta}{x} - r_{1}))  \\
& \qquad - \Post{x}{\pi}(\max(0, \papf{\pi}{\theta}{x} - r_{2})) | \\
&\leq \abs{ \Post{x}{\pi}(r_{2} - r_{1}) } = \abs{ r_{2} - r_{1} },
\end{align*}
obtaining the first upper bound.

For all $c > 0$ and all $0 \leq s < t < \infty$,
\[ \label{IneqMonSlope}
\corr{\pi}{x}(s) - \corr{\pi}{x}(s+c) \geq \corr{\pi}{x}(t) - \corr{\pi}{x}(t+c).
\]
In other words, the rate at which $\corr{\pi}{x}(\cdot)$ decreases over an interval of fixed size is itself decreasing.
Next, fix $k \in \Nats$. Since $\corr{\pi}{x}(1) = 0$, \cref{IneqMonSlope} implies that,
for all $s \in [0,1]$,
\begin{align*}
\corr{\pi}{x}(s) &= \corr{\pi}{x}(s) - \corr{\pi}{x}(1) \\
&= \sum_{i=1}^{k} \parens[\Big]{ \corr{\pi}{x}\parens[\Big]{s + (1-s)\frac{i-1}{k}} - \corr{\pi}{x}\parens[\Big]{s+(1-s)\frac{i}{k}} }\\
&\leq k \parens[\Big]{ \corr{\pi}{x}(s) - \corr{\pi}{x}\parens[\Big]{s+(1-s)\frac{1}{k}} }. 
\end{align*}

Rearranging, we have 
\[
\corr{\pi}{x}(s) - \corr{\pi}{x}\parens[\Big]{ s + (1-s) \frac{1}{k} }
\geq \frac{1}{k} \corr{\pi}{x}(s)
\geq (1-s)\frac{1}{k} \corr{\pi}{x}(s).
\]
Because $\corr{\pi}{x}$ is monotone decreasing, this proves the final inequality.
\end{proof}

We now have:

\begin{proof} [of \cref{LemmaTheComplicatedThingIsACredibleSet}]

Let  $x \in X$ and $\pi \in \Priors$.
By \cref{minorlem}, we have $\corr{\pi}{x}(0) > 1-\alpha$.
It is clear that $\corr{\pi}{x}(1) = 0$.
Since $\corr{\pi}{x}$ is continuous by \cref{LemmaContinuityOfL}, this implies that there exists some $0 < r^* < 1$ so that $\corr{\pi}{x}(r^*) = 1- \alpha$ exactly.
Since $\corr{\pi}{x}$ is also monotone, there is a greatest such $r^*$, in which case $R(x,\pi) = r^*$. By the definition of $\gfinalset{\pi}{x}$, it then follows that $\Post{x}{\pi}(\gfinalset{\pi}{x}) = 1-\alpha$.
\end{proof}

Before proving the main result \cref{PropContinuity1}, we establish the continuity of the correction $R(x,\pi)$ applied to the perturbation in order to achieve the correct credibility.

\begin{lemma} \label{LemmaContinuityLAnnoying}
Under the assumptions of \cref{PropContinuity1},
\[ \label{EqMapContByMon}
(x,\pi) \mapsto R(x,\pi)
\: (X\times \Priors, \dX \otimes W_1) \to ([0,1],\norm{\cdot})
\]
is a continuous map.
\end{lemma}

\begin{proof}
Fix $\mu, \nu \in \Priors$ and $x,y \in X$.
Let $\hr{z}{\xi} = \max(0, \papf{\xi}{\cdot}{z}- r)$,
which is $\beta$-Lipschitz by \cref{perlipschitz}. 
Then $\corr{\xi}{z}(r) = \Post{z}{\xi}(\hr{z}{\xi})$.
From the triangle inequality and \cref{AssumptionModelSmoothness},
\begin{align*}
\supnorm{ \corr{\mu}{x} &- \corr{\nu}{y} } = \sup_{r} \, \abs{ \corr{\mu}{x}(r) - \corr{\nu}{y}(r) } \\
&= \sup_{r} \,\abs{ \Post{x}{\mu}(\hr{x}{\mu}) - \Post{y}{\nu}(\hr{y}{\nu}) } \\
&\leq \sup_{r} \,( \abs{ \Post{x}{\mu}(\hr{x}{\mu}) -  \Post{y}{\nu}(\hr{x}{\mu}) } 
+ \abs{ \Post{y}{\nu}(\hr{x}{\mu})  - \Post{y}{\nu}(\hr{y}{\nu}) } ) \\
&\leq W_{1}(\Post{x}{\mu}, \Post{y}{\nu}) \, \sup_{r} \,\lipnorm{ \hr{x}{\mu}} 
+ \sup_{r}\, \supnorm{ \hr{x}{\mu} - \hr{y}{\nu} }  \\
&\leq \underbrace{
        \beta \, C \, \parens{ W_{1}(\mu,\nu) + \dX(x,y) } + \supnorm{ \papf{\mu}{\cdot}{x} - \papf{\nu}{\cdot}{y} }
        }_{f(x,\mu,y,\nu)}.
\end{align*}
Thus, by the continuity hypothesis, for all $r \in (0,1)$,
\begin{align}\label{ordering}
\corr{\mu}{x}(r) + f(x,\mu,y,\nu) \ge \corr{\nu}{y}(r) \ge \corr{\mu}{x}(r) - f(x,\mu,y,\nu)
\end{align}
with 
\[\label{limitst}
\text{$f(x,\mu,y,\nu) \to 0$ as $W_1(\mu,\nu) + \dX(x,y) \to 0$.}
\]
By second part of \cref{LemmaContinuityOfL}, for all $0 < c \le 1- r$,
\[
\corr{\pi}{x}(r+c) \leq \corr{\pi}{x}(r) - c\, \corr{\pi}{x}(r)
\]
and, by rearrangement,
\[\label{clesseq}
c\leq \frac{\corr{\pi}{x}(r)-\corr{\pi}{x}(r+c)}{\corr{\pi}{x}(r)}.
\]

Let 
\[
r_{+}&=\sup \set{ r \in [0,1] \st \corr{\mu}{x}(r) + f(x,\mu,y,\nu) \geq 1 - \alpha }, \\ 
r_0&=\sup \set{ r \in [0,1] \st \corr{\mu}{x}(r) \geq 1 - \alpha }, \text{ and } \\
r_{-}&=\sup \set{ r \in [0,1] \st \corr{\mu}{x}(r) - f(x,\mu,y,\nu) \geq 1 - \alpha }.
\]
Clearly, $r_{+}\geq r_0\geq r_{-}$.
By \cref{minorlem} and \cref{LemmaContinuityOfL},
$\corr{\mu}{x}(\cdot)$ is monotone nonincreasing,
$\corr{\mu}{x}(0) \geq 1-\alpha + F'$ for some constant $F' >0$, and $\corr{\mu}{x}(1) = 0$. 
By \cref{ordering} and the definition of $R$, we have
$\abs{ R(x,\mu) - R(y,\nu) }\leq \max\{r_0-r_{-},r_{+}-r_0\}$. 
By \cref{LemmaContinuityOfL}, 
$\corr{\mu}{x}(r_0)=1-\alpha$,
$\corr{\mu}{x}(r_{-})=1-\alpha+f(x,\mu,y,\nu)$,
and 
$\corr{\mu}{x}(r_{+})=1-\alpha-f(x,\mu,y,\nu)$.
By \cref{clesseq},
\[
r_0 - r_{-}
&\leq \frac{\corr{\pi}{x}(r_{-}) - \corr{\pi}{x}(r_{-}+r_0-r_{-})}{\corr{\pi}{x}(r_{-})}\\
&=\frac{f(x,\mu,y,\nu)}{1-\alpha+f(x,\mu,y,\nu)}.
\]
Similarly, by \cref{clesseq},
\[
r_{+}-r_0
&\leq \frac{\corr{\pi}{x}(r_0)-\corr{\pi}{x}(r_0+r_{+}-r_0)}{\corr{\pi}{x}(r_0)}\\
&=\frac{f(x,\mu,y,\nu)}{1-\alpha}.
\]
Hence, 
$\abs{ R(x,\mu) - R(y,\nu) }\leq \frac{f(x,\mu,y,\nu)}{1-\alpha}$.
By \cref{limitst}, this completes the proof.
\end{proof}

We are now ready to prove \cref{PropContinuity1}.

\begin{proof}[of \cref{PropContinuity1}]
By hypothesis, 
the map
\[
(x,\mu) \mapsto \papf{\mu}{\cdot}{x}
\: (X\times\Priors, \dX \otimes W_1) \to (\CFuncs(\Theta,[0,1]), \supnorm{\cdot})
\]
is continuous.
It is also the case that the map $(f,r) \mapsto (\theta \mapsto \max \parens{0, f(\theta) - r} )$ 
from  $(\CFuncs(\Theta,[0,1]) \times [0,1], \supnorm{\cdot} \times \norm{\cdot})$ 
to $(\CFuncs(\Theta,[0,1]), \supnorm{\cdot})$
is continuous. Therefore, by 
\cref{LemmaContinuityLAnnoying} and composition,
it follows that $(x,\pi) \mapsto \parens[\big]{ \theta \mapsto \max(0, \papf{\pi}{\theta}{x} - R(x,\pi)) }$ 
from $(X\times \Priors, \dX \otimes W_1)$ to $(\CFuncs(\Theta,[0,1]), \supnorm{\cdot})$ 
is continuous.

We now show that \cref{assumptionjc} holds. 
In fact, we establish the continuity for the acceptance probability, which is clearly  continuous iff the rejection probability function is.
Fix $\theta^* \in \Theta$. 
For all $\epsilon>0$, 
by \cref{LemmaContinuityLAnnoying},
there exists $\varpi>0$ such that,
for all $(x,\pi),(x',\pi') \in X \times \Priors$, 
we have 
$\supnorm{\papf{\pi}{\cdot}{x}-\papf{\pi'}{\cdot}{x'}} < \epsilon$ 
and $\abs{R(x,\pi)-R(x',\pi')} < \epsilon$ 
if $\dX(x,x')+W_{1}(\pi,\pi')<\varpi$.
The sup-norm bound implies
$\abs{\papf{\pi}{\theta^*}{x}-\papf{\pi'}{\theta^*}{x'}} < \epsilon$. 
Noting that subtraction and maximization over a finite set are continuous completes the proof.
\end{proof}

\subsection
[Proofs deferred from \cref{secfamily}]
{Proofs deferred from \cref{secfamily}}

No proofs from \cref{secfamily} were deferred.

\subsection
[Proofs deferred from \cref*{subsecpbpd}]
{Proofs deferred from \cref{subsecpbpd}}
\label{appsecdef53}

\begin{proof} [of \cref{LemmaContinuityTildeRAnnoying}]

We note that the map $(x,\pi) \mapsto \candid{\pi}{x}$ is well-defined
if and only if
the underlying acceptance probability function $\rbcs{\Post{x}{\pt{\pi}{\gamma}}}{\alpha}{\beta}{B}$ is uniformly Lipschitz in its first parameter. By \cref{assumptionPerturb} and \cref{AssumptionModelSupNorm}, $\Post{x}{\pt{\pi}{\gamma}}$ has a Lipschitz density function.
By \cref{rcbisacredset} (for $B=I$) and  \cref{rhpdisacredset} (for $B=H$), we established that $\rbcs{\Post{x}{\pt{\pi}{\gamma}}}{\alpha}{\beta}{B}$ is $\beta$-Lipschitz for $B\in \{I,H\}$, and so the map is well-defined.

Fix $x,y \in X$ and $\mu,\nu \in \Priors$. Let $\pt{\mu}{\gamma},\pt{\nu}{\gamma}$ be the corresponding perturbed priors.
In the case $B=\tBall$, let
$\set{f_{r}^{(\mu)}}_{r >0 }$,
$\set{f_{r}^{(\nu)}}_{r > 0}$
be the functions associated with distributions $\Post{x}{\pt{\mu}{\gamma}}$ and $\Post{y}{\pt{\nu}{\gamma}}$ as in \cref{EqLevelsBlah};
in the case $B=\tHPD$, 
let 
$\set{f_{r}^{(\mu)}}_{r > 0}$,
$\set{f_{r}^{(\nu)}}_{r > 0}$
be the functions associated with distributions $\Post{x}{\pt{\mu}{\gamma}}$ and $\Post{y}{\pt{\nu}{\gamma}}$ as in \cref{EqLevelsHPDBlah}. Recall from \cref{rcbisacredset} and \cref{rhpdisacredset} that these functions are always $\beta$-Lipschitz. In both cases, 
by \cref{AssumptionModelSmoothness},
for all $0 < r < \infty$,
\begin{align*}
\int_{\Theta} f_{r}^{(\mu)}(\theta) \Post{y}{\pt{\nu}{\gamma}}(\dee \theta) 
&\geq \int_{\Theta} f_{r}^{(\mu)}(\theta) \Post{x}{\pt{\mu}{\gamma}}(\dee \theta) - \lipnorm{f_{r}^{(\mu)} } \, W_{1}(\Post{x}{\pt{\mu}{\gamma}},\Post{y}{\pt{\nu}{\gamma}}) \\
&\geq  \int_{\Theta} f_{r}^{(\mu)}(\theta) \Post{x}{\pt{\mu}{\gamma}}(\dee \theta) - \beta \, C \, (D(\gamma)\,W_{1}(\mu,\nu) + \dX(x,y) ).
\end{align*}
Similarly,
\[ \label{IneqIntWrongMsr}
\int_{\Theta} f_{r}^{(\nu)}(\theta) \Post{x}{\pt{\mu}{\gamma}}(\dee \theta) 
\geq \int_{\Theta} f_{r}^{(\nu)}(\theta) \Post{y}{\pt{\nu}{\gamma}}(\dee \theta)- \beta \, C \, (D(\gamma)\,W_{1}(\mu,\nu) + \dX(x,y) ).
\]
In particular, these two values can be made arbitrarily close by letting $W_{1}(\mu,\nu) + \dX(x,y)$ go to zero.

Main Claim: for all $\mu \in \Priors$, $x \in X$, $0 < \alpha < 1$, and $z > 0$, there exists some $\delta > 0$ so that for all $\nu \in \Priors$, $y \in Y$ such that $\| x - y \|, \, W(\mu,\nu) < \delta$, we have 
\begin{equation} \label{EqDomination}
\rbcs{\Post{x}{\pt{\mu}{\gamma}}}{\alpha}{\beta}{B}
\leq
\rbcs{\Post{y}{\pt{\nu}{\gamma}}}{\alpha-z}{\beta}{B}
\qquad \text{and} \qquad
\rbcs{\Post{y}{\pt{\nu}{\gamma}}}{\alpha}{\beta}{B}
\leq
\rbcs{\Post{x}{\pt{\mu}{\gamma}}}{\alpha-z}{\beta}{B}.
\end{equation}

\newcommand{\LL}[2]{L(#2,#1)}
We fix $\mu \in \Priors$, $x \in X$, and $0 < \alpha < 1$. To establish the main claim, we establish several sub-claims. As they are rather similar, we list them in order and then prove them.
For notational convenience, let $\LL{\mu}{\alpha}$ be  $\cbLR{\mu}$ or $\hpdLR{\mu}$ as needed.

Numbered Sub-Claims:
\begin{enumerate}
\item First consider the case $B=I$. For all $\epsilon > 0$, there exists $\delta > 0$ such that for all $\nu \in \Priors$, $y \in Y$ such that $\| x - y \|, \, W(\mu,\nu) < \delta$, we have for all $r > 0$ \[\label{IneqNSC1P1}
f_{r - \epsilon}^{\nu} \leq f_{r}^{\mu} \leq f_{r + \epsilon}^{\nu}
\]
and similarly
\[\label{IneqNSC1P2}
f_{r - \epsilon}^{\mu} \leq f_{r}^{\nu} \leq f_{r + \epsilon}^{\mu},
\]
where in both expressions we should take $f_{r} = f_{0}$ in the degenerate case $r < 0$. In the case $B = H$, the same claim is true with all inequalities in \cref{IneqNSC1P1} and \cref{IneqNSC1P2} reversed.

\item  For all $z \geq 0$ and $\epsilon > 0$, there exists $\delta > 0$ such that for all $\nu \in \Priors$, $y \in Y$ such that $\| x - y \|, \, W(\mu,\nu) < \delta$, we have
\[ \label{IneqNumberedClaim2P1}
\LL{\Post{y}{\pt{\nu}{\gamma}}}{\alpha - z} + \epsilon 
      \geq \LL{\Post{x}{\pt{\mu}{\gamma}}}{\alpha-z} 
      \geq \LL{\Post{y}{\pt{\nu}{\gamma}}}{\alpha - z} - \epsilon.
\]
This is true for both $B=I$ and $B=H$.

\item Fix $B=I$. For all $z > 0$, there exists $\epsilon_{0} > 0$ such that for all $ 0 < \epsilon < \epsilon_{0}$, there exists $\delta > 0$ such that for all $\nu \in \Priors$, $y \in Y$ such that $\| x - y \|, \, W(\mu,\nu) < \delta$, %
\[ \label{IneqLLevelAdj1}
\LL{\Post{y}{\pt{\nu}{\gamma}}}{\alpha -z} 
   \geq \LL{\Post{x}{\pt{\mu}{\gamma}}}{\alpha} + \epsilon
\]
and similarly
\[\label{IneqLLevelAdj2}
\LL{\Post{x}{\pt{\mu}{\gamma}}}{\alpha -z} 
     \geq \LL{\Post{y}{\pt{\nu}{\gamma}}}{\alpha} + \epsilon.
\]
In the case $B=H$, the same claim is true with \cref{IneqLLevelAdj1} and \cref{IneqLLevelAdj2} replaced by
\[ 
\LL{\Post{y}{\pt{\nu}{\gamma}}}{\alpha -z} 
   \leq \LL{\Post{x}{\pt{\mu}{\gamma}}}{\alpha} - \epsilon
\]
and similarly
\[
\LL{\Post{x}{\pt{\mu}{\gamma}}}{\alpha -z} 
      \leq \LL{\Post{y}{\pt{\nu}{\gamma}}}{\alpha} - \epsilon.
\]

\end{enumerate}

We now check our numbered sub-claims in the same order.

\begin{enumerate}
\item In the case $B=I$, this follows immediately from \cref{LemmaSmoothDensitiesMeans}.
In the case $B=H$, this follows immediately from \cref{AssumptionModelSupNorm}; the conditions are satisfied by \cref{AssumptionModelSupNorm,rhpdisacredset}. 

\item We consider only the case $B=I$; the case $B=H$ is essentially the same, with some inequalities flipped. We begin by assuming the first inequality in \cref{IneqNumberedClaim2P1} does not hold, so that there exist some $z \geq 0$, $\epsilon > 0$ and some $c > 0$  so that for all $\delta > 0$, there exists a measure $\nu = \nu_{\delta} \in \Priors$ and data point $y = y_{\delta} \in X$ with $\| x - y\|, \, W(\mu,\nu) < \delta$ such that
\[
\LL{\Post{y}{\pt{\nu}{\gamma}}}{\alpha - z} + \epsilon + c 
     \leq \LL{\Post{x}{\pt{\mu}{\gamma}}}{\alpha - z}.
\]
For the same families of measures and points $\nu = \nu_{\delta}, y = y_{\delta}$ indexed by $\delta > 0$, then, there must  exist some $c' > 0$ so that for all $\delta > 0$, 
\[ \label{IneqSubclaim2ThingToCont}
\Post{x}{\pt{\mu}{\gamma}}(f_{\LL{\Post{y}{\pt{\nu}{\gamma}}}{\alpha-z} + \epsilon}^{\mu}) + c' 
     \leq \Post{x}{\pt{\mu}{\gamma}}(f_{\LL{\Post{x}{\pt{\mu}{\gamma}}}{\alpha-z}}^{\mu}) = 1 - (\alpha -z).
\]
By numbered sub-claim \textbf{(1)}, however, we have the pointwise inequality
\[
f_{\LL{\Post{y}{\pt{\nu}{\gamma}}}{\alpha-z} + \epsilon}^{\mu} \geq f_{\LL{\Post{y}{\pt{\nu}{\gamma}}}{\alpha-z}}^{\nu}
\]
for all $\delta > 0$ sufficiently small. Applying this pointwise inequality, and then \cref{IneqIntWrongMsr}, 
we have for all $\delta > 0$ sufficiently small
\begin{align*}
\Post{x}{\pt{\mu}{\gamma}}(f_{\LL{\Post{y}{\pt{\nu}{\gamma}}}{\alpha-z} + \epsilon}^{\mu}) 
      &\geq \Post{x}{\pt{\mu}{\gamma}}(f_{\LL{\Post{y}{\pt{\nu}{\gamma}}}{\alpha-z}}^{\nu}) \\
&\geq \Post{y}{\pt{\nu}{\gamma}}(f_{\LL{\Post{y}{\pt{\nu}{\gamma}}}{\alpha-z}}^{\nu}) -  \beta \, C \, (D(\gamma)\,W_{1}(\mu,\nu) + \dX(x,y) ) \\
&= 1- (\alpha-z) - \beta \, C \, (D(\gamma)\,W_{1}(\mu,\nu) + \dX(x,y) ).
\end{align*}
But for $\delta > 0$ sufficiently small, this contradicts \cref{IneqSubclaim2ThingToCont}.
This completes the proof of the first inequality in \cref{IneqNumberedClaim2P1}, in the case $B=I$.
The proof of the second inequality in \cref{IneqNumberedClaim2P1}, and the case $B=H$, are essentially identical.

\item We consider the case $B=I$. By \cref{rcbisacredset}, we can write
\[\label{eq1toadj2}
\LL{\Post{x}{\pt{\mu}{\gamma}}}{\alpha  - z} = \LL{\Post{x}{\pt{\mu}{\gamma}}}{\alpha} + C
\]
for some $C$ that is \emph{strictly} greater than $0$. 
By numbered sub-claim \textbf{(2)}, for all $\delta>0$ sufficiently small we also have 
\[
\LL{\Post{y}{\pt{\nu}{\gamma}}}{\alpha-z} 
    \geq \LL{\Post{x}{\pt{\mu}{\gamma}}}{\alpha-z} - \frac{C}{2}.
\]
Putting the two previous equations together, we have 
\[
\LL{\Post{y}{\pt{\nu}{\gamma}}}{\alpha-z} 
    \geq \LL{\Post{x}{\pt{\mu}{\gamma}}}{\alpha} + \frac{C}{2}.
\]
This establishes \cref{IneqLLevelAdj1}. We now establish \cref{IneqLLevelAdj2}.
By numbered sub-claim \textbf{(2)}, for all $\delta>0$ sufficiently small we also have
\[\label{eq2toadj2}
\LL{\Post{y}{\pt{\mu}{\gamma}}}{\alpha} 
     \geq \LL{\Post{x}{\pt{\nu}{\gamma}}}{\alpha} - \frac{C}{2}.
\]
Putting together \cref{eq1toadj2} and \cref{eq2toadj2}, we have
\[
\LL{\Post{y}{\pt{\mu}{\gamma}}}{\alpha-z} 
      \geq \LL{\Post{x}{\pt{\nu}{\gamma}}}{\alpha} + \frac{C}{2}.
\]
This establishes \cref{IneqLLevelAdj2}. The case $B=H$ is the same. 
\end{enumerate}

Finally, the main claim follows immediately from numbered sub-claims \textbf{(1)}  and \textbf{(3)}.
The result now follows almost immediately from \cref{EqDomination}. Recall from \cref{tfrsetdefn} that 
\[
\papf{\Post{x}{\pt{\mu}{\gamma}}}{\theta}{x}  
= \int_{[\delta \eta,\delta]} \rbcs{\Post{x}{\pt{\mu}{\gamma}}}{\alpha - z}{\beta}{B} \,\pb{\eta}^{(\delta)}(z) \,\dee z.
\]

By \cref{EqDomination}, for all $\epsilon > 0$ there exists $\delta_{0} > 0$ 
so that for $\nu \in \Priors$, $y \in Y$ such that $\| x - y \|, \, W(\mu,\nu) < \delta_{0}$, we have 
\begin{align*}
\papf{\Post{y}{\pt{\nu}{\gamma}}}{\theta}{x} &=  \int_{[\delta \eta,\delta]} \rbcs{\Post{y}{\pt{\nu}{\gamma}}}{\alpha-z}{\beta}{B} \,\pb{\eta}^{(\delta)}(z)\,\dee z \\
&\geq  \int_{[\delta \eta,\delta]} \rbcs{\Post{x}{\pt{\mu}{\gamma}}}{\alpha-z - \epsilon}{\beta}{B} \,\pb{\eta}^{(\delta)}(z) \,\dee z\\
&\geq  \int_{[\delta \eta+\epsilon,\delta-\epsilon]} \rbcs{\Post{x}{\pt{\mu}{\gamma}}}{\alpha-z}{\beta}{B} \,\pb{\eta}^{(\delta)}(z-\epsilon) \,\dee z.\\
&\geq \int_{[\delta \eta,\delta]} \rbcs{\Post{x}{\pt{\mu}{\gamma}}}{\alpha-z}{\beta}{B} \,\pb{\eta}^{(\delta)}(z) \,\dee z \\
& \qquad - \int_{[\delta \eta,\delta]} \rbcs{\Post{x}{\pt{\mu}{\gamma}}}{\alpha-z}{\beta}{B} \,|\pb{\eta}^{(\delta)}(z) - \pb{\eta}^{(\delta)}(z-\epsilon) | \,\dee z\\
& \qquad - \int_{[\delta \eta, \delta \eta + \epsilon] \cup[\delta-\epsilon,\delta]} \pb{\eta}^{(\delta)}(z-\epsilon) \,\dee z \\
&= \papf{\Post{x}{\pt{\mu}{\gamma}}}{\theta}{x} \\
& \qquad -  \int_{[\delta \eta,\delta]} \rbcs{\Post{x}{\pt{\mu}{\gamma}}}{\alpha-z}{\beta}{B} \,|\pb{\eta}^{(\delta)}(z) - \pb{\eta}^{(\delta)}(z-\epsilon) | \,\dee z \\
&\qquad - \int_{[\delta \eta, \delta \eta + \epsilon] \cup[\delta-\epsilon,\delta]} \pb{\eta}^{(\delta)}(z-\epsilon) \,\dee z,
\end{align*}
where in this calculation we freely use the fact that $0 \leq \rbcs{\Post{x}{\pt{\mu}{\gamma}}}{\alpha}{\beta}{B} \leq 1$ holds pointwise. Since $\pb{\eta}$ is continuous (hence uniformly continuous on compact sets), this shows that for all $\epsilon' > 0$, there exists $\delta' > 0$ so that for all $\nu \in \Priors$, $y \in Y$ such that $\| x - y \|, \, W(\mu,\nu) < \delta'$, we have 
\[
\papf{\Post{y}{\pt{\nu}{\gamma}}}{\theta}{x}  \geq \papf{\Post{x}{\pt{\mu}{\gamma}}}{\theta}{x} - \epsilon'.
\]

The proof of the reverse inequality is essentially the same, and completes the proof of the lemma.
\end{proof}

\subsection[Proofs deferred from \cref*{SubsecSupport}]
{Proofs deferred from \cref{SubsecSupport}}
\label{appsecdef54}

\begin{proof}[of \cref{basicsupportresultlemma}]

By \cref{StartingHighLemma},
$
\gfinalset{\pi}{x}
\leq
\papf{\pi}{\cdot}{x}$.
By \cref{tfrsetdefn},
\[
\papf{\pi}{\cdot}{x}
\leq
\sup_{z \in [ \delta \eta,\delta]}
     \bapf[\alpha-z]{\beta}{\pt{\pi}{\gamma}}{\cdot}{x}
      .
\]
Thus, by \cref{nestedass}, for every $\epsilon \in [\delta,\alpha)$,
\[
\papf{\pi}{\cdot}{x}
\leq
     \bapf[\alpha-\epsilon]{\beta}{\pt{\pi}{\gamma}}{\cdot}{x}
.
\]
Thus, the support of
$\gfinalset{\pi}{x}$
is contained in the support of
$     \bapf[\alpha-\epsilon]{\beta}{\pt{\pi}{\gamma}}{\cdot}{x}$
 for all $\epsilon \in [\delta,\alpha)$.
\end{proof}

\begin{proof}[of \cref{ballcontain}]
By \cref{DefUsualInterval}, the support of $\credball{\mu}{\alpha}{}$ is 
  \[
  S_{0}=\{\theta \in \fTheta: \norm{\theta - M(\mu)} \leq \cbL{\mu} \},
  \]
  where the reader may recall from \cref{cbLdefn} that
  $\cbL{\mu} \defas \inf R_0$ for
  \begin{equation}
    R_0 = \{r>0: \mu(\{\theta\in \fTheta: \norm{\theta-M(\mu)} \le r  \}  )\geq 1-\alpha\}.
  \end{equation}
  The support of
$\rcb{\mu}{\alpha}{\beta}$
is $S_{1}=\{\theta \in \fTheta :  \norm{\theta-M(\mu)} \leq r_{1}\beta^{-1}\}$ 
where
\[
r_1=\cbLR{\mu} \defas \inf \, R_1
\]
for
$R_1 = \{r > 0 \st \int_{\fTheta} f_{r}^{(I)} \,\dee\mu \geq  1- \alpha \}$.

We first show that the support of $\credball{\mu}{\alpha}{}$ is contained in the support of $\rcb{\mu}{\alpha}{\beta}$.
Pick $\theta_{0}\in S_{0}$. We want to show that $\norm{\theta_{0}-M(\mu)} \leq r_{1}\beta^{-1}$. 
It suffices to show that $\cbL{\mu} \leq r_{1}\beta^{-1}$. Pick some $r\in R_1$. 
It is sufficient to show 
\[
\mu(\{\theta\in \fTheta: \norm{\theta-M(\mu)} \le r\beta^{-1}\})\geq 1-\alpha. 
\]
Define $v : \fTheta\to \Reals$ by
\begin{equation*}
v(\theta) \defas
\begin{cases}
1, & \norm{ \theta - M(\mu)} \leq r\beta^{-1}, \\
0, & \text{otherwise.}
\end{cases}
\end{equation*}
Note that $f_{r}^{I}(\theta)\leq v(\theta)$ for every $\theta\in \fTheta$. As $\int_{\fTheta} f_{r}^{I} \dee \mu\geq 1-\alpha$, we have
\[
\mu(\{\theta\in \fTheta: \norm{\theta-M(\mu)} \le r\beta^{-1}\})\geq 1-\alpha.
\]

We now show that the support of
$\rcb{\mu}{\alpha}{\beta}$
is contained in the $\beta^{-1}$-fattening of $E_0$.
Pick $\theta_1\in S_1$.
To finish the proof, it suffices to find a $\theta\in E_0$ such that $\norm{\theta-\theta_1}<\beta^{-1}$.
The result is immediate if $r_1<1$. Suppose $r_1\geq 1$.
Let $\theta_0\in \Reals^{d}$ be any point on the line segment connecting $M(\mu)$ and $\theta_1$ satisfying
$\norm{\theta_0-M(\mu)} \le r_{1}\beta^{-1}-\beta^{-1}$ and $\norm{\theta_1-\theta_0}\leq \beta^{-1}$.
We now show that $\theta_0\in E_0$, completing the proof.

Pick $r_0 \in R_0$.
It suffices to show that $r_1\beta^{-1}-\beta^{-1} \le r_0$, which is equivalent to $r_1\le \beta r_0+1$.
As $f_{r}^{(I)}(\cdot)$ is an increasing function of $r$, 
it suffices to show that $\int_{\fTheta} f_{\beta r_0+1}^{(I)} \,\dee\mu \geq  1- \alpha $.
Define $u : \fTheta\to \Reals$ by
\begin{equation*}
u(\theta) \defas
\begin{cases}
1, & \norm{ \theta - M(\mu) } \leq r_0, \\
0, & \text{otherwise.}
\end{cases}
\end{equation*}
Note that $f_{\beta r_0+1}^{(I)}(\theta)=1$ whenever $u(\theta)=1$.
Hence $f_{\beta r_0+1}^{(I)}(\theta)\geq u(\theta)$ for every $\theta\in \fTheta$.
As $r_0 \in R_0$,
we know that $\int_{\fTheta} u \,\dee\mu \geq  1- \alpha $.
Thus we have $\int_{\fTheta} f_{\beta r_0+1}^{(I)} \,\dee\mu \geq  1- \alpha $, completing the proof.
\end{proof}

\begin{proof} [of \cref{ballsupportresult}]

By \cref{Rbound} and the definition of $\papf{\pi}{\cdot}{x}$, to finish the proof,
it suffices to show that the support of $\credset{I}{\Post{x}{\pt{\pi}{\gamma}}}{\alpha}{}$ 
is contained in the $\beta^{-1}$-fattening of the support of $\max\{0,\rcb{\Post{x}{\pt{\pi}{\gamma}}}{\alpha}{\beta}-\alpha\}$.

By \cref{DefUsualInterval}, the support of $\credball{\Post{x}{\pt{\pi}{\gamma}}}{\alpha}{}$ is
\[
S_{0}=\{\theta \in \fTheta: \norm{\theta - M(\Post{x}{\pt{\pi}{\gamma}})} < \cbL{\Post{x}{\pt{\pi}{\gamma}}} \},
\]
where the reader may recall from \cref{cbLdefn} that
$\cbL{\Post{x}{\pt{\pi}{\gamma}}} \defas \inf R_0$ for
\begin{equation}
R_0 = \{r>0: \Post{x}{\pt{\pi}{\gamma}}(\{\theta\in \fTheta: \norm{\theta-M(\Post{x}{\pt{\pi}{\gamma}})} \le r  \}  )\geq 1-\alpha\}.
\end{equation}
The support of
$\rcb{\Post{x}{\pt{\pi}{\gamma}}}{\alpha}{\beta}$
is 
\[
S_{1}=\{\theta \in \fTheta :  \norm{\theta-M(\Post{x}{\pt{\pi}{\gamma}})}<r_{1}\beta^{-1}\},
\]
where $r_1=\cbLR{\Post{x}{\pt{\pi}{\gamma}}} \defas \inf R_1$ for
\[
R_1 = \{r > 0 \st \int_{\fTheta} f_{r}^{(I)} \,\dee\Post{x}{\pt{\pi}{\gamma}} \geq  1- \alpha \}.
\]

Pick $\theta_0\in S_0$. 
It suffices to show there exists $\theta_{1}$ with $\norm{\theta_{0}-\theta_{1}}\leq \beta^{-1}$ and
$\rcb{\Post{x}{\pt{\pi}{\gamma}}}{\alpha}{\beta}(\theta_{1})-\alpha>0$,
or equivalently, $\norm{\theta_{1}-M(\Post{x}{\pt{\pi}{\gamma}})}<r_{1}\beta^{-1}-\alpha\beta^{-1}$.
It suffices to show that $\norm{\theta_{0}-M(\Post{x}{\pt{\pi}{\gamma}})}<r_{1}\beta^{-1}-\alpha\beta^{-1}+\beta^{-1}$.
As $0<\alpha<1$,
it suffices to show that $\norm{\theta_{0}-M(\Post{x}{\pt{\pi}{\gamma}})}<r_{1}\beta^{-1}$.
But this follows from the first statement of \cref{ballcontain}.
\end{proof}

\begin{proof}[of \cref{relaxedhpddensprop}]

We use notation from \cref{DefUsualHPD} and \cref{DefSimpleHPDCredInt}.
We begin by observing that, for all $d \geq 0$ and $\theta \in \Theta$, 
we have the pointwise inequality 
\[
\onething_{\rho(\theta) \geq d} \leq f_{d}^{(H)}(\theta) \leq \onething_{\rho(\theta) \geq d- \hat{\beta}^{-1}}.
\] 

This implies the corresponding inequality for the levels
\[\label{IneqLevelIneq}
\hpdL{\mu} \leq  \hpdLR{\mu} \le \hpdL{\mu} + \hat{\beta}^{-1}.
\]
Inspecting \cref{DefUsualHPD} and \cref{DefSimpleHPDCredInt}, we note that 
\[
\essinf_{\theta \in A} \rho(\theta) = \hpdL{\mu}
\]
and 
\[
\essinf_{\theta \in A_{\beta}} \rho(\theta) \in [\hpdLR{\mu} - \hat{\beta}^{-1}, \hpdLR{\mu}]. 
\]
Combining these bounds with the bound in \cref{IneqLevelIneq} completes the proof.
\end{proof}

\section
[Proofs deferred from \cref*{SecAppBern}]
{Proofs deferred from \cref{SecAppBern}}
\label{secbernproof}

\begin{proof}[of \cref{LemmaZInjBern}]
We prove the stronger statement: there exists a constant $D > 0$ such that
\[ \label{IneqActualContainment}
\Q{S}(\PM{S}) \subseteq \bigcap_{0 <a,b < D} \Restrict{\AGD{a}{\frac{b}{2}}}{S}.
\] 
This immediately implies the statement of the lemma.
Roughly speaking, \cref{IneqActualContainment} implies that $\Q{S}(\pi)$ is never ``too concentrated" in a very small interval near 0 and 1.
Our proof consist of two parts: first, we show that $\Q{S}(\pi)$ is never \emph{much more} concentrated than $\pi$ in \emph{any} interval; next, we show that $\Q{S}(\pi)$ is, in fact, very uniform if $\pi$ is heavily concentrated near 0 and 1.
The proofs for the three cases are nearly identical; we specify the single major difference where it occurs.

Checking that $\Q{S}(\pi)$ is not too much more concentrated than $\pi$ is straightforward.
Inspecting the definition of $\Q{S}$, we have 
\[
\Q{S}(\pi)(s) \leq \frac{1}{2}(\pi(s) + 1)
\]
for all $s \in S$.
In particular, for all $U \subseteq S$,
\[ \label{IneqNoFastConc}
\Q{S}(\pi)(U) \leq 1 - \frac{1}{2}(1 - \pi(U)).
\]

The other direction is slightly longer.
We fix some $a,b > 0$, put $c = \frac{\alpha}{1000}$, and consider a measure $\pi \in \PM{S} \setminus \Restrict{\AGD{a}{b}}{S}$.
Without loss of generality, assume that $\pi([0,a]\cap S) \geq \pi([1-a,1]\cap S)$.
We note that the posterior having observed zero, $\Post{0}{\pi}$, assigns most of its mass to $[0,a]\cap S$.
In particular, by monotonicity of the function $\cd(\cdot,0)$, we have 
\[
\frac{\Post{0}{\pi}([0,a]\cap S)}{\Post{0}{\pi}((a,1-a)\cap S)} \geq \frac{\pi([0,a]\cap S)}{\pi((a,1-a)\cap S)} \geq \frac{0.5(1-b)}{b}
\]
and 
\[
\frac{\Post{0}{\pi}([0,a]\cap S)}{\Post{0}{\pi}([1-a,1]\cap S)} \geq \frac{1-a}{a} \frac{\pi([0,a]\cap S)}{\pi([1-a,a]\cap S)} \geq \frac{1-a}{a}.
\]
Rearranging these two bounds, we have
\[
\frac{2b}{1-b} \Post{0}{\pi}([0,a]\cap S) \geq  \Post{0}{\pi}((a,1-a)\cap S)
\]
and 
\[
\frac{a}{1-a}\Post{0}{\pi}([0,a]\cap S)  \geq \Post{0}{\pi}([1-a,1]\cap S).
\]
Summing,
\[
(\frac{a}{1-a} + \frac{2b}{1-b}) \Post{0}{\pi}([0,a]\cap S) 
&\geq \Post{0}{\pi}((a,1-a)\cap S) + \Post{0}{\pi}([1-a,1]\cap S) 
\\&= 1 - \Post{0}{\pi}([0,a]\cap S),
\]
and so rearranging again gives
\[\label{IneqRightBeforeDiffCaseBern}
\Post{0}{\pi}([0,a]\cap S) \geq \frac{1}{1 + \frac{a}{1-a} + \frac{2b}{1-b}}.
\] 

For the next short part of the proof, our three cases require slightly different arguments.
We begin by considering the first case: credible regions given by \cref{DefUsualInterval}.

Using notation as in \cref{DefUsualInterval}, Inequality~\eqref{IneqRightBeforeDiffCaseBern} implies that, for all $a,b <C_{0}$ sufficiently small,
\[ \label{IneqMeanBern}
M(\Post{0}{\pi}) \leq a + 2b + c.
\]
When this happens, the ball of radius $(a + 2b + c)$ around $M(\Post{0}{\pi})$ contains the interval $[0,a]$; thus, for $a,b < \min(C_{0}, c)$ (so that $a + 2b + c < \alpha$), we also have
\[ \label{IneqLevelBern}
\cbL{\Post{0}{\pi}} \leq M(\Post{0}{\pi}) \leq a + 2b + c.
\]
But this implies that the associated rejection function satisfies
\[ 
\frset{\pi}{\theta}{0} = 1 
\]
for all $\theta \in S \cap [2a + 4b + 2c,0.5]$. 

Let $C_{1} = 0.01$; for $a,b < \min (C_{0}, C_{1},c) \equiv C_{2}$, this implies 
\[ \label{InMainConcAllThree}
\frset{\pi}{\theta}{0} = 1 
\]
for all $\theta \in S \cap [0.05,0.5]$.

We now check that \cref{InMainConcAllThree} holds in the other two cases, possibly for different values of $C_{0},C_{1}, C_{2}$.
In the second case, this follows immediately from Inequalities~\eqref{IneqMeanBern} and \eqref{IneqLevelBern} along with \cref{ballcontain}; in the third case, it follows from  Inequalities \eqref{IneqMeanBern} and \eqref{IneqLevelBern} along with \cref{basicsupportresult}.

We now continue from Equality~\eqref{InMainConcAllThree} in all three cases.
From this equality and our assumption on $\FS$, we have
\[ \label{IneqManyBig}
|S \cap [0.05,0.5]|  \geq 0.44 |S| > 0.
\] 
Next, note that  $\cd(\theta,0) \geq \frac{1}{2}$ for all $\theta \in S \cap [0.05,0.5]$.
Thus, for such $\theta$,
\[ \label{Conc1}
z_{\theta}(\pi) \geq \frac{1}{2}(1-\alpha).
\] 

Let $\pi_{\max} = \max_{\theta \in S} \pi(\theta)$.
Combining Inequalities~\eqref{IneqManyBig} and \eqref{Conc1} the formula for $\Q{S}(\pi)$ immediately gives, for all $\theta \in S$,
\[ \label{IneqStrongUnif}
\Q{S}(\pi)(\theta) \leq \frac{\pi_{\max}+1}{ \pi_{\max} + \frac{1}{2}(1-\alpha)|S \cap [0.05,0.5] | } \leq \frac{8}{(1-\alpha) |S|}
\]
for appropriate sufficiently small $a,b < C_{2}$.

To complete the proof, consider $ 0 < a,b < C_{2}$.
If $\pi \in \PM{S} \cap \Restrict{\AGD{a}{b}}{S}$, then Inequality~\eqref{IneqNoFastConc} implies $\Q{S}(\pi) \in \Restrict{\AGD{a}{\frac{1}{2}b}}{S}$.
If $\pi \in \PM{S} \setminus \Restrict{\AGD{a}{b}}{S}$, then Inequality~\eqref{IneqStrongUnif}  and the fact that $a,b < 0.001$ imply $\Q{S}(\pi) \in \Restrict{\AGD{a}{b}}{S}$.
Thus, for $\pi\in \PM{S}$, we have $\Q{S}(\pi) \in \Restrict{\AGD{a}{\frac{1}{2}b}}{S}$.
This completes the proof of Inequality~\eqref{IneqActualContainment} and thus the lemma.
\end{proof}

\begin{proof} [of \cref{LemmaContPriorPosteriorBern}]
We note that the conditional density $\cd$ is uniformly continuous and uniformly bounded away from infinity, but not uniformly bounded away from 0. The proof of \cref{implication4} applies here with the exception of one implication: since $\cd$ is not uniformly bounded away from 0, there may exist $\Post{x}{\pi}$-continuity sets that are not $\pi$-continuity sets. We give here the resolution to this difficulty, continuing from the sentence ``Let $B$ be a $\Post{x}{\pi}$-continuity set."

We note that, since $\pi \in \AGD{a}{b}$, $\Phi(\Theta) \geq b a (1-a) > 0$. Thus, since $B$ is a $\Post{x}{\pi}$-continuity set, it is also a $\Phi$-continuity set. Next, assume WLOG that $x=1$. For $x=1$ and fixed $\epsilon > 0$, $\log(q)$ is uniformly bounded on $[\epsilon,1]$. Thus, if $B$ is a $\Phi$-continuity set, then by the argument in the proof of \cref{implication4} it is also a $\Phi$-continuity set for the restriction of $\pi$ to $[\epsilon,1]$. Thus, for all $\epsilon > 0$, 
\[ \label{SillyBernCont1}
\Phi_{n}([\epsilon,1]\cap B) \mapsto \Phi([\epsilon,1] \cap B).
\]

On the other hand, $\cd(\theta,1) \leq \epsilon$ for $\theta \in [0,\epsilon]$. Thus, for all $n \in \Nats$,
\[ \label{SillyBernCont2}
\Phi_{n}([0,\epsilon) \cap B) \leq \epsilon \pi([0,\epsilon) \cap B) \leq \epsilon.
\]
Combining \cref{SillyBernCont1} and \cref{SillyBernCont2}, we see that 
\[
\Phi_{n}(B) \mapsto \Phi(B).
\]
That is, $\Phi_{n}$ converges weakly to $\Phi$. On the other hand, since $\pi \in \AGD{a}{b}$, we have already seen that $\Phi(\Theta) \geq b a (1-a) > 0$. 
Thus, $\Post{\pi_{n}}{1} \mapsto \Post{\pi}{1}$.
The same argument gives the same convergence guarantee for $x=0$.

Having shown that 
$\Post{\pi_{n}}{x} \mapsto \Post{\pi}{x}$, 
the remainder of the proof continues as in the proof of \cref{implication4} with $M$ replaced by
$\max_{\theta \in [0,1]} \log(\cd(x,\theta)) < \infty$ where it appears. 
\end{proof}

\begin{proof} [of \cref{LemmaTrivContAGD}]
Let $\mu$ be in the $\epsilon$-fattening of $\AGD{a}{b}$, and let $\nu \in \AGD{a}{b}$ be within distance $\epsilon$ of $\mu$. Then it is possible to couple $X \sim \mu$, $Y \sim \nu$ so that $\expect[ \| X - Y \|] \leq \epsilon$. For $c > 0$, define $I(c) = [0,c] \cup [1-c,1]$. Using Markov's inequality in the last line, we compute:
\begin{align*}
\mu(I(\frac{a}{2})) &= \mathbb{P}[X \in I(\frac{a}{2})] \\
&\leq \mathbb{P}[Y \in I(a)] + \mathbb{P}[\|X-Y\| \geq \frac{a}{2}] \\
&\leq (1-b) + \frac{2 \epsilon}{a}.
\end{align*}
\end{proof}

\section
{Three Concrete Examples}
\label{SecAppExamples}

\subsection{Bernoulli Example}
\label{SecAppExamplesBernoulli}

In this section, we provide a rigorous and detailed proof that a particular prior is matching in a simple Bernoulli model. This example appears as \cref{ExBernoulliMatching} in the main paper.

Recall the usual Bernoulli model with sample space $X = \{0,1\}$, parameter space $\Theta = [0,1]$ and conditional density $\cd(\theta,x) = \theta^{x} (1 - \theta)^{1-x}$. For $\alpha \in (0,0.5)$, define the probability measure $\pi_{\alpha} = \mathrm{Unif}(\{\alpha,1-\alpha\})$.
Consider the set estimator
$\tau: X \to \BorelSets{\Theta}$ given by $\tau(0) = [0,1-\alpha)$ and $\tau(1)=(\alpha,1]$. (We do not need random sets here.)

\begin{theorem} \label{ThmNewBernEx1}
The prior $\pi_{\alpha}$ is matching prior at level $1-\alpha$ for the set estimator $\tau$. 
\end{theorem}

\begin{proof}
By Bayes rule, the posterior probability on $\theta = \alpha$ given the observation $x=0$ is
\begin{align*} 
\Post{0}{\pi_{\alpha}}(\{\alpha\}) 
&= \frac{P_\alpha \{0\}\, \pi_\alpha\{\alpha\} }
           {P_\alpha \{0\}\, \pi_\alpha\{\alpha\}  + P_{1-\alpha} \{0\}\, \pi_\alpha\{1- \alpha\} } \\
&= \frac{(1-\alpha)}{(1-\alpha) + \alpha} \\
&= 1-\alpha. 
\end{align*} 
Since the prior is concentrated on the two points $\{\alpha,1-\alpha\}$, the full posterior, given $X=0$, is
\[ 
\Post{0}{\pi_{\alpha}} = (1-\alpha) \delta_{\alpha} + \alpha \delta_{1-\alpha}.
\] 
Similarly, the full posterior, given $X=1$, is
\[
\Post{1}{\pi_{\alpha}} = \alpha \delta_{\alpha} + (1-\alpha) \delta_{1-\alpha}.
\] 

We are now in a place to establish that $\tau$ is a $1-\alpha$ credible region with respect to $\tau$. In particular,
\[ 
\Post{0}{\pi_{\alpha}}(\tau(0)) 
= \Post{0}{\pi_{\alpha}}([0,1-\alpha)) 
= \Post{0}{\pi_{\alpha}}(\{\alpha\}) = 1-\alpha 
\] 
and similarly for $\Post{1}{\pi_{\alpha}}(\tau(1))$.
To establish the matching property, it remains to verify that $\tau$ is a confidence set at level $1-\alpha$. To do so, it must hold that, 
for all $\theta \in [0,1]$,
\[
  A_{\theta}
  \defas P_{\theta} (\set{ x \in X \st \theta \in \tau(x) })
  \ge 1-\alpha.
  \]
There are three cases:
\begin{enumerate}
    \item If $\theta \in (\alpha,1-\alpha)$, then $\theta \in \tau(0) \cap \tau(1)$, so $A_\theta = 1$. 
    \item  If $\theta \in [0,\alpha]$, then $\theta \in \tau(0)$ but $\theta \not\in \tau(1)$. Thus, $A_\theta = P_\theta \{0 \} \ge 1-\alpha$.
    \item  If $\theta \in [1-\alpha,1]$, then $\theta \in \tau(1)$ but $\theta \not\in \tau(0)$. Thus, $A_\theta = P_\theta \{1 \} \ge 1-\alpha$.
\end{enumerate}

\end{proof}

\subsection{Binomial Example}
\label{SecAppExamplesBinomial}

Let $n\ge 1$ be a positive integer and 
consider the binomial model with $n$ independent trials with an unknown success probability $\theta$, i.e., the model with sample space $X = \{0,1,2,\dots,n\}$, parameter space $\Theta = [0,1]$, and conditional density $\cd(\theta,x) = {n \choose x} \theta^{x} (1 - \theta)^{n-x}$ with respect to counting measure. 

Put $\theta_0=0$ and let $\theta_1,\dots,\theta_{n} \in (0,1/2)$, with $\theta_0 < \theta_1 < \dotsm < \theta_{n}$, 
and consider the set estimator $\tau : X \to \BorelSets{\Theta}$ given by
\[
\tau(x) &= [\theta_{x},1-\theta_{n-x}), \text{ for $x=0$,}\\
\tau(x) &= (\theta_{x},1-\theta_{n-x}), \text{ for $x=1,\dots,n-1$,}\\
\tau(x) &= (\theta_{x},1-\theta_{n-x}], \text{ for $x= n$.}
\]
We first derive conditions on $\theta_1,\dots,\theta_n$ such that $\tau$ is a confidence set at level $1-\alpha$. 
After that point, we will derive a prior $\pi_{\alpha}$ such that $\pi_{\alpha}$ is a matching prior for $\tau$ at level $1-\alpha$.

To begin, recall that $\tau$ is a confidence set at level $1-\alpha$ iff,
for all $\theta \in [0,1]$,
\[
  A_{\theta}
  \defas P_{\theta} (\set{ x \in X \st \theta \in \tau(x) })
  \ge 1-\alpha.
  \]
By design, 
$A_0 = P_{0} \{0\} = 1 \ge 1 - \alpha$ and 
$A_1 = P_{1} \{n\} = 1 \ge 1- \alpha$. 
It remains to check coverage for $\theta \in (0,1)$.
Let $F_\theta(k) = P_\theta(\{x \in X : x \le k\})$ and let $\bar F_\theta(k) = P_{\theta} (\{x \in X : x \ge k \})$.
Note that $F_{\theta}(k) = \bar F_{1-\theta}(n-k)$.
Of course, if $ \theta < \theta'$, then $F_{\theta}(k) \ge F_{\theta'} (k)$ and $ \bar F_{\theta}(k) \le \bar F_{\theta'}(k)$.
There are essentially two cases:
\begin{itemize}

\item
Let $k \in \{0,\dots,n-1\}$.
If $\theta \in (\theta_{k},\theta_{k+1}]$, then $\theta \in \tau(x)$ iff $x \le k$,
and so $A_\theta = F_{\theta} (k) \ge F_{\theta_{k+1}}(k)$.
Symmetrically,
if $\theta \in [1-\theta_{n-k+1},1-\theta_{n-k})$, then $\theta \in \tau(x)$ iff $x \ge k$, 
and so $A_\theta = \bar F_{\theta} (k) = F_{1-\theta}(n-k) \ge F_{\theta_{n-k+1}}(n-k)$.

\item
If $\theta \in (\theta_{n},1-\theta_{n})$ then $\theta \in \tau(x)$ for all $x \in X$, and so $A_{\theta} = 1 \ge 1-\alpha$.

\end{itemize}

Thus, for $\tau$ to be a confidence interval at level $1-\alpha$, it suffices to have
\begin{itemize}
\item $F_{\theta_{k+1}}(k) \ge 1-\alpha$ for $k \in \{0,\dots,n-1\}$.
\item $F_{\theta_{n-k+1}}(n-k) \ge 1-\alpha$ for $k \in \{1,\dots,n\}$.
\end{itemize}

It is clear to see that these two sets of conditions are equivalent. The following result follows from monotonicity and continuity of the tail probabilities with respect to the parameter:
\begin{lemma}
For every $k \in \{0,1,\dots,n-1\}$, there is a unique solution  $\theta_{k+1}$ to the equation $F_{\theta_{k+1}}(k) = 1-\alpha$.
Moreover, $\theta_k < \theta_{k+1}$ and, for sufficiently small $\alpha$, we have $\theta_{n} < 1/2$.
\end{lemma}

For $k \in \{0,1,\dots,n-1\}$, let $\theta_{k+1}$ be the unique solution to the equation $F_{\theta_{k+1}}(k) = 1-\alpha$. 
It follows from the above lemma that $\tau$ is then a confidence set at level $1-\alpha$. 
We now proceed to identify a prior to establish credibility of the same set estimator.

Let $\pi_{\alpha} = \sum_{k=1}^{n} p_k (\delta_{\theta_k} + \delta_{1-\theta_k}) $.
To have the correct posterior coverage, we must verify that $\Post{k}{\pi_{\alpha}}(\tau(k)) =  1-\alpha$ for all $k \in \{0,\dots,n\}$.
It is easy to check that
\[
\Post{k}{\pi_{\alpha}}(\tau(k)) 
              &= \Post{k}{\pi_{\alpha}} (\{\theta_{k+1},\dots,\theta_{n}\} ), &&\text{ for $k=0$,}\\
\Post{k}{\pi_{\alpha}}(\tau(k)) 
              &= \Post{k}{\pi_{\alpha}} (\{\theta_{k+1},\dots,\theta_{n},1-\theta_{n},\dots,1-\theta_{n-k+1}\} ), &&\text{ for $k = 1, \dots,n-1$,}\\
\Post{k}{\pi_{\alpha}}(\tau(k)) 
              &= \Post{k}{\pi_{\alpha}} (\{1-\theta_{n},\dots,1-\theta_{n-k+1}\} ), &&\text{ for $k= n$.}
\]
Note that, by symmetry in the likelihood and prior, $\Post{k}{\pi_{\alpha}} = \Post{n-k}{\pi_{\alpha}} \circ r^{-1}$ for the reflection $r(\theta)=1-\theta$.
As a result, the condition for $k$ above is equivalent to the condition for $n-k$.

For $x =0,\dots,n$, let $Z(x) = \sum_{k=1}^{n} p_k \big( q(\theta_k,x) + q(1-\theta_k,x) \big)$ be the marginal probability of observing $x$ successes,
which satisfies $Z(x) = Z(n-x)$ on the basis of the symmetries mentioned above.
In terms of parameters $p_1,\dots,p_{n}$ defining our prior,
the $\floor{ n/2 }$ unique conditions above are equivalent to the equations
\[\label{redundanteqns}
\sum_{y=x+1}^{n} p_y q(\theta_y,x) + \sum_{y=n-x+1}^{n} p_y q(1-\theta_y,x)  = (1-\alpha) Z(x), \text{ for $x = 0, 1,\dots, \floor{ n/2} $.}
\]

In addition to the constraint that
$2\sum_{y=1}^{n} p_y = 1$ and $0 \le p_y \le 1$ for all $y$,
we have a linear system of $\floor{n/2}$ equations for $n$ unknowns,
given by
$ 
(A'+B'-(1-\alpha)C') [p_1,\dots,p_n]^T = 0$
where $A'$, $B'$, and $C'$ are the first $\floor{n/2}$ rows of the matrices
\[
A=
\begin{bmatrix}
q(\theta_1,0) & q(\theta_2,0) & \dotsm & q(\theta_n,0) \\
0 & q(\theta_2,1) & \dotsm & q(\theta_n,1) \\
0 & 0 & \ddots & \vdots \\
0 & 0 & 0 & q(\theta_n,n-1) \\
0 & 0 & 0 & 0 
\end{bmatrix},
\]
\[
B=
\begin{bmatrix}
0 & 0 & 0 & 0 \\
0 & 0 & 0 & q(1-\theta_n,1) \\
0 & 0 & q(1-\theta_{n-1},2) & q(1-\theta_n,2) \\
0 & \iddots & \dotsm & \vdots \\
0 & q(1-\theta_2,n-1) & \dotsm & q(1-\theta_n,n-1)  \\
q(1-\theta_1,n) & q(1-\theta_2,n) & \dotsm & q(1-\theta_n,n) 
\end{bmatrix},
\]
and $C_{x,k} = q(\theta_k,x) + q(1-\theta_k,x)$ for all $x = 0,\dots,n$ and $k=1,\dots,n$. (These matrices correspond to all $n+1$ equations of the form \cref{redundanteqns}.)

It is useful to consider special cases. For $n=1$ trial, we are back in the Bernoulli problem (\cref{ExBernoulliMatching} and \cref{SecAppExamplesBernoulli}). For $\alpha \in (0,0.5)$, the set estimator $\tau: X \to \BorelSets{\Theta}$ is given by $\tau(0) = [0,1-\alpha)$ and $\tau(1)=(\alpha,1]$, 
 exactly as in \cref{ExBernoulliMatching}. The system of linear equations that must be satisfied by a matching prior for $\tau$ has a unique solution that lies in the probability simplex. Not surprisingly, it is the uniform distribution on $\{\alpha,1-\alpha\}$ from \cref{ExBernoulliMatching}.

For $n=2$ trials, and $\alpha \in (0,0.2)$,
the set estimator
$\tau: X \to \BorelSets{\Theta}$ is given by $\tau(0) = [0,1-b)$, $\tau(1)=(a,1-a)$ and $\tau(2) = (b,1]$, where $a = 1 - \sqrt{1-\alpha}$ and $b = \sqrt{\alpha}$.
The system of linear equations yields a unique matching prior
 $\pi_\alpha = p \,\mathrm{Unif}(\{a,1-a\}) + (1-p) \mathrm{Unif}\,(\{b,1-b\})$,
 where
\[
p=
\frac{b^3}{2 b^3+a-a^2-a b}.
\]
It is straightforward to verify that $0 < p < 1$ for $\alpha$ in the above range.

For $n \ge 3$ trials, the matching prior is not unique.

\subsection{Gaussian Example}
\label{SecAppExamplesGaussian}

The following example is the classical highest-posterior-density credible region (see \cref{DefUsualHPD}) for Gaussian data.

\begin{example}[Gaussian model] \label{ExGaussianMatching}
Consider sample space $X = \mathbb{R}^{n}$, parameter space $\Theta = \mathbb{R}$, and model with conditional density 
$\cd(\theta,x) = \prod_{i=1}^{n} \frac{1}{\sqrt{2 \pi}} e^{-\frac{1}{2}(x_{i}-\theta)^{2}} $. Consider the usual $1-\alpha$ confidence set $\tau$,
given by $\tau(x_1,\dots,x_n) = \{ \theta \in \Reals : |\overline{x}-\theta| \leq z_{\alpha/2} / \sqrt{n} \}$, where $\overline{x}$ is the empirical mean.
Then the  
improper prior with density $\rho(\theta) \equiv 1$ is a matching prior at level $1-\alpha$ for $\tau$. %
\end{example}

Note that the density $\rho$ of the ``prior" in \cref{ExGaussianMatching} is not integrable, and so does not correspond to an element of $\Priors$, as we study in this work. 
Nonetheless, one can compute a posterior for such ``improper" priors, and in this case it is straightforward to check that the posterior is integrable.

\section
[Details Deferred from \cref*{SecFindingMP}]
{Details Deferred from \cref{SecFindingMP}}
\label{sec8app}

\subsection{Convergence}
\label{SecAppFindingMP}

Recall the Bernoulli model, credible region, and matching prior defined in \cref{SecAppExamples}. That is, we have a sample space $X = \{0,1\}$, parameter space $\Theta = [0,1]$ and conditional density $\cd(\theta,x) = \theta^{x} (1 - \theta)^{1-x}$. For $\alpha \in [0,0.5]$, define the measure $\pi_{\alpha} = \mathrm{Unif}(\{\alpha,1-\alpha\})$ and acceptance functions $(\theta,x) \mapsto \psi_{\alpha}(\theta,x)$, where $\psi_{\alpha}(\cdot,0) = \textbf{1}_{[0,1-\alpha)}$, $\psi_{\alpha}(\cdot,1) = \textbf{1}_{(\alpha,1]}$. 

Here, we consider small perturbations of these matching priors and set estimators. Consider parameters $0 < c < \frac{\alpha}{2} < 0.2$. We recall that the families of credible regions in \cref{DefSimpleCredInt}  and \cref{DefContFamilyCredibleMain} come with a collection of parameters: constants $\beta^{-1}, \rtau,\eta,\delta$ and smoothing function $h$ with support $[h_{\min}, h_{\max}] \subset (0, \infty)$. To slightly ease the notational burden, for the remainder of this section, we consider any fixed collection of such parameters with $0 < \beta^{-1}, \rtau,\eta,\delta, h_{\min}, h_{\max} < c$. We will not explicitly notate dependence of our functions on these parameters. Define $\pi_{\alpha,c} = \mathrm{Unif}([\alpha-c,\alpha+c] \cup [1-\alpha-c,1-\alpha+c])$, and let $\psi_{\alpha,c}$ denote the corresponding acceptance function from \cref{DefContFamilyCredibleMain}, with parameters as above.

\begin{theorem} \label{ThmNewBernEx2}
The acceptance function $\psi_{\alpha,c}$ represents a credible region with respect to the prior $\pi_{\alpha,c}$. Furthermore, there exists a constant $C = C(\alpha)$ such that, for $0 < c < C$, the set estimator is also a confidence set at level $1-\alpha' \geq 1- \alpha - (2 + \frac{1}{1-\alpha})c$.
\end{theorem}

\begin{remark}
As one can see in the proof of this theorem, the set estimators studied in \cref{ThmNewBernEx1} (whose acceptance functions are $\psi_{\alpha}$) are the limits (in, e.g., the $L^{2}([0,1])$ topology on functions from $[0,1]$ to $[0,1]$) of the set estimators (with acceptance functions $\psi_{\alpha,c}$) 
studied in this theorem as $c$ goes to 0.
\end{remark}

\begin{proof}
Fix $0 < c < \frac{\alpha}{2} < 0.2$ and let $I =[\alpha-c,\alpha+c] \cup[1-\alpha-c,1-\alpha+c]$ denote the support of the prior $\pi_{\alpha,c}$.

Consider the observation $x=0$. (The case $x=1$ is essentially the same by symmetry.)
By a straightforward calculation, when $x=0$, the posterior density is
\[
\dee\Post{0}{\pi_{\alpha}}(\theta)/\dee\theta = \frac{1-\theta}{2c},\quad \theta \in I.
\] 
A straightforward calculation shows that the posterior mean of $\Post{0}{\pi_{\alpha}}$ is
$
2 \alpha(1-\alpha).
$ 
Note that
\begin{align*}
\Post{0}{\pi_{\alpha}}([\alpha-c,\alpha+c]) &= \int_{\alpha-c}^{\alpha+c} \frac{1-\theta}{2c}\, \dee \theta \\
&= \frac{1}{2c}(2c - \frac{1}{2}((\alpha+c)^{2}-(\alpha-c)^{2})) \\
&=1-\alpha. 
\end{align*} 
In particular, the interval $[0,C]$ is a credible region at level $\alpha$ provided $\alpha + c \leq C \leq 1 - \alpha -c$. This demonstrates that these intervals have $(1-\alpha)$ credibility. They contain the usual credible balls and relaxed credible balls, since the slopped edges fall outside $I$. Note that these intervals are generally not confidence sets at level $\alpha$, so this calculation does not give us a matching prior.

Next, denote by $\phi_{\alpha',c}$ the acceptance function given in \cref{DefSimpleCredInt} with the same parameter $\beta^{-1} \in (0,c)$ as $\psi_{\alpha,c}$ but a possibly different level $\alpha' = \alpha - \epsilon$ for some $\epsilon > 0$. Before analyzing $\psi_{\alpha,c}$, we collect facts about the simpler $\phi_{\alpha',c}$:

Observe that, by construction, $\phi_{\alpha',c} \in \{0,1\}$ except for $\theta$ in a union of at most two intervals of length $\beta^{-1}$, over which the acceptance probability drops quickly. Since the interval $[\alpha-c,\alpha+c]$ is a credible ball at level $\alpha \geq \alpha'$ and the posterior mean when $x=0$ is 
$2 \alpha(1-\alpha)$, the set $\{\theta \, : \, \phi_{\alpha',c}(\theta,0) = 1 \}$ must be an interval containing 0 as long as $c < 1 - 5 \alpha + 4 \alpha^{2}$. Thus, for $c$ sufficiently small there is in fact only \textit{one} interval on which $\phi_{\alpha',c}(\cdot,0) \in (0,1)$; we restrict ourselves to such values of $c$ for the remainder of the proof. Furthermore, $\phi_{\alpha',c}(\theta,0) = 1$ for all $\theta < 1-\alpha-c-\beta^{-1}$.

Let $\hat{\psi}_{\alpha,c}^{x}$ denote the ``inflated" acceptance function given in \cref{tfrsetdefnMain}. 
We are now in a position to analyze the coverage of $\hat{\psi}_{\alpha,c}$: 
\begin{enumerate}
\item \textbf{$\theta \in [\alpha + c + \beta^{-1},1-\alpha-c - \beta^{-1}]$:} 
We note that $\hat{\psi}_{\alpha,c}(\theta,0) = \psi_{\alpha,c}(\theta,0)= 1$ 
in this case, so these points are always covered.
\item \textbf{$\theta \in [0,\alpha+c+\beta^{-1}]$:} We note that $\hat{\psi}_{\alpha,c}(\theta,0) = 1$ in this case, so the coverage is at least
\[ 
P_{\theta}\{0\} = 1 - \theta \geq 1 - \alpha - c - \beta^{-1}.
\] 
\item \textbf{$\theta \in [1-\alpha-c-\beta^{-1},1]$:} We obtain the same bound on the coverage as in the previous case by symmetry.
\end{enumerate}

Thus, for $c$ sufficiently small, the ``coverage'' of $\hat{\psi}_{\alpha,c}$ satisfies
\[ \label{IneqCoverageInflated}
\inf_{\theta \in [0,1]} \int_X \hat{\psi}_{\alpha,c}(\theta,x)  \, P_{\theta}(\dee x) 
\ge 1 - \alpha - c - \beta^{-1}.
\] 

We now return to $\psi_{\alpha,c}^{x}$.  Denote by $R_{\alpha,c}$ the function satisfying
\[ 
\psi_{\alpha,c}(\theta,x) = \max(0, \hat{\psi}_{\alpha,c}(\theta,x) -R_{\alpha,c}(x)).
\]
Note that this is the same object as denoted by $R$ in \cref{DefContFamilyCredibleMain}. 

Recall that $h$ has support $[h_{\min}, h_{\max}]$, and so the  $\Post{0}{\pi_{\alpha}}$-credibility of $\hat{\psi}_{\alpha,c}$ is at most $c \, h_{\max}$.
For any fixed $r > 0$ and level $\alpha' \leq \alpha$, the fact that $[0,1-\alpha-c]$ has posterior credibility at least $1-\alpha$, together with the fact that $\hat{\psi}_{\alpha,c}(\theta,0) = 1$ 
for $\theta \leq \alpha+c$, implies that $\max(0, \hat{\psi}_{\alpha,c}(\cdot,0) -r) = \hat{\psi}_{\alpha,c}(\cdot,0) -r$ for $\theta \in [0,\alpha+c]$. Thus,
\[
1-\alpha &= \Post{0}{\pi_{\alpha}}(\max(0, \hat{\psi}_{\alpha,c}(\cdot,0) -r)) 
\\&\leq \Post{0}{\pi_{\alpha}}(\hat{\psi}_{\alpha,c}(\cdot,0)) - r(1-\alpha) \leq 1-\alpha + c h_{\max} - r(1-\alpha).
\] 

Inspecting the outer terms, we see $R_{\alpha,c}(0)(1-\alpha) \leq c \, h_{\max}$, so $R_{\alpha,c}(0) \leq \frac{c \,h_{\max}}{1-\alpha}$. Combining this with \cref{IneqCoverageInflated}, the coverage satisfies the following lower bound:
\begin{align*} 
\inf_{\theta \in [0,1]} \int_X \psi_{\alpha,c}(\theta,x) \, P_{\theta}(\dee x) 
&\geq \inf_{\theta \in [0,1]} \int_X [\hat{\psi}_{\alpha,c}(\theta,x) - R_{\alpha,c}(x)] \, P_{\theta}(\dee x) \\
&\geq  1 - \alpha - c - \beta^{-1} - \frac{c \, h_{\max}}{1-\alpha} \\
&\geq 1 - \alpha - \Big(2 + \frac{1}{1-\alpha}\Big)c.
\end{align*}
\end{proof}

\newpage

\subsection{Approximately Matching Priors in the Binomial Problem} \label{SecAppBin}

In this section,
we investigate our heuristic algorithm in the binomial model defined in
\cref{SecAppExamplesBinomial}.
Like in the Bernoulli model, 
the heuristic seems to quickly identify approximately matching priors for the family of relaxed credible balls.
As the number of trials $n$ grows, and as we attempt to find sharper relaxed credible balls (i.e., with growing slope $\beta$), numerical issues require high-precision numerical libraries. 
Ultimately, we believe that new algorithmic ideas may be needed, perhaps combining iteration with other tactics.

In \cref{binomial:n2Beta5,binomial:n2Beta220}, we present the approximately matching priors and relaxed credible balls obtained for the binomial model with $n=2$ trials. (These and all other results are at level $\alpha = 0.05$, unless stated otherwise, and use a 1200-point discretization, with iterations starting from the uniform distribution.) 
In \cref{binomial:n2Beta5},  the slope of the relaxed credible balls is $\beta=5$.
After only one iteration starting from the uniform prior, the coverage is within 0.013 of the target level $\alpha=0.05$. (On average, we are off by only 0.002.)
In \cref{binomial:n2Beta220}, the slope is $\beta = 220$. 
Here, we observe convergence after $k=20$ iterations, obtaining coverage within 0.01.  %
(The particular choice for $\beta$ here was based on considerations of numerical stability and convergence.)
In our experiments, we consistently obtained more concentrated priors as the slope $\beta$ increased, with mass collapsing towards four point masses.
In particular, the approximately matching prior for $\beta=220$ has mass tightly concentrated around four points: $a,b,1-b,1-a$ where $a=1-\sqrt{1-\alpha} \approx 0.025$ and $b=\sqrt{\alpha}\approx0.22$. These locations agree with the hand-constructed exact matching prior for $n=2$ described in \cref{SecAppExamplesBinomial}. In fact, the amount of probability mass associated to each of these four point masses agrees with the exact matching prior up to two digits of precision.

\begin{figure}[t]
\centering %
\small Prior (CDF)\\
\includegraphics[width=.8\linewidth,trim=20pt 45pt 30pt 40pt,clip]{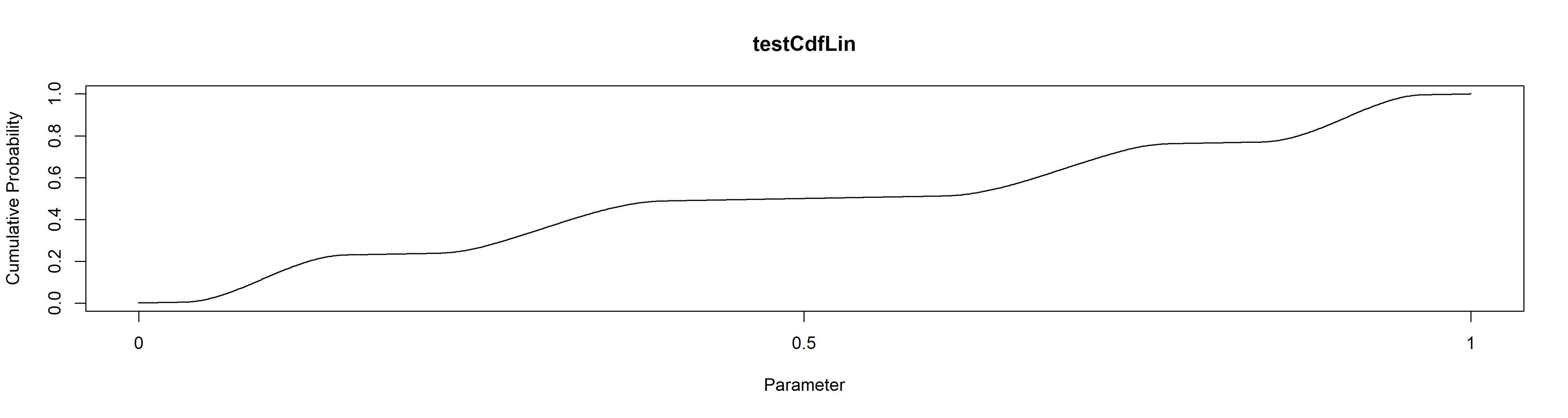}\\
Credible Region (rejection probability function) for $x=0,1,2$ successes\\
\includegraphics[width=.8\linewidth,trim=20pt 45pt 30pt 40pt,clip]{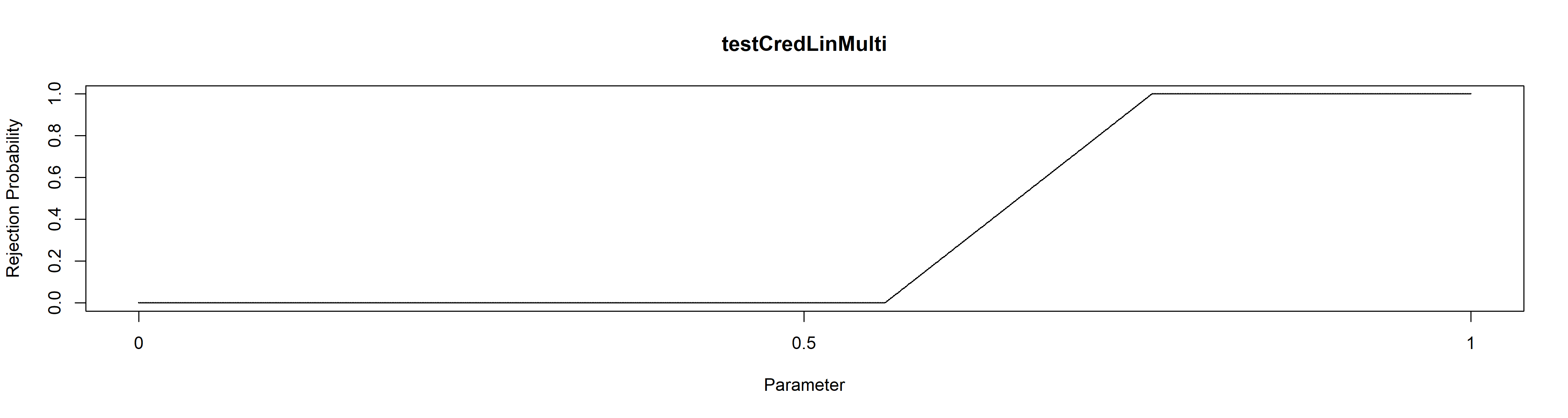}\\
\includegraphics[width=.8\linewidth,trim=20pt 45pt 30pt 40pt,clip]{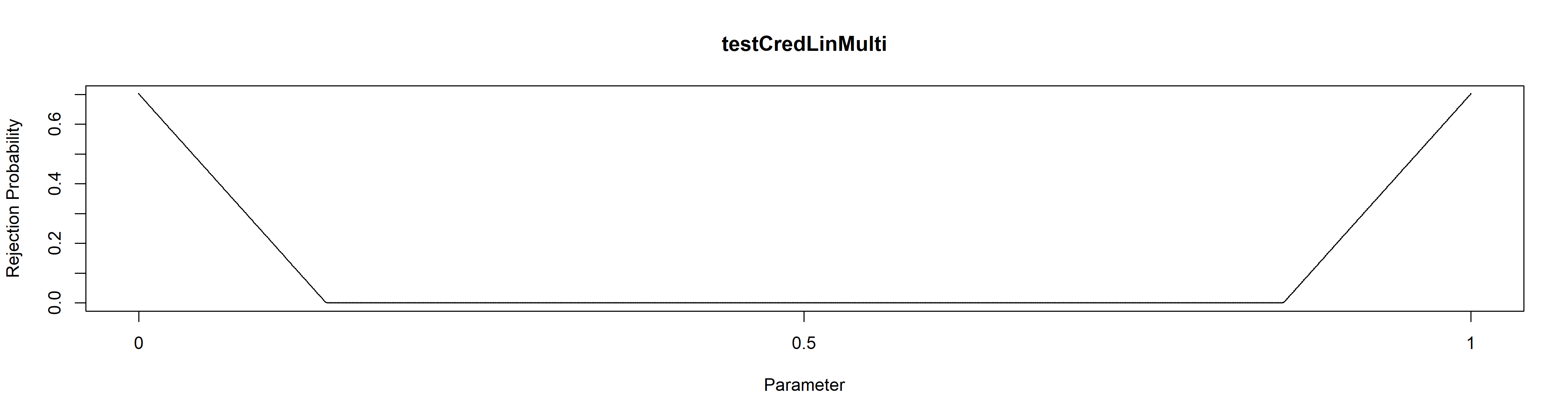}\\
\includegraphics[width=.8\linewidth,trim=20pt 45pt 30pt 40pt,clip]{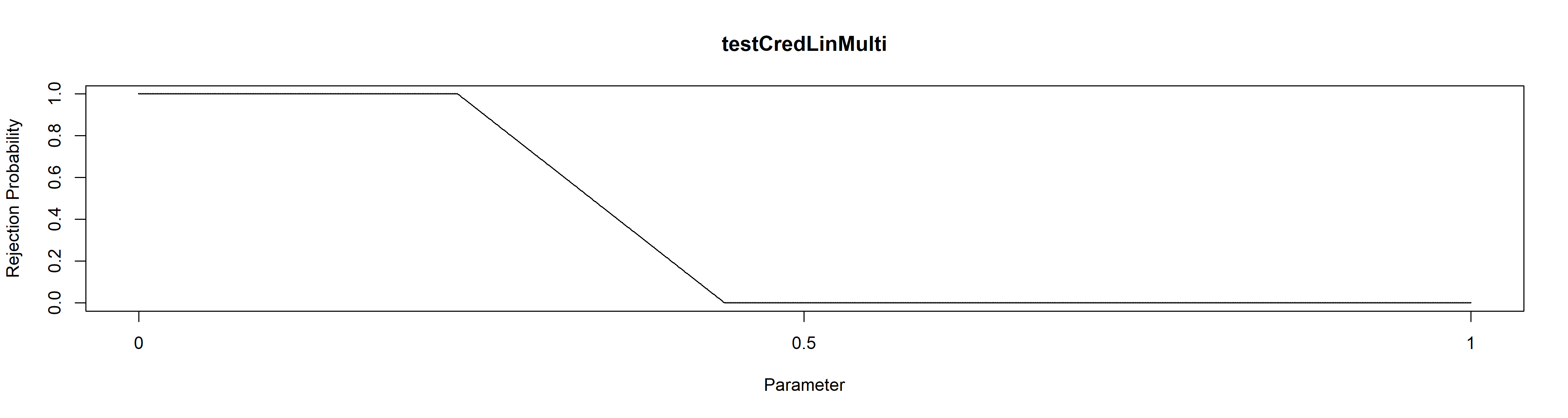}
\caption{Approximately matching prior and relaxed credible ball for $n=2$ trials, slope $\beta=5$, level $\alpha=0.05$, $k=2$ iterations.} 
\label{binomial:n2Beta5}
\end{figure}

\begin{figure}[t]
\centering %
\small 
Prior (CDF)\\
\includegraphics[width=.8\linewidth,trim=20pt 45pt 30pt 40pt,clip]{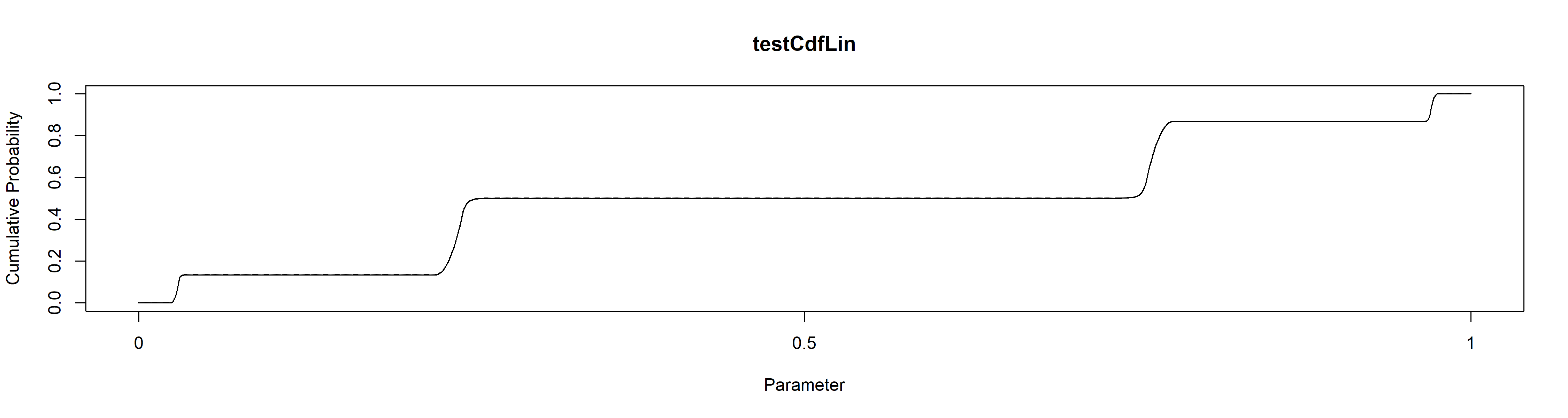}\\
Credible Region (rejection probability function) for $x=0,1,2$ successes\\
\includegraphics[width=.8\linewidth,trim=20pt 45pt 30pt 40pt,clip]{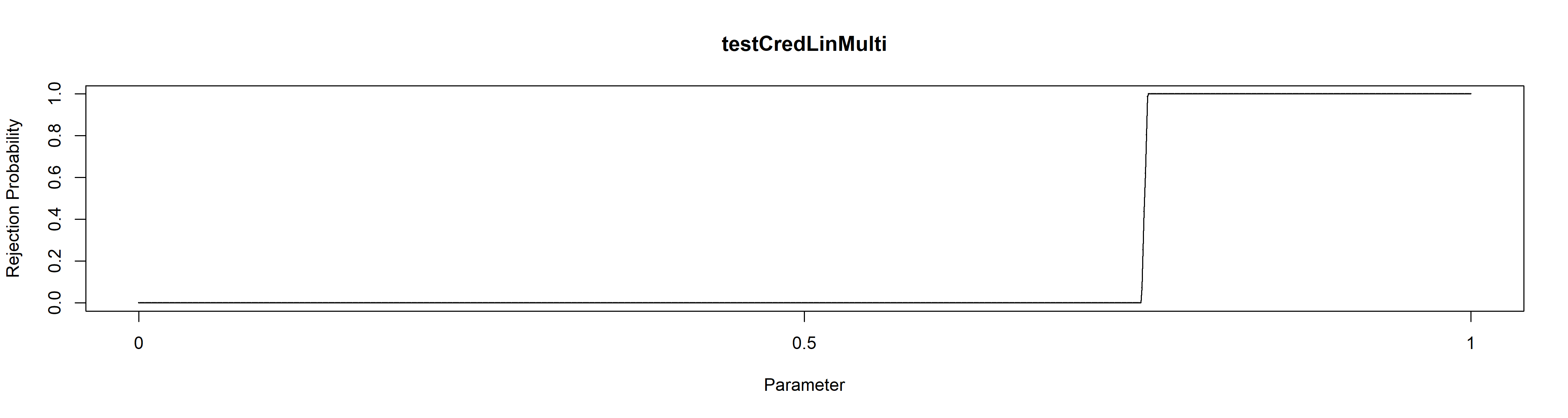}\\
\includegraphics[width=.8\linewidth,trim=20pt 45pt 30pt 40pt,clip]{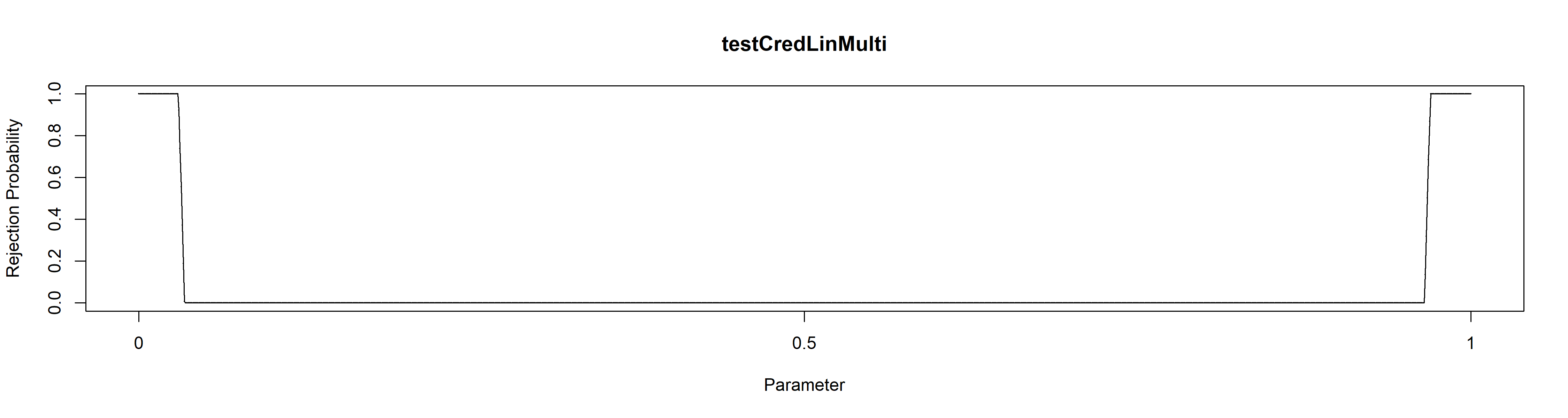}\\
\includegraphics[width=.8\linewidth,trim=20pt 45pt 30pt 40pt,clip]{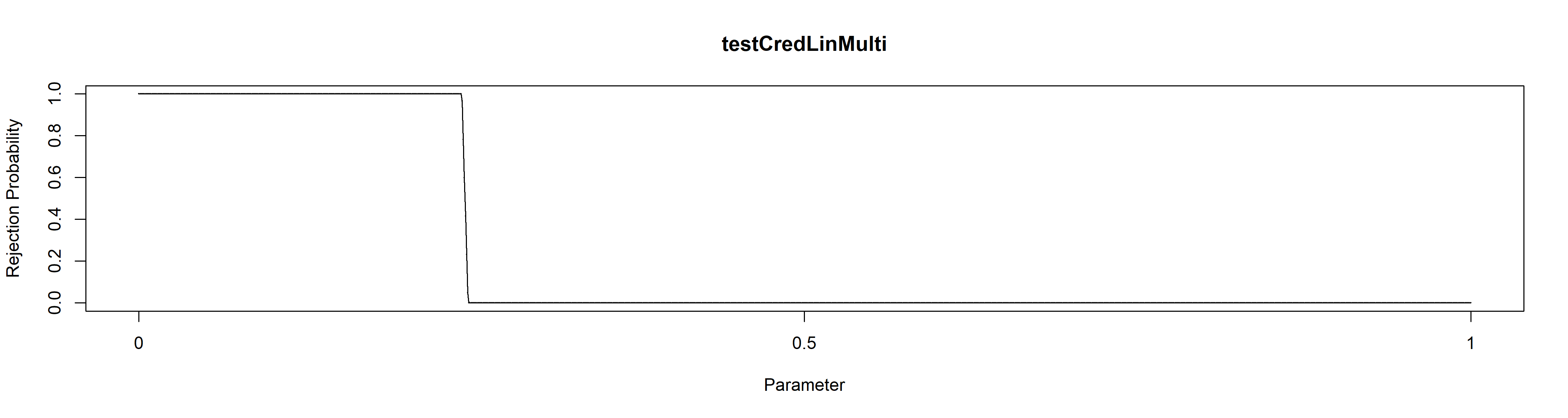}
\caption{Approximately matching prior and relaxed credible ball for $n=2$ trials, slope $\beta=220$, level $\alpha=0.05$, $k=20$ iterations.} 
\label{binomial:n2Beta220}
\end{figure}

If our hand-constructed matching priors were in fact the limiting fixed points of our heuristic algorithm as $\beta \to \infty$, we would expect the approximately matching priors for $n$ trials to concentrate around $2n$ point masses, with the location and masses described in \cref{SecAppExamplesBinomial}. In fact, for $n \ge 3$, the apparent limiting fixed point changes, although we cannot rule out the possibility that numerical issues or slow convergence explain the shift. In \cref{binomial:Beta220priors}, we present the matching priors obtained for $n=1,\dots,5$ trials, with $\beta=220$, $\alpha=0.05$, and $k=20$ iterations each. We obtain coverage within 0.01 or better in all cases except for $n=5$ where further iterations would be required. %
In \cref{binomial:credreg220}, we visualize the credible regions for $n=2,3,4$ trials, in terms of their rejection probability functions. For $n\ge 3$, the credible regions are larger than the ones we constructed by hand, potentially explaining why the associated approximately matching priors no longer resemble our hand-construct exact matching priors.

\begin{figure}[t]
\centering
\includegraphics[width=0.8\linewidth,trim=20pt 45pt 30pt 40pt,clip]{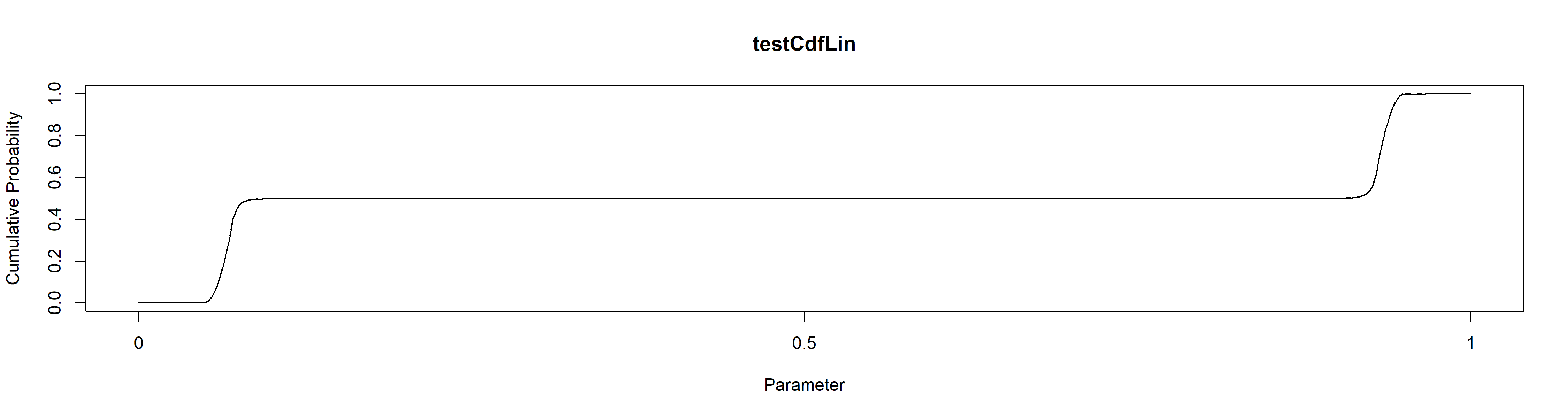}\\
\includegraphics[width=0.8\linewidth,trim=20pt 45pt 30pt 40pt,clip]{figures/binomial/n2Beta220/testCdfLin.png}\\
\includegraphics[width=0.8\linewidth,trim=20pt 45pt 30pt 40pt,clip]{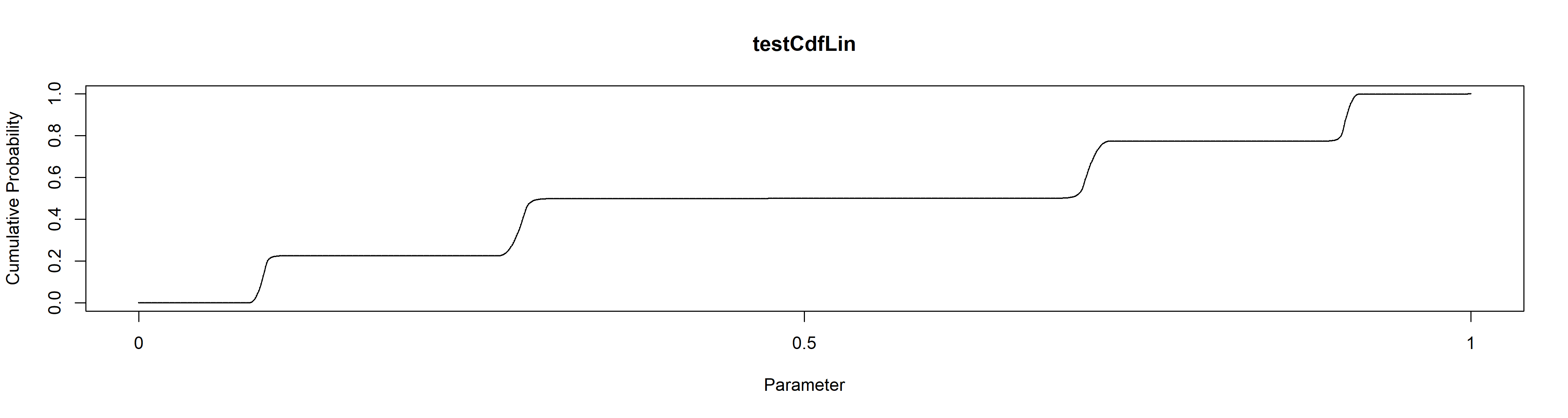}\\
\includegraphics[width=0.8\linewidth,trim=20pt 45pt 30pt 40pt,clip]{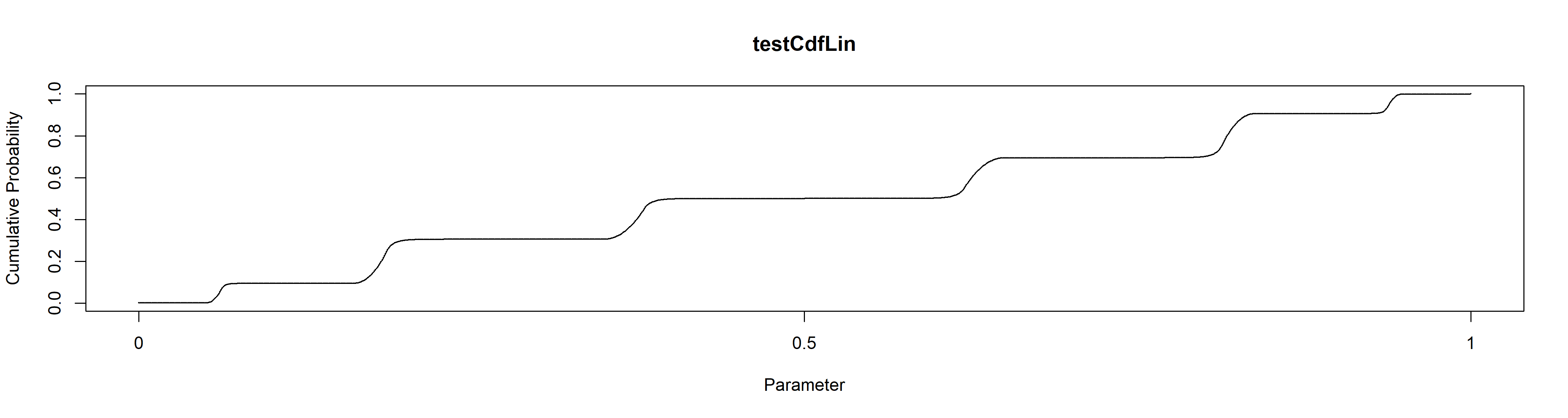}\\
\includegraphics[width=0.8\linewidth,trim=20pt 45pt 30pt 40pt,clip]{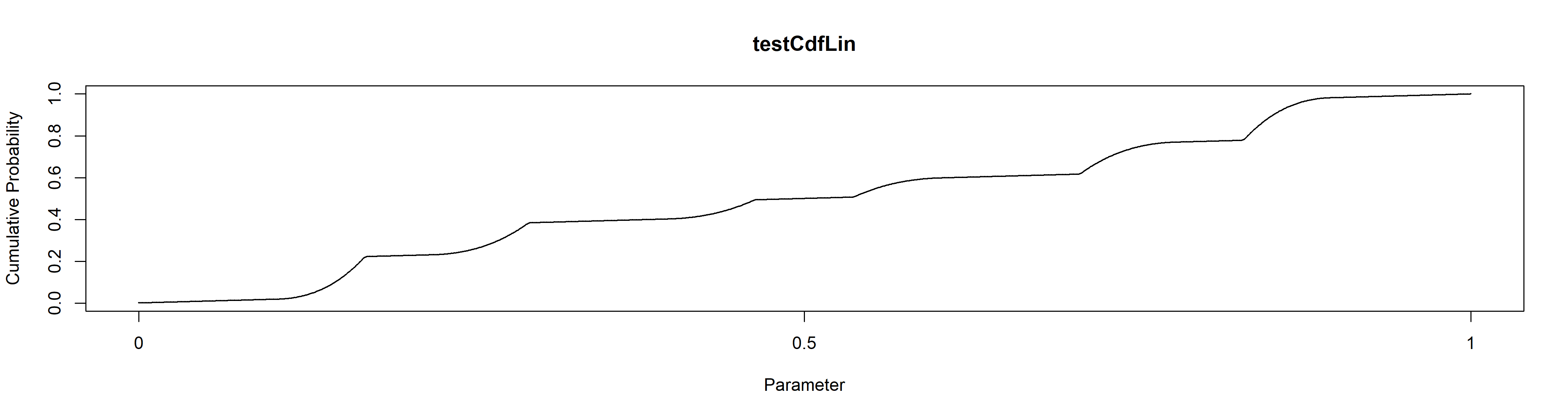}\\
\caption{Approximately matching priors for $n=1,2,3,4,5$ trials, level $\alpha=0.05$, slope $\beta=220$, and $k=20$ iterations.} 
\label{binomial:Beta220priors}
\end{figure}

\begin{figure}[t]
\centering
\small
$n=2$ trials\\
\scalebox{0.5}{
\begin{minipage}{1\linewidth}
\centering
\includegraphics[width=1\linewidth,trim=50pt 65pt 30pt 40pt,clip]{figures/binomial/n2Beta220/testCredLin0.png}\\
\includegraphics[width=1\linewidth,trim=50pt 65pt 30pt 40pt,clip]{figures/binomial/n2Beta220/testCredLin1.png}\\
\includegraphics[width=1\linewidth,trim=50pt 65pt 30pt 40pt,clip]{figures/binomial/n2Beta220/testCredLin2.png}
\end{minipage}%
}%
\\$n=3$ trials\\
\scalebox{0.5}{
\begin{minipage}{1\linewidth}
\centering
\includegraphics[width=1\linewidth,trim=50pt 65pt 30pt 40pt,clip]{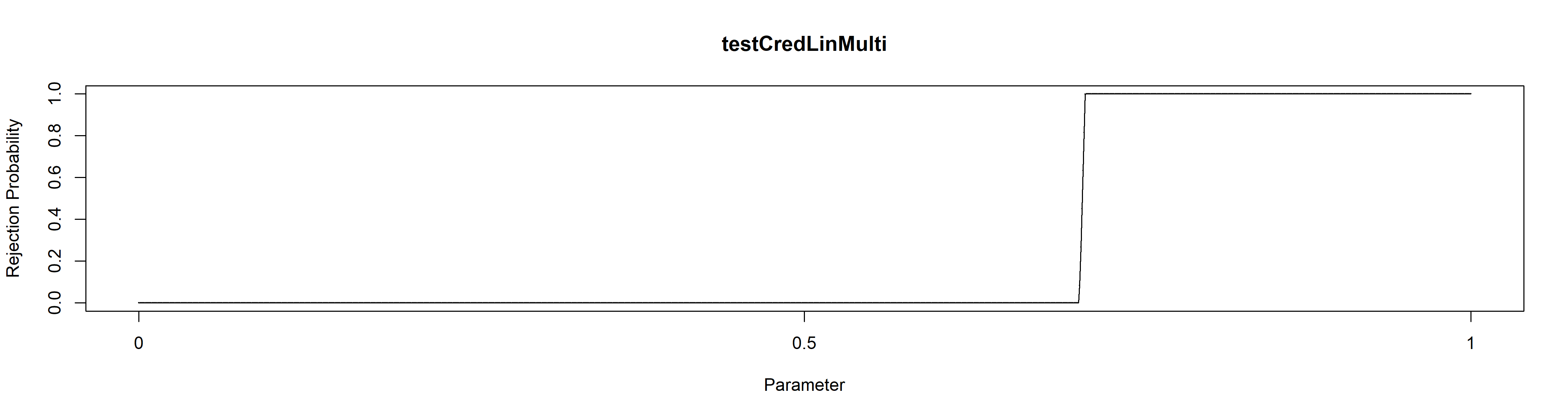}\\
\includegraphics[width=1\linewidth,trim=50pt 65pt 30pt 40pt,clip]{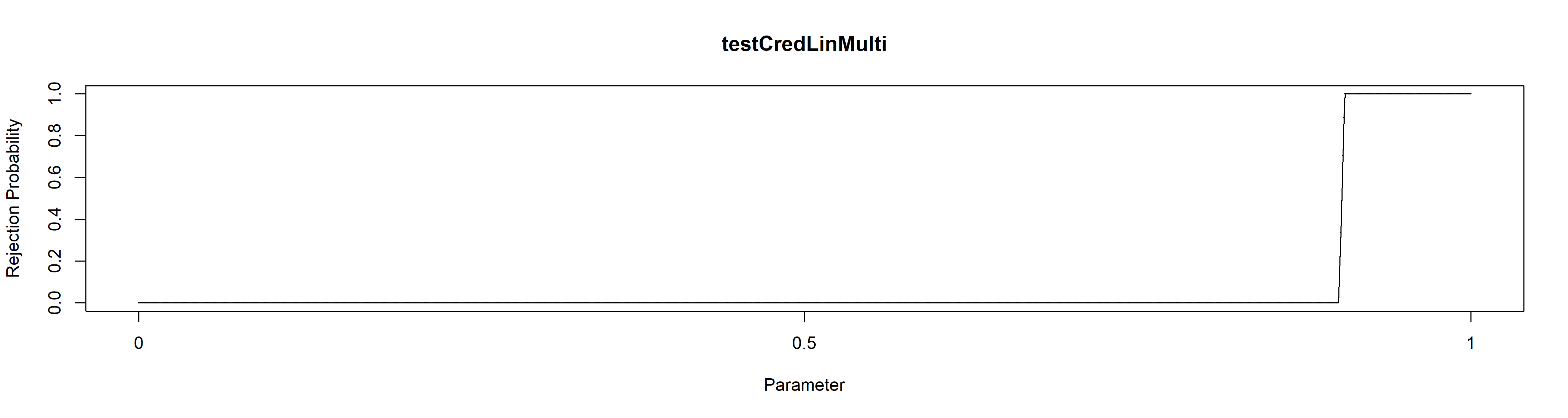}\\
\includegraphics[width=1\linewidth,trim=50pt 65pt 30pt 40pt,clip]{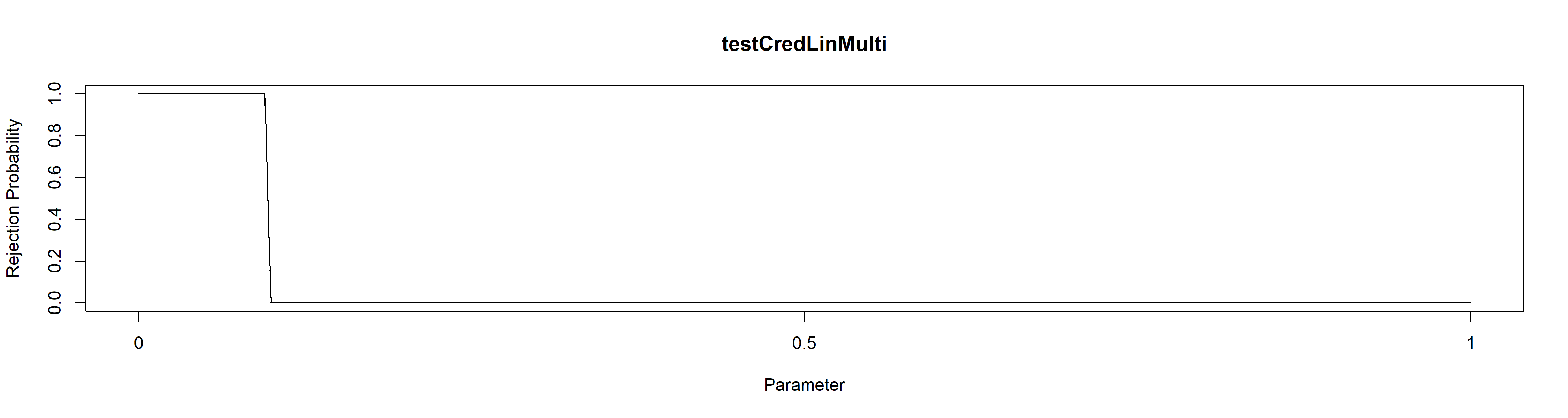}
\includegraphics[width=1\linewidth,trim=50pt 65pt 30pt 40pt,clip]{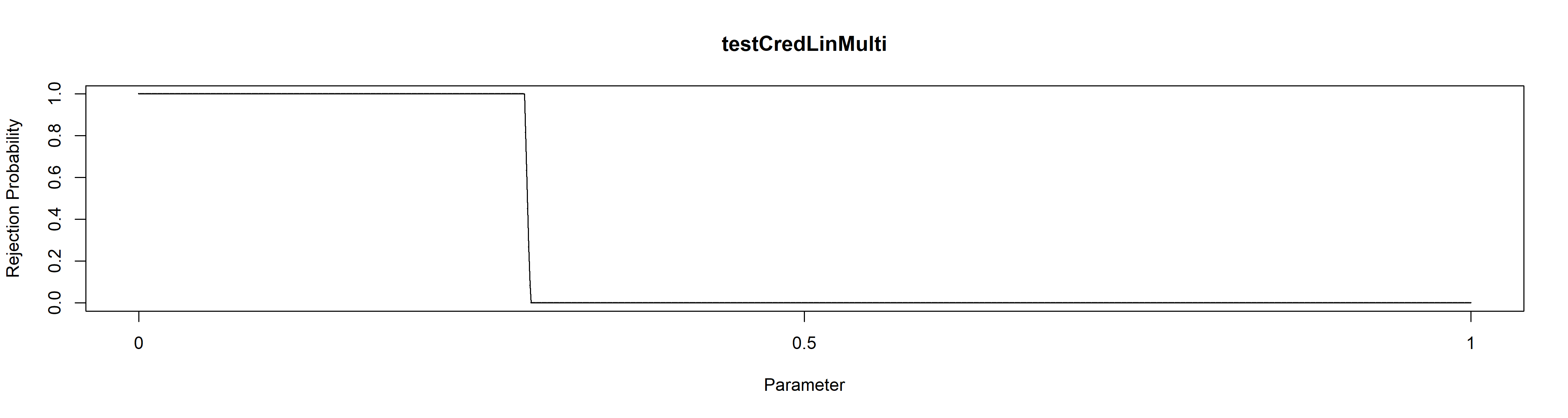}
\end{minipage}%
}%
\\$n=4$ trials\\
\scalebox{0.5}{
\begin{minipage}{1\linewidth}
\centering
\includegraphics[width=1\linewidth,trim=50pt 65pt 30pt 40pt,clip]{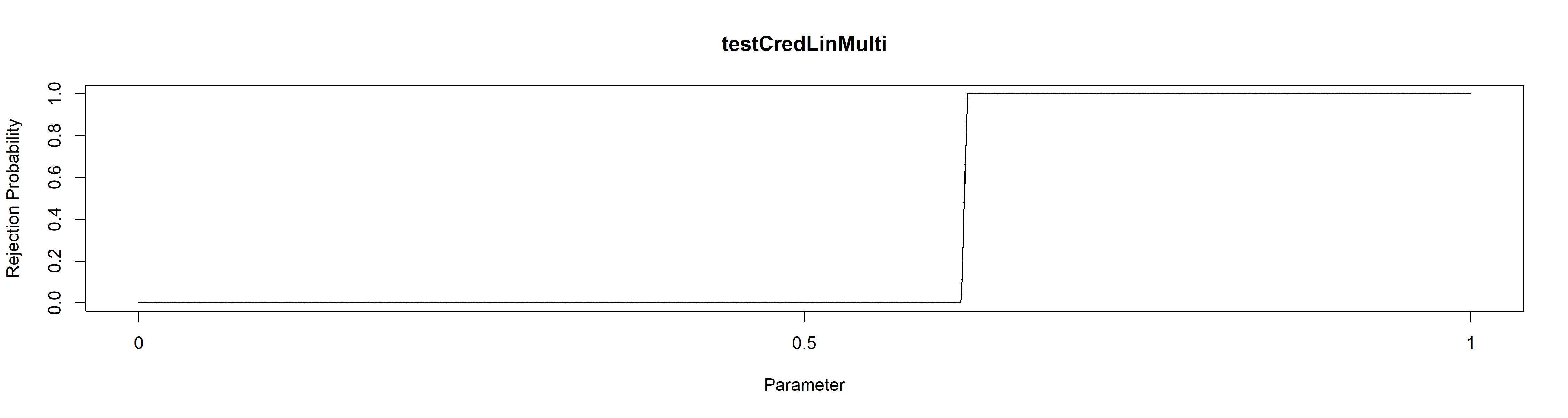}\\
\includegraphics[width=1\linewidth,trim=50pt 65pt 30pt 40pt,clip]{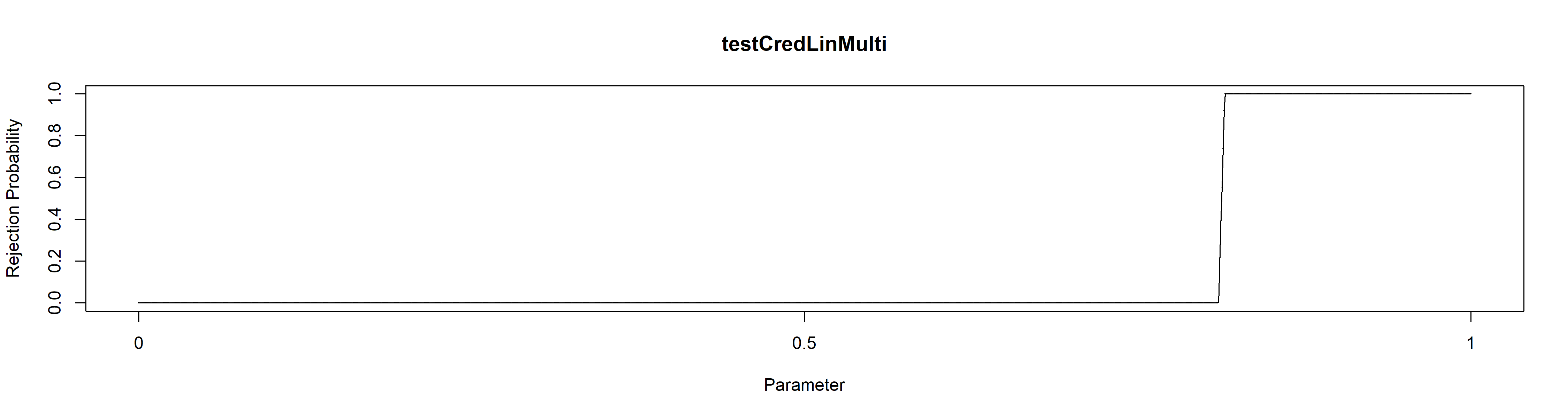}\\
\includegraphics[width=1\linewidth,trim=50pt 65pt 30pt 40pt,clip]{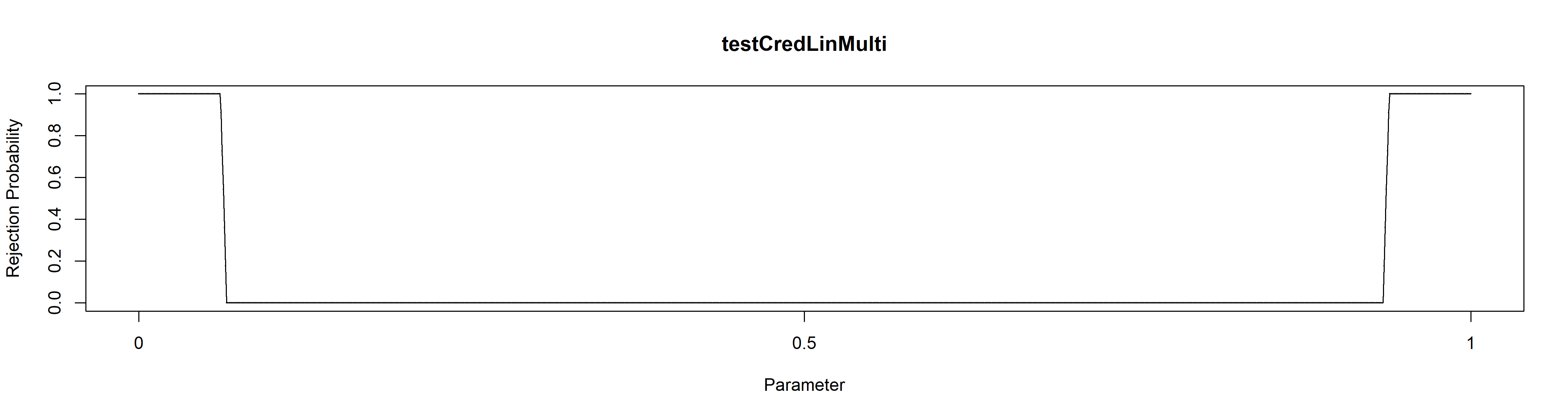}
\includegraphics[width=1\linewidth,trim=50pt 65pt 30pt 40pt,clip]{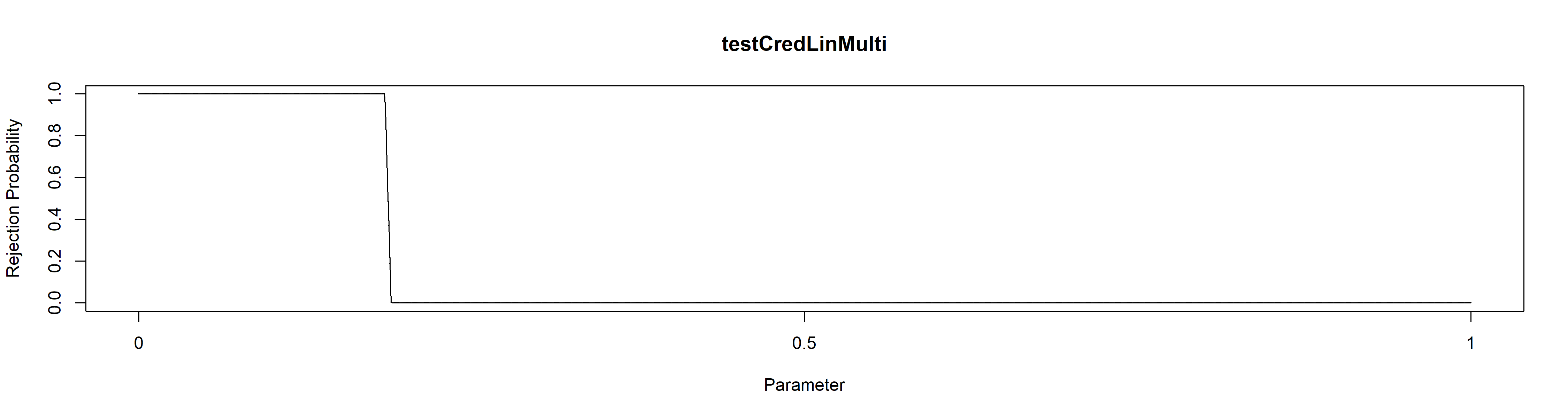}
\includegraphics[width=1\linewidth,trim=50pt 65pt 30pt 40pt,clip]{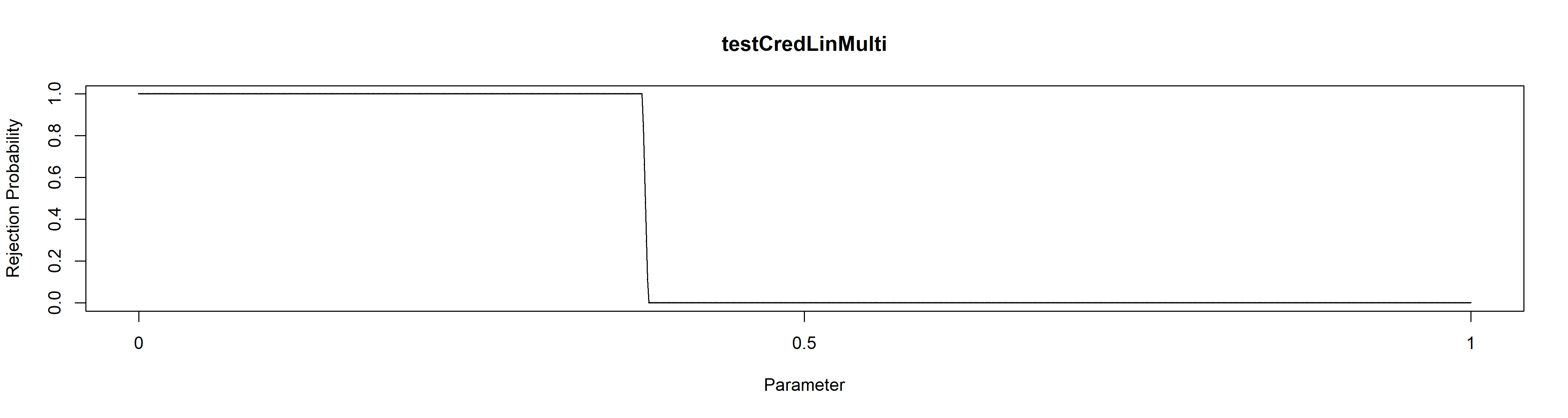}
\end{minipage}%
}%
\caption{Peturbed-and-relaxed credible balls for $n=2,3,4$ trials, level $\alpha=0.05$, slope $\beta=220$, and $k=20$ iterations. For $n$ trials, we plot the rejection probability function associated to $x=0,\dots,n$ successes.} 
\label{binomial:credreg220}
\end{figure}

In summary, our observations based on numerical experiments for the binomial model agree with our observations in \cref{SecFindingMP} for the Bernoulli model. Although we do not yet have theorems guaranteeing that our estimated priors are close to true matching priors, in practice we seem to find approximately matching priors quite quickly, provided the slope is not too extreme.

\vfill

\end{document}